\documentclass[11pt,a4paper]{article}
\usepackage{amsmath,amssymb,amsthm}
\usepackage{mathrsfs}
\usepackage{ascmac}
\usepackage{color}
\usepackage{comment}
\usepackage{ulem}
\usepackage{mathtools}
\mathtoolsset{showonlyrefs=true}

\usepackage{cite,amsrefs}
\usepackage{ulem}
\usepackage[top=30mm, bottom=30mm, left=25mm, right=25mm]{geometry}
\theoremstyle{definition}
 \newtheorem{dfn}{Definition}[section]
 \newtheorem{remark}[dfn]{Remark}
\theoremstyle{plain}
 \newtheorem{thm}{Theorem}[section]
 \newtheorem{prop}[thm]{Proposition}
 \newtheorem{lem}[thm]{Lemma}
 \newtheorem{cor}[thm]{Corollary}
\newtheorem{assumption}[thm]{Assumption}
\numberwithin{equation}{section}

\newcommand{\bu}{{\mathbf u}}
\newcommand{\bv}{{\mathbf v}}
\newcommand{\bw}{{\mathbf w}}
\newcommand{\bff}{{\mathbf f}}

\newcommand{\bG}{{\mathbf G}}
\newcommand{\bg}{{\mathbf g}}
\newcommand{\bh}{{\mathbf h}}
\newcommand{\bT}{{\mathbf T}}
\newcommand{\bU}{{\mathbf U}}
\newcommand{\BC}{\mathbb C}
\newcommand{\BM}{\mathbb M}
\newcommand{\BN}{\mathbb N}
\newcommand{\BR}{\mathbb R}
\newcommand{\BZ}{\mathbb Z}
\newcommand{\CA}{{\mathcal A}}
\newcommand{\CD}{{\mathcal D}}
\newcommand{\CF}{{\mathcal F}}
\newcommand{\CL}{{\mathcal L}}
\newcommand{\CM}{{\mathcal M}}
\newcommand{\CP}{{\mathcal P}}
\newcommand{\CR}{{\mathcal R}}
\newcommand{\CH}{{\mathcal H}}
\newcommand{\CS}{{\mathcal S}}
\newcommand{\CT}{{\mathcal T}}

\DeclareMathOperator{\dv}{\mathrm{div}}

\DeclareMathOperator{\supp}{\mathrm{supp}}

\newcommand{\pd}{\partial}
\newcommand{\HS}{\BR^N_+}

\title{Maximal $L_1$ regularity for the linearized compressible Navier-Stokes equations }
\author{Jou-Chun Kuo
\thanks{Graduate School of Fundamental Science and Engineering, \enskip Waseda University, 
3-4-5 Ohkubo Shinjuku-ku, Tokyo, 169-8555, Japan.
\endgraf 	
e-mail address: kuojouchun@asagi.waseda.jp
\endgraf
\endgraf
\endgraf
This work was partially supported by JST SPRING, Grant Number JPMJSP2128.
}}
\date{}
\begin{document}
\maketitle

\begin{abstract}
In this paper, we consider the linearized compressible Navier-Stokes equations with non-slip boundary conditions in the 
half space $\HS$. We prove the generation of a continuous analytic semigroup associated with this compressible Stokes 
system with non-slip boundary conditions in the half space $\HS$ and its $L_1$ in time maximal regularity.  We choose
the Besov space $\CH^s_{q,r} = B^{s+1}_{q,r}(\HS)\times B^s_{q,r}(\HS)^N$ as an underlying space, 
where  $1 < q < \infty$, $1\leq r < \infty$,  and $-1+1/q < s < 1/q$. We prove the generation of a continuous analytic
semigroup $\{T(t)\}_{t\geq 0}$ on $\CH^s_{q,r}$, and show that its generator admits maximal $L_1$ regularity. Our approach is to prove the existence of the resolvent in 
$\CH^s_{q,1}$ and some new estimates for the resolvent by using 
$B^{s+1}_{q,1}(\HS) \times B^{s\pm\sigma}_{q,1}(\HS)$
norms for some small $\sigma > 0$ satisfying the condition $-1+1/q < s-\sigma < s < s+\sigma < 1/q$.
\end{abstract}
\section{Introduction}
Let $\HS:=\{x=(x',x_N)\in \BR^N \mid x'\in \BR^{N-1}, \, x_N>0\}$, $N\ge2$,
be the half space. In this paper, we consider the following linear system:
\begin{equation}
\label{Eq:Linear}
\left\{\begin{aligned}
\pd_t \rho + \gamma\dv\bu & = 0 &\quad &\text{in $\HS\times(0, \infty)$}, \\
\pd_t\bu - \alpha\Delta\bu - \beta\nabla\dv\bu + \gamma\nabla\rho
& = 0 &\quad &\text{in $\HS\times(0, \infty)$},\\
\bu & = 0 &\quad &\text{on $\pd\HS\times(0, \infty)$},\\
(\rho,\bu)(0,x) & = (\rho_0,\bu_0) &\quad &\text{in $\HS$}.
\end{aligned}\right.
\end{equation} 
Here, $\rho$ and $\bu=(u_1,\cdots,u_N)$ are unknown functions, 
while the initial datum $(\rho_0,\bu_0)$ is assumed to be given.
Moreover, the coefficients $\alpha$, $\beta$, and $\gamma$ are assumed to be 
constants such that $\alpha > 0$, $\alpha+\beta>0$ and $\gamma>0$. 
 The aim of this paper  is to show  the generation of a continuous analytic semigroup 
associated with equations \eqref{Eq:Linear} and its $L_1$ in time maximal regularity property
 in some  Besov spaces. \par
The system \eqref{Eq:Linear} is the linearized system of the compressible Navier-Stokes equations with homogeneous Dirichlet boundary conditions:
\begin{equation}
\label{Eq:Nonlinear}
\left\{\begin{aligned}
\pd_t\varrho + \dv(\varrho\bv)&= 0&\quad&\text{in $\HS\times(0,  \infty)$}, \\
\varrho(\pd_t\bv+(\bv \cdot \nabla)\bv)-\mu\Delta\bv-(\mu+\nu)\nabla\dv \bv+\nabla P(\varrho)&=0&\quad&\text{in $\HS\times(0,  \infty)$},\\
\bv&=0&\quad&\text{on $\pd\HS\times(0,  \infty)$},\\
(\varrho,\bv)(0,x)&=(\varrho_0,\bv_0)&\quad&\text{in $\HS$},
\end{aligned}\right.
\end{equation} 
where $\varrho$ and $\bv$ describe the unknown density and 
the velocity field of the compressible viscous field, respectively, while
the initial datum $(\varrho_0,\bv_0)$ is a pair of given functions.
The coefficients $\mu$ and $\nu$ are assumed to satisfy the ellipticity conditions $\mu>0$ and $\mu+\nu>0$. 
In addition, the pressure of the fluid $P$ is a given smooth function with respect to $\varrho$, 
which is assumed to satisfy the stability condition $P'(\rho_*)>0$. 
Here, $\rho_*$ stands for the reference density that is a positive constant,
and the initial density $\varrho_0$ is given as a perturbation from $\rho_*$. 
As discussed in \cite[Sect. 8]{ES13}, the coefficients $\alpha$, $\beta$, and $\gamma$ are 
defined by $\alpha=\mu/\rho_*$, $\beta=\nu/\rho_*$, and $\gamma=\sqrt{P'(\rho_*)}$, respectively. 
Clearly, $\alpha$, $\beta$, and $\gamma$ satisfy the aforementioned given conditions. \par
There are a lot of results concerning the compressible Navier-Stokes equations \eqref{Eq:Nonlinear}.
Let us briefly summarize the results. 
Mathematical studies on the compressible Navier-Stokes equations started with 
the uniqueness results in a bounded domain by Graffi \cite{G53}, whose result is extended by Serrin \cite{S59}
in the sense that there is no assumption on the equation of state of the fluid. 
In the studies \cite{G53} and \cite{S59}, the fluid occupies a bounded domain surrounded by a smooth boundary.
A local in time existence theorem in H\"older continuous spaces was first proved by
Nash \cite{N62} and Itaya~\cite{I71,I76a}, independently, for the whole space case.
As for the boundary value problem case, Tani~\cite{T77} proved a local in time existence theorem 
in a similar setting provided that a (bounded or unbounded) domain $\Omega$ has a smooth boundary.
In Sobolev-Slobodetskii spaces, the local existence was shown by Solonnikov \cite{Sol},
see also the work due to Danchin \cite{D10} for an improvement of Solonnikov's result.
Matsumura and Nishida \cite{MN80} made a breakthrough in proving a unique global-in-time solution 
for the initial value problem of the compressible Navier-Stokes equations 
for the multidimensional case.
More precisely, Matsumura and Nishida~\cite{MN80} investigated the system with 
heat-conductive effects in $\BR^3$ and proved the global existence theorem with the aid of
a local existence theorem together with a priori estimates for the solution.
In particular, the a priori estimates were established by a combination of 
the linear spectral theory and the $L_2$-energy method.
They also succeeded to prove the global existence result in the half space and 
exterior domains cases with sufficiently small given data, see \cite{MN83}.
We here mention that the rate of convergence (as $t\to\infty$) of the solution 
to the system, which is constructed in~\cite{MN83}, is established in \cite{KS99} for the half space case
and \cite{KK05} for the exterior domains cases provided that the initial data are close to 
the constant equilibrium state. We also refer to a recent work due to Shibata and Enomoto \cite{SE18}
as well as Shibata \cite{S22} for some refinement of~\cite{KS99} in the sense that
the class of initial data $(\varrho_0-\rho_*,\bv)$ may be weakened. Notice that the approach of Shibata and Enomoto \cite{SE18} and Shibata \cite{S22} are completely different from
Kawashita's argument \cite{Ka02}, where Kawashita \cite{Ka02} also required less regularity on the initial conditions, in contrast to \cite{MN80}. \par
In the aforementioned works, the proof of the global existence theorem was mainly based on
the $L_2$-energy method (excluded the contribution due to Shibata \cite{S22}), 
but another approach was established by Ströhmer \cite{St90}.
His idea was to rewrite the system in Lagrangian coordinates, which is often said to be Lagrangian transformation.
Thanks to this reformulation, the convection term in the density equation, namely $\varrho\cdot\nabla\bv$, 
may be dropped off, so that the transformed system becomes the evolution equation of parabolic type,
and he used the semigroup theory. On the basis of a different approach,  
Mucha and Zaj\polhk aczkowski \cite{MZ04} applied $L_p$-energy estimates to show 
the global existence theorem in the $L_p$ in time and $L_q$ in space framework. 
Recently, maximal $L_p$ regularity approach was developed by Enomoto and Shibata \cite{ES13}, 
which extended the result of Mucha and Zaj\polhk aczkowski \cite{MZ04} in the sense that
it was allowed to construct global strong solution in the $L_p$ in time and $L_q$ in space framework.
We emphasize that, on page 418 in \cite{MZ04}, it was declared that there is no possibility to obtain a global existence theorem in the $L_p$-framework whenever we investigate the system 
in Eulerian coordinates, but this was wrong if the domain $\Omega$ is a bounded smooth domain. 
In fact, Kotschote \cite{K14} constructed global strong $L_p$-solutions in Eulerian coordinates, 
without making use of transformation to Lagrangian coordinates.
For a list of relevant references of studies of the local or global existence theorem
 (for classical or strong solutions), the readers may consult \cite[Section 2]{SE18} and references therein. \par
Recall that the Jacobian of Lagrange transformation is given by $I + \int_0^t\nabla\bu(\tau,\xi)d\tau$,
where $\bu(\tau,\xi)$ stands for the velocity field of a fluid particle at time $t$
which was located in $\xi$ at initial time $t=0$. Hence, to obtain the global existence theorem 
with the aid of Lagrangian transformation, it is always crucial to get a control of
$\int_0^t \nabla\bu(\tau,\xi)d\tau$ in a suitable norm. In particular, it is necessary to 
find a small constant $c>0$ such that
\begin{equation}
\label{bound-nabla-u}
\bigg\lVert\int_0^t\nabla\bu(\tau,\xi)d\tau\bigg\rVert_{L_\infty}\le c,
\end{equation}
which ensures that Lagrangian transformation is invertible. If the estimate \eqref{bound-nabla-u} 
is stemmed from an $L_p$ in time estimate for $\bu$ with $1 < p < \infty$, then 
we may only expect to have a $t$-dependent bound, if $\Omega$ is unbounded.
Of course, as we mentioned before, it is still possible to prove the global existence theorem
even if the constant $c$ appearing in \eqref{bound-nabla-u} depends on $t$, 
but the proof becomes more involved. \par
Recently, Danchin and Tolksdorf \cite{DT22} 
proved maximal $L_1$ regularity estimate for $\bu$, which implies that one may find
a $t$-independent constant $c$ such that \eqref{bound-nabla-u} is valid.
Here, they studied the system in the $L_1$ in time and $B^s_{p,1}$ in space framework,
where $p$ and $s$ are taken such that $1<p<\infty$ and $s=-1+N/p$. Their function space
is similar to the spaces used in \cite{CD10,D00}, but it was not necessary to consider
homogeneous Besov spaces in~\cite{DT22}, since it is well-known that  homogeneous Besov spaces 
$\dot B^s_{p,1}(\Omega)$ coincide with inhomogeneous Besov spaces $B^s_{p,1}(\Omega)$ 
if $-1+1/p<s<1/p$ and if the domain $\Omega$ is bounded of class $C^1$, see \cite[Remark 2.2.1]{DM15}.
The essential assumption in \cite{DT22} is that the fluid domain is bounded, which is required to prove
their extension version of Da Prato-Grisvard theory \cite{DG75}. 
 \par
We want to consider the viscous compressible 
fluid flow in general domains, which is described in \eqref{Eq:Nonlinear} 
when the fluid domain is $\BR^N_+$. 
 The $L_p$-$L_q$, $1 < p, q < \infty$, 
 maximal regularity theorem for \eqref{Eq:Nonlinear} 
was constructed in the paper due to Enomoto-Shibata \cite{ES13}, but following this paper, 
we want to construct the maximal $L_1$-$B^s_{p,1}$ regularity theory for equations
\eqref{Eq:Nonlinear}, where $1 < p < \infty$ and $-1 +1/p < s < 1/p$ for  the Stokes equations 
 and $1 < p < \infty$ and $-1+N/p \leq s < 1/p$ for the Navier-Stokes equations. \par
 \par
We want to study the viscous barotropic compressible fluid flow in an unbounded domain
in an $L_1$ in time maximal regularity framework.  Because, $L_1$ in time maximal regularity is 
the best framework to use  the Lagrange
transformation to solve the nonlinear problem. As a first step, in this paper 
we establish the maximal $L_1$-$B^s_{p,1}(\BR^N_+)$ regularity theorem
in the half space $\BR^N_+$ with $1 < p < \infty$ and $-1+1/p < s < 1/p$ 
for equations \eqref{Eq:Linear}, which is the model problem.  
The local well-posedness of the
nonlinear problem \eqref{Eq:Nonlinear} is treated in another paper \cite{KSprep}.
Although there are several contributions toward 
this topic \cite{St90,ES13,K14,DT22}, we intend to study the problem in the half space
within inhomogeneous Besov spaces setting. \par
 Before stating our main results, 
we introduce basic spaces in our paper as follows. 
Let $1 < q < \infty$, $-1+1/q < s < 1/q$, $1 \leq r < \infty$, $\mu \in \BR$, and $\Omega \in \{\BR^N, \HS\}$. 
Let $B^\mu_{q,r}(\Omega)$ denote standard
Besov spaces on $\Omega$. Let  
\begin{equation}\label{fundspace:1}
\begin{aligned}
\CH^s_{q, r}(\Omega) & = B^{s+1}_{q, r}(\Omega) \times B^s_{q,r}(\Omega)^N, \\
\CD^s_{q,r}(\BR^N)&= B^{s+1}_{q,r}(\BR^N)\times B^{s+2}_{q,r}(\BR^N)^N, \\
\CD^s_{q,r}(\HS)&= \{(\rho, \bv) \in B^{s+1}_{q,r}(\HS) \times B^{s+2}_{q,r}(\HS)^N \mid \bv|_{\pd\HS}=0\}, \\
\|(f, \bg)&\|_{\CH^s_{q,r}(\Omega)} = \|f\|_{B^{s+1}_{q,r}(\Omega )} + \|\bg\|_{B^s_{q,r}(\Omega)}, \\
\|(f, \bg)&\|_{\CD^s_{q,r}(\Omega)} = \|f\|_{B^{s+1}_{q,r}(\Omega)} + \|\bg\|_{B^{s+2}_{q,r}(\Omega)}.
\end{aligned}\end{equation} 
In addition, we introduce the operator $\CA^s_{q,r}$ corresponding to equations \eqref{Eq:Linear} 
which is defined by setting 
\begin{equation}\label{Def:A} \CA^s_{q,r}(\rho, \bv)  
 = (\gamma \dv \bv,  -\alpha\Delta \bv - \beta\nabla\dv\bv + \gamma\nabla\rho)
\quad\text{for $(\rho, \bv) \in \CD^s_{q,r}(\HS)$}.
\end{equation}
Using $\CA^s_{q,r}$, equations \eqref{Eq:Linear} are written as
\begin{equation}\label{semigroup:1}
\pd_t (\rho, \bu) + \CA^s_{q,r} (\rho, \bu) = (0, 0) \quad\text{for $t > 0$}, 
\quad (\rho, \bu)|_{t=0} = (\rho_0, \bu_0) \in \CH^s_{q,r}(\HS)
\end{equation}
for $ (\rho, \bu)$ with 
$$(\rho, \bu) \in C^0[(0, \infty), \CH^s_{q,r}(\HS)  \cap C^1((0, \infty), \CH^s_{q,r}(\HS)) \cap C^0((0, \infty), 
\CD^s_{q,r}(\HS){\color{red}{]}}.$$
\par
Our main results of this paper read as follows. 
\begin{thm}
\label{thm:0}
Let $1 < q < \infty$, $-1+1/q < s < 1/q$, and $1 \leq r < \infty$.  
Then, the operator $\CA^s_{q,r}$ generates a continuous analytic
 semigroup $\{T(t)\}_{t\geq0}$ on $\CH^s_{q,r}(\HS)$. \par
Moreover, 
	there exists a large $\omega_0\geq1$ such that,  
	for any $\omega\geq\omega_0$ and $(\rho_0,\bu_0) \in \CH^s_{q,1}(\HS)$, 
	\begin{align*}
		&\int_{0}^{\infty}e^{-\omega t}(\|\pd_tT(t)(\rho_0,\bu_0)\|_{\CH^s_{q,1}(\HS)}+\|T(t)(\rho_0,\bu_0)\|_{\CD^s_{q,1}(\HS)})dt
\leq C \|(\rho_0,\bu_0)\|_{\CH^s_{q,1}(\HS)}.
	\end{align*}
\end{thm}

To prove Theorem \ref{thm:0},  we consider 
the following resolvent problem:
\begin{equation}
\label{Eq:Resolvent}
\left\{\begin{aligned}
\lambda\rho + \gamma\dv\bu&=f &\quad &\text{in $\HS$}, \\
\lambda\bu - \alpha\Delta\bu - \beta\nabla\dv\bu + \gamma\nabla\rho
& = \bg &\quad &\text{in $\HS$},\\
\bu & = 0 &\quad &\text{on $\pd\HS$},
\end{aligned}\right.	
\end{equation} 
for $\lambda \in \Lambda_{\epsilon, \nu_0}$.  Here, $\Lambda_{\epsilon, \nu_0}$ 
is a subset of $\BC$ defined as follows:
\begin{equation}\label{sector:1}\begin{aligned}
\Sigma_{\epsilon} &= \{\lambda \in \BC \setminus \{0\} \mid 
|\arg \lambda| \leq   \pi-\epsilon,\}, \\
K_\epsilon &= \Bigl\{\lambda \in \BC \mid \Bigl({\rm Re}\,\lambda+ \frac{\gamma^2}{\alpha+\beta}
+\epsilon\Bigr)^2 + ({\rm Im}\, \lambda)^2 \geq \Bigl(\frac{\gamma^2}{\alpha+\beta} + \epsilon\Bigr)^2
\Bigr\}, \\
\Lambda_{\epsilon, \nu_0} &= K_\epsilon \cap \Sigma_\epsilon \cap \{\lambda \in \BC \mid |\lambda| \geq \nu_0\}.
\end{aligned}\end{equation}
\begin{remark} If one considers the inhomogeneous problem:
\begin{equation}\label{semigroup:?}
\pd_t (\rho, \bu) + \CA^s_{q,r}(\rho, \bu) = (F, \bG) \quad\text{for $t > 0$}, 
\quad (\rho, \bu)|_{t=0} = (0, 0) \in \CH^s_{q,r}(\HS),
\end{equation}
 then one may infer from Theorem \ref{thm:0} that there holds
\begin{equation}\label{ss1}
\int_{0}^{\infty}e^{-\omega t}(\|\pd_t(\rho,\bu)\|_{\CH^s_{q,1}(\HS)}+\|\CA^s_{q,1}(\rho,\bu)\|_{\CH^s_{q,1}(\HS)})dt\leq C\int_{0}^{\infty}e^{-\omega t}\|(F,\bG)\|_{\CH^s_{q,1}(\HS)}dt.
\end{equation}
This estimate follows from the Duhamel principle and the estimate for the semigroup, that is, by virtue of the Duhamel principle, the solution $(\rho, \mathbf u)$ to the inhomogeneous problem  is given by $$(\rho, \mathbf u)=\int_{0}^{t}T(t-s)(F, \mathbf G)(s)ds.
$$
Then, 
$$\|(\rho,\mathbf u)\|_{\mathcal D^s_{q,1}(\mathbb R^N_+)}\leq C \int_{0}^{t}\|T(t-s)(F, \mathbf G)(s)\|_{\mathcal D^s_{q,1}(\mathbb R^N_+)}ds.$$
Therefore, by Fubini's theorem, change of variables (let $t-s=\ell$) and Theorem 1.1, we have
\begin{align*}
\int_{0}^{\infty}e^{-\omega t}\|(\rho,\mathbf u)\|_{\mathcal D^s_{q,1}(\mathbb R^N_+)}dt&\leq C \int_{0}^{\infty}(\int_{0}^{t}e^{-\omega t}\|T(t-s)(F, \mathbf G)(s)\|_{\mathcal D^s_{q,1}(\mathbb R^N_+)}ds)dt\\
&\leq C \int_{0}^{\infty}(\int_{s}^{\infty}e^{-\omega t}\|T(t-s)(F, \mathbf G)(s)\|_{\mathcal D^s_{q,1}(\mathbb R^N_+)}dt)ds\\
&\leq C \int_{0}^{\infty}e^{-\omega s}(\int_{0}^{\infty}e^{-\omega \ell}\|T(\ell)(F, \mathbf G)(s)\|_{\mathcal D^s_{q,1}(\mathbb R^N_+)}d\ell)ds\\
&\leq C \int_{0}^{\infty}e^{-\omega s}\|(F, \mathbf G)(s)\|_{\mathcal H^s_{q,1}(\mathbb R^N_+)}ds.
\end{align*}
We first prove the maximal $L_1$  regularity for $\pd_t T(t)(\rho_0, \bu_0)$, and then
by \eqref{semigroup:1}, the $L_1$ estimate of $\CA^s_{q,1}T(t)(\rho_0, \bu_0)$ follows, which is 
$$\| e^{-\omega t} \CA^s_{q,1} T(\cdot)(\rho_0, \bu_0)\|_{L_1((0,\infty),\CH^s_{q,1}(\HS))}\leq C\|e^{-\omega t}\pd_tT(\cdot)(\rho_0, \bu_0)\|_{L_1((0,\infty),\CH^s_{q,1}(\HS))}.
$$
Thus, we have the standard maximal $L_1$ regularity, which read as
\begin{equation}\label{smr1}
\int_{0}^{\infty}e^{-\omega t}(\|\pd_tT(t)(\rho_0,\bu_0)\|_{\CH^s_{q,1}(\HS)}+\|\CA^s_{q,1}T(t)(\rho_0,\bu_0)\|_{\CH^s_{q,1}(\HS)})dt\leq C\|(\rho_0,\bu_0)\|_{\CH^s_{q,1}(\HS)}.
\end{equation}
\end{remark}
Theorem \ref{thm:0} may be proved by real interpolation theorem with the help of 
the following theorem. 
\begin{thm}
\label{thm:1}
Let $1 < q < \infty$, $1 \leq r < \infty$, $-1+1/q < s < 1/q$, 
and $\epsilon \in (0, \pi/2)$. Then, 
there exists  a large constant $\omega > 0$  
 such that 
 for every $\lambda \in \Lambda_{\epsilon, \omega}$ and
$(f, \bg) \in \CH^s_{q,r}(\HS)$, there exists a unique solution  
$(\rho, \bu) \in \CD^s_{q,r}(\HS)$ to \eqref{Eq:Resolvent} satisfying
\begin{align}
\label{Est:Resolvent 1}
\|\lambda(\rho, \bu)\|_{\CH^s_{q,r}(\HS)} + \|\bu\|_{B^{s+2}_{q,r}(\HS)} 
&\leq C\|(f, \bg)\|_{\CH^s_{q,r}(\HS)}
\end{align}
Moreover, let $\sigma$ be a small positive number such that
$-1+1/q < s-\sigma < s < s+\sigma < 1/q$.  Then, 
there exist $\bu_1$, $\bu_2 \in B^{s+2}_{q,r}(\HS)^N$ such that $\bu=\bu_1 + \bu_2$ and 
for any $\lambda \in \Lambda_{\epsilon, \gamma}$ there hold
\begin{equation}
	\label{Est:Resolvent 2}\begin{aligned}
\|\bu_1\|_{B^{s+2}_{q,r}(\HS)} & \leq C|\lambda|^{-\frac{\sigma}{2}}\|\bg\|_{B^{s+\sigma}_{q,r}(\HS)}, \\
\|\pd_\lambda 
\bu_1\|_{B^{s+2}_{q, r}(\HS)} & \leq C|\lambda|^{-(1-\frac{\sigma}{2})}\|\bg\|_{B^{s-\sigma}_{q, r}(\HS)}
\end{aligned}\end{equation}
for any $\bg \in C^\infty_0(\HS)^N$ as well as 
\begin{equation}\label{Est:Resolvent 2}\begin{aligned}
\|(\rho, \bu_2)\|_{\CD^s_{q,r}(\HS)}
& \leq C|\lambda|^{-1}\|(f, \bg)\|_{\CH^s_{q,r}(\HS)}, \\
\|\pd_\lambda (\rho, \bu_2)\|_{\CD^s_{q,r}(\HS)}
& \leq C|\lambda|^{-2}\|(f, \bg)\|_{\CH^s_{q,r}(\HS)}
\end{aligned}\end{equation}
for any $(f, \bg) \in \CH^s_{q,r}(\HS)$.
\end{thm}
\begin{remark} The conditions $1 < q < \infty$, $1 \leq r < \infty$ and $-1+1/q < s < 1/q$
assure that  $C^\infty_0(\Omega)$ is a dense subset of 
$B^s_{q,r}(\Omega)$ for $\Omega \in \{\BR^N, \HS\}$. This fact 
is an important point for our analysis in this paper. For a proof of this fact, refer to 
\cite[Theorems 2.9.3 and 2.10.3]{Tbook2}. 
\end{remark}
The rest of this  paper is unfold as follows. In the next section, we recall the notation of functional spaces. Then, in Sect. \ref{sec:3}, we prove boundedness properties of integral operators that will appear in the solution formula for \eqref{Eq:Resolvent} given in Sect. \ref{sec:4}. In Section \ref{sec:5}, Theorem \ref{thm:1} will be proved in the $\BR^N$ case, and in Sect. \ref{sec:6}, 
Theorem \ref{thm:1} is proved in the half space case. Finally, in Sect. \ref{sec:7}, 
we shall prove Theorem \ref{thm:0}. 
\section{Preliminaries}
\subsection{Notation}
Let us fix the symbols in this paper. Let $\BR$, $\BN$, and $\BC$ be the set of all real, natural, 
complex numbers, respectively, while let $\BZ$ be the set of all integers. Moreover, 
$\mathbb K$ stands for either $\BR$ or $\BC$. 
Set $\BN_0 := \BN \cup \{0\}$. \par
For $N\in \BN$ and a Banach space $X$, let $\CS(\BR^N;X)$ be the Schwartz class of 
$X$-valued rapidly decreasing functions on $\BR^N$. We denote $\mathcal{S'}(\BR^N;X)$ by the space of $X$-valued tempered distributions, which means the set of all 
continuous linear mappings from $\CS(\BR^N)$ to  $X$.  For $N\in \BN$, 
we define the Fourier transform $f \mapsto \CF [f]$ from $\CS(\BR^N;X)$ onto itself and its inverse  as
\begin{equation}
\CF [f](\xi):=\int_{\BR^{N}}f(x) e^{-ix\cdot\xi }\,dx, \qquad 
\CF_{\xi}^{-1} [g] (x) :=\frac{1}{(2\pi)^N}\int_{\BR^{N}}g(\xi) e^{ix\cdot\xi }\,d\xi,
\end{equation}
respectively. In addition, we define the partial Fourier transform
$\CF'[f(\,\cdot\,, x_N)] = \hat f(\xi', x_N)$ 
and partial inverse Fourier transform $\CF^{-1}_{\xi'}$ by
\begin{align*}
\CF'[f(\,\cdot\,, x_N)] (\xi') &:= \hat f(\xi', x_N) = \int_{\BR^{N-1}} f(x', x_N) e^{-ix'\cdot\xi'} \,dx',\\
\CF^{-1}_{\xi'}[g (\,\cdot\,, x_N)](x') &:= \frac{1}{(2\pi)^{N-1}}\int_{\BR^{N-1}}
g(\xi', x_N) e^{ix'\cdot\xi'}\,d\xi',
\end{align*}
where we have set $x'=(x_1,\cdots,x_{N-1}) \in \BR^{N-1}$ and $\xi' = (\xi_1, \cdots, \xi_{N-1})\in\BR^{N-1}$.
For $N \ge 2$, we set $(\bff, \bg)_{\BR^N_+} = \int_{\BR^N_+} \bff (x) \cdot \bg (x) \,dx$ 
for $N$-vector functions $\bff$ and $\bg$ on $\BR^N_+$, 
where we will write $(\bff, \bg) = (\bff, \bg)_{\BR^N_+}$ for short if there is no confusion.
For a Banach space $X$, $\|\cdot\|_X$ denotes its
norm.  For Banach spaces $X$ and $Y$, $X\times Y$ denotes the product of $X$ and $Y$,
that is $X\times Y = \{(x, y) \mid x \in X, \enskip y \in Y\}$, while  $\|\cdot\|_{X\times Y}$ denotes
its norm. 
$X \hookrightarrow Y$ means that $X$ is continuously imbedded into $Y$, that is $X \subset Y$
and $\|x\|_Y \leq C\|x\|_X$ with some constant $C$. For any interpolation couple $(X, Y)$
of Banach spaces $X$ and $Y$, the operations $(X, Y) \to (X, Y)_{\theta, p}$ and 
$(X, Y) \to (X, Y)_{[\theta]}$ are called the real interpolation functor for each $\theta \in (0, 1)$ and $
p \in [1, \infty]$ and the complex interpolation functor for each $\theta \in (0,1)$, respectively.
By $C > 0$ we will often denote a generic constant that does not depend on the quantities at stake.
For  differentiation with respect to space variables $x=(x_1, \ldots, x_N)$,
$D^\delta f :=\pd_x^\delta f = \pd^{|\delta|}f/{\pd x_1^{\delta_1} \cdots}$ $\pd x_N^{\delta_N}
$ for multi-index $\delta=(\delta_1, \ldots, \delta_N)$ with  $|\delta| = \delta_1+\cdots + \delta_N$. 
For the notational simplicity, we write $\nabla f = \{ \pd_x^\delta f \mid|\delta|=1\}$, $\nabla^2 f
= \{\pd_x^\delta f\mid |\delta|=2\}$, $\bar\nabla f=(f, \nabla f)$, and $\bar\nabla^2 f = 
(f, \nabla f, \nabla^2 f)$. 

\subsection{Function spaces on $\BR^N$}
Let us recall the definitions of Bessel potential spaces and inhomogeneous Besov spaces. In the following, let $s \in \BN$ and $p\in(1,\infty)$.
Bessel potential spaces $H^s_p(\BR^N)$ are defined as the set of all $f\in\CS' (\BR^N)$ such that 
$\lVert f \rVert_{H^s_p(\BR^N)} < \infty$, where the norm $\lVert \,\cdot\, \rVert_{H^s_p(\BR^N)}$ is defined by
\begin{equation}
\|f\|_{H^s_p(\BR^N)}:=\left\|\CF^{-1}_\xi\left[(1+|\xi|^2)^\frac{s}{2}\CF [f](\xi)\right]\right\|_{L_p(\BR^N)}.
\end{equation}
It is well-known that, if $s=m\in \BN_0$, then $H^s_p(\BR^N)$ coincides with the classical Sobolev space $W^m_p (\BR^N)$,
see, e.g., \cite[Theorem 3.7]{AFbook}. \par
To define inhomogeneous Besov spaces, we need to introduce Littlewood-Paley decomposition.
 Let $\phi \in \CS (\BR^N)$ with $\supp \phi = \{ \xi \in \BR^N \mid 1 \slash 2 \le \lvert \xi \rvert \le 2 \}$ such that
$\sum_{k \in \BZ} \phi (2^{- k} \xi) = 1$ for all $\xi \in \BR^N \setminus \{0\}$. 
Then, define
\begin{equation}\label{little:1}
\phi_k := \CF^{- 1}_\xi [\phi (2^{- k} \xi)], \quad k \in \BZ, \qquad \CF[\psi]= 
1 - \sum_{k \in \BN} \phi (2^{- k} \xi).
\end{equation}
For $1 \le p, q \le \infty$ and $s \in \BR$ we denote
\begin{equation}\label{pld:2}
\lVert f \rVert_{B^s_{p, q} (\BR^N)} := 
\left\{\begin{aligned}
& \lVert \psi * f \rVert_{L_p (\BR^N)} + \bigg(\sum_{k \in \BN} \Big(2^{s k} \lVert \phi_k * f \rVert_{L_p (\BR^N)} \Big)^q \bigg)^{\frac{1}{q}}
& \quad & \text{if $1\le q < \infty$}, \\
& \lVert \psi * f \rVert_{L_p (\BR^N)} + \sup_{k \in \BN} \Big(2^{s k} \lVert \phi_k * f \rVert_{L_p (\BR^N)} \Big)
& \quad & \text{if $q=\infty$}.
\end{aligned}\right.
\end{equation}
Here, $f \, * \, g$ means the convolution between $f$ and $g$.
Then inhomogeneous Besov spaces $B^s_{p, q} (\BR^N)$ are defined as the sets of all 
$f \in \CS' (\BR^N)$ such that $\lVert f \rVert_{B^s_{p, q} (\BR^N)} < \infty$. \par
It is well-known that $B^s_{p,q}(\BR^N)$ may be \textit{characterized} by means of real interpolation.
In fact, for $-\infty<s_0<s_1<\infty$, $1<p<\infty$, $1\le q\le\infty$, and $0<\theta<1$, it follows that
\begin{equation}
\label{interpolation}
B^{\theta s_0+(1-\theta)s_1}_{p,q}(\BR^N)=\left(H^{s_0}_p(\BR^N),H^{s_1}_p(\BR^N)\right)_{\theta,q},
\end{equation}
cf.  \cite[Theorem 8]{M73},  \cite[Theorem 2.4.2]{Tbook}.
\subsection{Function spaces on $\HS$}
Let $\CD'(\HS)$ be the collection of all complex-valued distributions on $\HS$.
Let $s\in\BR$, $p\in(1,\infty)$, and $q\in[1,\infty]$. Then for any $X\in\{H^s_p,B^s_{p,q}\}$,
the space $X(\HS)$ is the collection of all $f\in\CD'(\HS)$ such that
there exists a function 
$g \in X(\BR^N)$ with $g|_{\HS}=f$.
 Moreover, the norm of $f\in X(\HS)$ is given by
\begin{equation}
\lVert f \rVert_{X(\HS)} = \inf \lVert g \rVert_{X(\BR^N)},
\end{equation}
where the infimum is taken over all $g \in X(\BR^N)$ such that its restriction 
$g \vert_{\HS}$ coincides in $\CD' (\HS)$ with $f$. We also define
\begin{equation}
X_0 (\HS) := \{f \in X (\BR^N) \mid \supp f \subset \overline{\HS}\}.
\end{equation}
Clearly, we always have $X_0(\HS)\hookrightarrow X(\HS)$.  \par
According to \cite[Section 2.9]{Tbook2}, for $s\in\BR$, $p\in(1,\infty)$, and $q\in[1,\infty)$,
we have the following density result:
\begin{equation}
\label{desity}
X_0 (\HS)=\overline{C^\infty_0(\HS)}^{\lVert \, \cdot \, \rVert_{X(\BR^N)}}.
\end{equation} 
Here, $X(\HS)$ and $X_0(\HS)$ may coincide if one restricts $s$ such that $-1+1/p<s<1/p$.
\begin{prop}
\label{prop:2.1}
Let $1<p<\infty$, $1\le q  < \infty$, and $-1+1/p<s<1/p$. Then $H^s_p(\HS)=H^s_{p,0}(\HS)$ as well as $B^s_{p,q}(\HS)=B^s_{p,q,0}(\HS)$.
\end{prop}
Finally, let us mention duality results. If one considers function spaces on $\BR^N$,
then it follows that $(H^s_p(\BR^N))'=H^{-s}_{p'}(\BR^N)$ and $(B^s_{p,q}(\BR^N))'=B^{-s}_{p',q'}(\BR^N)$
for all $s\in\BR$, $p\in(1,\infty)$, and $q\in[1,\infty)$, where $p'$ and $q'$ stand for
the H\"older conjugate of $p$ and $q$, respectively. Indeed, these proofs may be found in 
 \cite[Sect. 6]{M73}, \cite[Theorem 2.11.2]{Tbook}.
However, if one considers function spaces on $\HS$, one has to pay attention to discuss the dual of 
function spaces due to the existence of the boundary $\pd\HS$. Let us summarize the duality results 
and real interpolation functors for the half space case. 
\begin{prop}
\label{prop-duality}
Let $p\in(1,\infty)$. Then the following assertions are valid.
\begin{itemize}
\item[$(1)$]~ For $s \in \BR$, there holds 
\begin{equation}
(H^s_{p,0} (\HS))' = H^{- s}_{p'} (\HS) .
\end{equation}
\item[$(2)$]~ For  $- \infty < s \le 1 \slash p$, there holds
\begin{equation}
(H^s_p (\HS))' = H^{- s}_{p', 0} (\HS) .
\end{equation}
\item[$(3)$]~ For $-\infty<s_0<s_1<\infty$, $1<p<\infty$, $1\le q\le\infty$, and $0<\theta<1$, there 
holds 
\begin{equation}
\label{interpolation-half}
B^{\theta s_0+(1-\theta)s_1}_{p,q}(\HS)=\left(H^{s_0}_p(\HS),H^{s_1}_p(\HS)\right)_{\theta,q}.
\end{equation}
\end{itemize}	
\end{prop} 
\begin{proof} For proofs of $(1)$ and $(2)$, refer to \cite[Section 2.10]{Tbook2}, and 
for a proof of $(3)$, refer to  \cite[Theorem 8, Theorem 11]{M73}, 
\cite[Theorem 1.2.4]{Tbook2}.
\end{proof}

\subsection{Class of multipliers}
Let $U$ be a domain in $\BC$. Let $m (\lambda,\xi')$
be a function defined on $U\times(\BR^{N-1}\setminus\{0\})$ which is holomorphic in $\lambda \in U$ and 
infinitely many times differentiable with respect to $\xi' \in \BR^{N-1}\setminus\{0\}$. 
If there exists a real number $\kappa$ such that for any multi-index $\delta'\in\BN^{N-1}_0$ and $(\lambda, \xi') \in \Sigma_{\epsilon, \lambda_0} \times (\BR^{N-1}\setminus\{0\})$ there hold the estimate
\begin{equation}
\left|D^{\delta'}_{\xi'}m(\lambda, \xi')\right| \leq C_{\delta'}\left(|\lambda|^{1/2}
+|\xi'|\right)^{\kappa-|\delta'|}
\end{equation}
for some constant $C_{\delta'}$ depending on $\delta'$, then $m(\lambda,\xi')$ is called a multiplier of order $\kappa$ with type $\BM_\kappa(U)$.\par
Obviously, for any $m_i \in \BM_{\kappa_i}(U)$ ($i=1,2$), we see that $m_1m_2 \in \BM_{\kappa_1+\kappa_2}(U)$.  Notice that 
$|\xi'|^2 \in \BM_2(\BC)$ and $\xi_j \in \BM_1(\BC)$, but any functions of $|\xi'|$ is usually not in $\BM_{\kappa}(U)$ for any
$\kappa$ and $U$.
%
\subsection{Interpolation of small $\ell^p$ spaces of vector-valued sequences}
Let $X$ be a  Banach space, and $(a_\nu)_{\nu=-\infty}^\infty$ be a sequence in $X$. For $s\in\BR$, the norm $\lVert \,\cdot\, \rVert_{\ell^s_q(X)}$ is defined by
\begin{align*}
	\|(a_\nu)\|_{\ell^s_q(X)}=\left\{\begin{aligned}
		&\left(\sum_{\nu=-\infty}^{\infty}(2^{\nu s}\|a_\nu\|_X)^q\right)^{\frac{1}{q}}   &\quad &\text{($1\leq q<\infty$)}, \\
		&\sup_{\nu \in \BZ} 2^{\nu s}\|a_\nu\|_X
		 &\quad& \text{($q=\infty$)}, 
	\end{aligned}\right.
\end{align*}
where \begin{align*}
	\ell^s_q(X)=\{(a_\nu)_{\nu=-\infty}^\infty\mid  \|a_\nu\|_{\ell^s_q(X)}<\infty\}.
\end{align*}
\begin{thm}\cite[Theorem 5.6.1]{BLbook}.
	Assume that  $1 \leq q_0\leq \infty$, $1 \leq q_1\leq \infty$ and that $s_0\neq s_1$. 
Then we have, for all $1 \leq   q\leq \infty$
	$$\left(\ell^{s_0}_{q_0}(X),\ell^{s_1}_{q_1}(X)\right)_{\theta,q}=\ell^s_q(X)$$
	where $s=(1-\theta)s_0+\theta s_1$.
\end{thm}
\section{Technical tools} \label{sec:3}
	We know the following three lemmas due to Enomoto-Shibata \cite[Lemma 3.1]{ES13}.
\begin{lem}\label{lem:3.1} Let $0 < \epsilon < \pi/2$ and $\nu_0 > 0$. 
Let $\Sigma_\epsilon$ and $\Lambda_{\epsilon, \nu_0}$ be the sets defined in
\eqref{sector:1}. Then, we have the following 
assertions.
\begin{itemize}
\item[$(1)$] 
For any $\lambda \in \Sigma_\epsilon$ and $\xi \in \BR^N$, there holds
$$|\alpha^{-1}\lambda + |\xi|^2| \geq (\sin(\epsilon/2))(\alpha^{-1}|\lambda| + |\xi|^2).$$
\item[$(2)$] Let $p(\lambda) = (\alpha+\eta_\lambda)^{-1}\lambda$, where 
$\eta_\lambda = \beta + \gamma^2\lambda^{-1} $. 
For any $\nu_0>0$ there exist constants $\epsilon' \in (0, \pi/2)$ and $c_1>0$ depending 
solely on $\epsilon$ and $\nu_0$ such that for any $\lambda \in \Lambda_{\epsilon, \nu_0}$ 
and $\xi\in\BR^N$, there hold
$$|\arg p(\lambda)| \leq \pi-\epsilon', \quad 
|p(\lambda)+|\xi|^2| \geq c_1(|\lambda| + |\xi|^2).$$
\item[$(3)$] There exists a constant $c_2>0$ depending solely on $\alpha$, $\beta$ and $\epsilon$
such that for any $\lambda \in \Sigma_\epsilon$ there holds $|\alpha + \eta_\lambda| \geq c_2$.
\end{itemize}
\end{lem}
By Lemma \ref{lem:3.1}, we have the following multiplier estimates which is used to estimate 
solution formulas in $\BR^N$.
\begin{lem}\label{lem:3.2}
 Let $0 < \epsilon < \pi/2$, $\nu_0 > 0$ and $s \in \BR$.  
Let $\Sigma_\epsilon$ and $\Lambda_{\epsilon, \nu_0}$ be the sets defined in
\eqref{sector:1}. 
Then, for any $\delta \in \BN_0^N$
there hold
$$|D^{\delta}_\xi (\alpha^{-1}\lambda + |\xi|^2)^s| \leq C_\delta(|\lambda|^{1/2}+|\xi|)^{s-|\delta|}
$$
for any $(\lambda, \xi) \in \Sigma_\epsilon \times (\BR^N/\{0\})$ as well as
$$|D^{\delta}_\xi (p(\lambda) + |\xi|^2)^s| \leq C_\delta(|\lambda|^{1/2}+|\xi|)^{s-|\delta|}
$$
for any $(\lambda, \xi) \in \Lambda_{\epsilon, \nu_0} \times (\BR^N/\{0\})$,
where $p(\lambda) = (\alpha+\eta_\lambda)^{-1}\lambda$ and 
$\eta_\lambda = \beta+\gamma^2\lambda^{-1}$.
\end{lem}
Set 
\begin{equation}\label{halfsymbol:3.1}\begin{aligned} 
	 A= \sqrt{p(\lambda) + |\xi'|^2}, 
	\quad B&=\sqrt{\alpha^{-1}\lambda + |\xi'|^2},\quad K=(\alpha+\eta_\lambda)A + \alpha B, \\
	\CM(x_N)& = \frac{e^{-Ax_N} - e^{-Bx_N}}{A-B}. 
\end{aligned}\end{equation}
These symbols appear in the solution formula \eqref{sol:form5} below.
We know the following multiplier's  estimates.
\begin{lem}\label{lem:3.3} Let $0 < \epsilon < \pi/2$, $\nu_0 > 0$ and $s \in \BR$.  
Then, for any multi-index $\delta' \in \BN_0^{N-1}$ there hold
$$|D^{\delta'}_{\xi'} M^s| \leq C_{\delta'}(|\lambda|^{1/2} + |\xi'|)^{s-|\delta'|}
$$
for any $(\lambda, \xi' ) \in \Lambda_{\epsilon, \nu_0} \times(\BR^{N-1}/\{0\})$, 
where $M \in \{A, B, K\}$. 
\end{lem}
Using Lemma \ref{lem:3.3} we have the following lemma.
\begin{lem}\label{lem:3.4} 
Let $0 < \epsilon < \pi/2$, $\nu_0 > 0$, $s \in \BR$, and $x_N > 0$.
Then, for any multi-index $\delta' \in \BN_0^{N-1}$ and $\lambda \in \Lambda_{\epsilon, \nu_0}$,
there hold
\begin{align}
|D^{\delta'}_{\xi'} e^{- Mx_N}| \leq C_{\delta'}(|\lambda|^{1/2} + |\xi'|)^{-|\delta'|}
e^{-c(|\lambda|^{1/2}
+|\xi'|)x_N}, \label{lem.3.4.1}\\
|D^{\delta'}_{\xi'} (B \CM(x_N))| \leq C_{\delta'}(|\lambda|^{1/2} + |\xi'|)^{-|\delta'|}
e^{-c(|\lambda|^{1/2}+|\xi'|)x_N} \label{lem.3.4.2}
\end{align}
with some positive constant $c$, where $M \in \{A, B\}$.
\end{lem}
\begin{proof}
For any $\theta \in [0, 1]$, by Bell's formula we have
\begin{align*}
&|D^{\delta'}_{\xi'}e^{-((1-\theta)A + \theta B)x_N}|\\
&\quad\leq  C_{\delta'}\sum_{\ell=1}^{|\delta'|}x^\ell_N|e^{-((1-\theta)A + \theta B)x_N)}|
\Bigl(\sum_{\delta'_1+\cdots +\delta'_\ell = \delta' \atop |\delta'_\ell| \geq 1}
|D^{\delta'_1}((1-\theta)A + \theta B)|\cdots
|D^{\delta'_\ell}((1-\theta)A + \theta B)|\Bigr). 
\end{align*}
Using Lemma \ref{lem:3.3}, there exists a constant $c>0$ such that 

$$|e^{-((1-\theta)A + \theta B)x_N)}| \leq e^{-2c(|\lambda|^{1/2}+|\xi'|)x_N}.$$
Therefore, we have
\begin{equation}\label{kernel:3.3}
|D^{\delta'}_{\xi'}e^{-((1-\theta)A + \theta B)x_N}| \leq C_{\delta'}(|\lambda|^{1/2}+|\xi'|)^{-|\delta'|}
e^{-c(|\lambda|^{1/2}+|\xi'|)x_N}.
\end{equation}
Therefore, setting $\theta=0$ or $\theta = 1$, we have \eqref{lem.3.4.1}.

We write
$$B\CM(x_N) = Bx_N \int^1_0 e^{-((1-\theta)A + \theta B)x_N)}\,d\theta.$$
Applying \eqref{kernel:3.3} and Lemma \ref{lem:3.3} implies  \eqref{lem.3.4.2}.
This completes the proof of Lemma \ref{lem:3.4}. 
\end{proof}

In this section, we record the following proposition, which plays a crucial role in the proof of Theorem~\ref{thm:1}.
\begin{prop} \label{prop:2}
Let $1 < q < \infty$, $\epsilon \in (0, \pi/2)$, $\lambda_0>0$, and $\lambda\in\Lambda_{\epsilon, \lambda_0}$. Suppose that $m_0 \in \BM_{0}$. 
Define the integral operators $L_i$, $i=1,\cdots,6$, by the formula:
\begin{align*}
L_1(\lambda)f &= \int^\infty_0\CF^{-1}_{\xi'}\left[m_0(\lambda, \xi')
B^3\CM(x_N)\CM(y_N)\hat f(\xi', y_N)\right](x')\,dy_N,\\
L_2(\lambda)f &= \int^\infty_0\CF^{-1}_{\xi'}\left[m_0(\lambda, \xi')
B^2\CM(x_N)e^{-Ay_N}\hat f(\xi', y_N)\right](x')\,dy_N, \\
L_3(\lambda)f &= \int^\infty_0\CF^{-1}_{\xi'}\left[m_0(\lambda, \xi')
B^2\CM(x_N)e^{-By_N}\hat f(\xi', y_N)\right](x')\,dy_N, \\
L_4(\lambda)f &= \int^\infty_0\CF^{-1}_{\xi'}\left[m_0(\lambda, \xi')
B^2e^{-Ax_N}\CM(y_N)\hat f(\xi', y_N)\right](x')\,dy_N, \\ 
L_5(\lambda)f &= \int^\infty_0\CF^{-1}_{\xi'}\left[m_0(\lambda, \xi')
B^2e^{-Bx_N}\CM(y_N)\hat f(\xi', y_N)\right](x')\,dy_N,\\
L_6(\lambda)f &= \int^\infty_0\CF^{-1}_{\xi'}\left[m_0(\lambda, \xi')
Be^{-Jx_N}e^{-Qy_N}\hat f(\xi', y_N)\right](x')\,dy_N,
\end{align*}
respectively, where $(J,Q)$ stands for an element of $\{(A,A),(A,B),(B,A),(B,B)\}$ in the formula of $L_6$.
Then for every $f\in L_q(\HS)$, it holds
\begin{equation}
\|L_i(\lambda)f\|_{L_q(\HS)} \leq C_q\|f\|_{L_q(\HS)}\quad (i=1,2, 3, 4, 5,6).
\end{equation}
\end{prop}
\begin{prop} \label{prop:3}
	Let $1 < q < \infty$, $\epsilon \in (0, \pi/2)$, $\lambda_0>0$, and $\lambda\in
\Lambda_{\epsilon, \lambda_0}$. Suppose that $m_0 \in \BM_{0}$. 
	Define the integral operators $P_i$, $i=1,\cdots,6$, by the formula:
	\begin{align*}
		P_1(\lambda)f &= \int^\infty_0\CF^{-1}_{\xi'}\left[m_0(\lambda, \xi')
		B^2\pd_\lambda \left(B^3\CM(x_N)\CM(y_N)\right)\hat f(\xi', y_N)\right](x')\,dy_N,\\
		P_2(\lambda)f &= \int^\infty_0\CF^{-1}_{\xi'}\left[m_0(\lambda, \xi')
		B^2\pd_\lambda \left(B^2\CM(x_N)e^{-Ay_N}\right)\hat f(\xi', y_N)\right](x')\,dy_N, \\
		P_3(\lambda)f &= \int^\infty_0\CF^{-1}_{\xi'}\left[m_0(\lambda, \xi')
		B^2\pd_\lambda \left(B^2\CM(x_N)e^{-By_N}\right)\hat f(\xi', y_N)\right](x')\,dy_N, \\
		P_4(\lambda)f &= \int^\infty_0\CF^{-1}_{\xi'}\left[m_0(\lambda, \xi')
		B^2\pd_\lambda \left(B^2e^{-Ax_N}\CM(y_N)\right)\hat f(\xi', y_N)\right](x')\,dy_N, \\ 
		P_5(\lambda)f &= \int^\infty_0\CF^{-1}_{\xi'}\left[m_0(\lambda, \xi')
		B^2\pd_\lambda \left(B^2e^{-Bx_N}\CM(y_N)\right)\hat f(\xi', y_N)\right](x')\,dy_N,\\
		P_6(\lambda)f &= \int^\infty_0\CF^{-1}_{\xi'}\left[m_0(\lambda, \xi')
		B^2\pd_\lambda \left(Be^{-Jx_N}e^{-Qy_N}\right)\hat f(\xi', y_N)\right](x')\,dy_N,
	\end{align*}
	respectively, where $(J,Q)$ stands for an element of $\{(A,A),(A,B),(B,A),(B,B)\}$ in the formula of $P_6$.
	Then for every $f\in L_q(\HS)$, it holds
	\begin{equation}
		\|P_i(\lambda)f\|_{L_q(\HS)} \leq C_q\|f\|_{L_q(\HS)}\quad (i=1,2, 3, 4, 5,6).
	\end{equation}
\end{prop}
To show Proposition \ref{prop:2}, we need the following propositions.
\begin{prop}\cite[A.3 p.271]{S70}.
\label{lem:2}
Let $1< q < \infty$. Define the integral operator $G$ by the formula:
\begin{equation}
Gf(x_N) = \int^\infty_0\frac{f(y_N)}{x_N+y_N}\,dy_N.
\end{equation}
Then, for every $f \in L_q(0,\infty)$ there exists a constant $A_q$ such that
\begin{equation}
\|Gf\|_{L_q((0, \infty))} \leq A_q\|f\|_{L_q((0, \infty))}.
\end{equation}
\end{prop}
\begin{prop}\cite[Theorem 2.3]{SS01}.
\label{lem:3} Let $X$ be a Banach space,
$\sigma$ be a real number satisfying $0 < \sigma \leq 1$, and $m$ be a nonnegative integer. Set $\zeta = m+\sigma-(N-1)$.
In addition, let $\ell(\sigma)$ be an integer part of $\sigma$.
Suppose that a function $f(\xi') \in C^{m+1+\ell(\sigma)}(\BR^{N-1}\setminus\{0\}, X)$ satisfies the following conditions:
\begin{itemize}
\item[(1)] For every $\delta' \in \BN_0^{N-1}$ satisfying $|\delta'| \leq m$, it holds $D^{\delta'}_{\xi'}f(\xi') \in L_1(\BR^{N-1}, X)$.
\item[(2)] For every $\delta' \in \BN_0^{N-1}$ satisfying $|\delta'| \leq m+1+\ell(\sigma)$, there exists a constant $C_{\delta'}$ such that $\|D^{\delta'}_{\xi'}f(\xi')\|_X
\leq C_{\delta'}|\xi'|^{\zeta-|\delta'|}$ for all $\xi'\in\BR^{N-1}\setminus\{0\}$.
\end{itemize}
Then, there exists a constant $C_{N,\zeta}$ depending on $N$ and $\zeta$ such that
\begin{equation}
\left\|\CF^{-1}_{\xi'}[f](x')\right\|_X \leq C_{N,\zeta}
\left(\max_{|\delta'|\leq m+1+\ell(\sigma)}\,C_\gamma\right)
|x' |^{-N-1+\zeta}, \quad(x'\in\BR^{N-1}\setminus\{0\}).
\end{equation}	
\end{prop}	
\begin{proof}[Proof of Proposition \ref{prop:2}]
Here, we only consider the estimate for $ L_1(\lambda)f $, since the others may be proved in a similar way.
First, we rewrite $B$ as
\begin{equation}
B = \dfrac{B^2}{B} = \dfrac{\lambda^{1/2}}{B} \lambda^{1/2}+ \dfrac{|\xi'|}{B}|\xi'|.
\end{equation}
We set $m^1_0(\lambda, \xi') = m_0(\lambda, \xi')\lambda^{1/2}B^{-1}$ and
$m^2_0(\lambda, \xi') = m_0(\lambda, \xi)|\xi'|B^{-1}$.
Define
\begin{align*}
L_1^{(1)}(\lambda)f &:=  \int^\infty_0\CF^{-1}_{\xi'}\left[m^1_0(\lambda, \xi')
\lambda^{1/2}B^2\CM(x_N)\CM(y_N)\hat f(\xi', y_N)\right](x')\,dy_N, 
\\
L^{(2)}_1(\lambda)f &:= \int^\infty_0\CF^{-1}_{\xi'}\left[m^2_0(\lambda, \xi')
|\xi'|B^2\CM(x_N)\CM(y_N)\hat f(\xi', y_N)\right](x')\,dy_N.
 \end{align*}
Since for all $\delta'\in \BN_0^{N-1}$ we have
\begin{equation}\label{est:2}\begin{aligned}
\left|D^{\delta'}_{\xi'}m_0^1(\lambda, \xi')\right|
\leq C_{\delta'}\left(|\lambda|^{1/2}+|\xi'|\right)^{-|\delta'|}, \quad 
\left|D^{\delta'}_{\xi'}m_0^2(\lambda, \xi')\right|
\leq C_{\delta'}|\xi'|^{-|\delta'|}
\end{aligned}\end{equation}
as well as
\begin{align*}
\left|D^{\delta'}_{\xi'}B\CM(x_N)\right| & \leq C_{\delta'}e^{-c(|\lambda|^{1/2}+|\xi'|)x_N}
\left(|\lambda|^{1/2}+|\xi'|\right)^{-|\delta'|}, \\
	\left|D^{\delta'}_{\xi'}B\CM(y_N)\right| & \leq C_{\delta'}
e^{-c(|\lambda|^{1/2}+|\xi'|)y_N}
\left(|\lambda|^{1/2}+|\xi'|\right)^{-|\delta'|},
\\
\left|D^{\delta'}_{\xi'}\, B\right| & \leq C_{\delta'}\left(|\lambda|^{1/2}+|\xi'|\right)^{1-|\delta'|}
\end{align*}
as follows from Lemma \ref{lem:3.4}, 
we see that 
\begin{equation}\label{est:3}\begin{aligned}
&\left|D^{\delta'}_{\xi'}\left(m^1_0(\lambda, \xi')\lambda^{1/2}B^2\CM(x_N)\CM(y_N)\right)\right|
 \leq C_{\delta'}|\lambda|^{1/2}\left(|\lambda|^{1/2}+|\xi'|\right)^{-|\delta'|}
e^{-(c/2)(|\lambda|^{1/2}+|\xi'|)(x_N+y_N)}.
\end{aligned}\end{equation}
By virtue of the identity:
\begin{equation}
e^{ix'\cdot\xi'} = \sum_{j=1}^{N-1}\frac{x_j}{i|x'|^2}\frac{\pd}{\pd \xi_j}
e^{ix'\cdot\xi'},
\end{equation}
we may write
\begin{align*}
& \CF^{-1}_{\xi'}\left[m^1_0(\lambda, \xi')\lambda^{1/2}B^2\CM(x_N)\CM(y_N)\right](x')\\
&\quad = \frac{1}{(2\pi)^N}
\int_{\BR^{N-1}}\left(\sum_{|\delta'|=N}\left(\frac{x'}{i|x'|^2}\right)^{\delta'}
D^{\delta'}_{\xi'}e^{ix'\cdot \xi'}\right)
m^1_0(\lambda, \xi')\lambda^{1/2}B^2\CM(x_N)\CM(y_N)\,d\xi'\\
&\quad = \frac{1}{(2\pi)^N}\sum_{|\delta'|=N}\left(\frac{-x'}{i|x'|^2}\right)^{\delta'}
\int_{\BR^{N-1}}e^{ix'\cdot\xi'}
D^{\delta'}_{\xi'}\left(
m^1_0(\lambda, \xi')\lambda^{1/2}B^2\CM(x_N)\CM(y_N)\right)\,d\xi'.
\end{align*}
Hence, we obtain
\begin{equation}
\left|\CF^{-1}_{\xi'}\left[m^1_0(\lambda, \xi')\lambda^{1/2}B^2\CM(x_N)\CM(y_N)\right](x')\right|
\leq C|x'|^{-N}\int_{\BR^{N-1}}|\lambda|^{1/2}\left(|\lambda|^{1/2}+|\xi'|\right)^{-N}\,d\xi'.
\end{equation}
By changing of variables $\xi'=|\lambda|^{1/2}\eta'$, it follows that
\begin{equation}
\int_{\BR^{N-1}}|\lambda|^{1/2}\left(|\lambda|^{1/2}+|\xi'|\right)^{-N}\,d\xi'
= \int_{\BR^{N-1}}\left(1+|\eta'|\right)^{-N}\,d\eta' < \infty.
\end{equation}
Therefore, we have
\begin{equation}\label{est:4}
\left|\CF^{-1}_{\xi'}\left[m^1_0(\lambda, \xi')\lambda^{1/2}B^2\CM(x_N)\CM(y_N)\right](x')\right|
\leq C|x'|^{-N}.
\end{equation}
In addition, if we take $\delta'=0$ in \eqref{est:3}, it follows that
\begin{align*}
&\left|\CF^{-1}_{\xi'}\left[m^1_0(\lambda, \xi')\lambda^{1/2}B^2\CM(x_N)\CM(y_N)\right](x')\right| \\
&\quad  
\leq C\int_{\BR^{N-1}}|\lambda|^{1/2}e^{-(c/2)(|\lambda|^{1/2} + |\xi'|)(x_N+y_N)}\,
d\xi'\\
&\quad\leq \frac{C}{(x_N+y_N)^N}
\int_{\BR^{N-1}}\frac{|\lambda|^{1/2}}{\left(|\lambda|^{1/2}+|\xi'|\right)^{N}}\,d\xi'
\\ 
& \quad
= \frac{C}{(x_N+y_N)^N}\int_{\BR^{N-1}}\left(1+|\eta'|\right)^{-N}\,d\eta',
\end{align*}
which together with \eqref{est:4} implies
\begin{equation}\label{est:5}
\left|\CF^{-1}_{\xi'}\left[m^1_0(\lambda, \xi')\lambda^{1/2}B^2\CM(x_N)\CM(y_N)\right](x')\right| 
\leq C\left(|x'|+x_N+y_N\right)^{-N}.
\end{equation}
To shorten notation, set $\ell(x'):=\CF^{-1}_{\xi'}\left[m^1_0(\lambda, \xi')\lambda^{1/2}
B^2\CM(x_N)\CM(y_N)\right](x').$
Then $L^{(1)}_1(\lambda)f$ may be bounded by
\begin{align}
\|L^{(1)}_1(\lambda)f\|_{L_q(\BR^{N-1})}
& \le \int^\infty_0\|\ell*f(\,\cdot\,, y_N)\|_{L_q(\BR^{N-1})}\,dy_N \\
& \le \int^\infty_0\|\ell\|_{L_1(\BR^{N-1})}\|f(\,\cdot\,, y_N)\|_{L_q(\BR^{N-1})}\,dy_N.
\end{align}
Then \eqref{est:5} together with changing of variables $x' = (x_N+y_N) z'$ yields
\begin{align}
\|\ell\|_{L_1(\BR^{N-1})} & = \int_{\BR^{N-1}}\frac{C}{\left(|x'|+x_N+y_N\right)^N}\,dx' \\
& = \frac{C}{x_N+y_N}\int_{\BR^{N-1}}\frac{1}{\left(1+|z'|\right)^N}\,dz'.
\end{align}
Thus, we observe
\begin{equation}\label{est:6}
\|L^{(1)}_1(\lambda)f\|_{L_q(\BR^{N-1})} \leq C\int^\infty_0 
\frac{\|f(\,\cdot\,, y_N)\|_{L_q(\BR^{N-1})}}{x_N+y_N}\,dy_N.
\end{equation}
 From Proposition \ref{lem:2} and \eqref{est:6}, we have
\begin{equation}\label{est:7}
\begin{aligned}
&\|L^{(1)}_1(\lambda)f\|_{L_q(\HS)} \\
&\quad \leq C\left\|\int^\infty_0 
\frac{\|f(\,\cdot\,, y_N)\|_{L_q(\BR^{N-1})}}{(x_N+y_N)}\,dy_N
\right\|_{L_q((0, \infty))} 
\\
&\quad \leq C_q\left\|G	\|f\|_{L_q(\BR^{N-1})}\right\|_{L_q((0, \infty))}
\\
&\quad \leq C_q\|f\|_{L_q(\HS)}.
\end{aligned}
\end{equation}
\par
It remains to establish the estimate for $L^{(2)}_1(\lambda)$. 
In a similar way as in \eqref{est:3}, we obtain
\begin{align*}
&\left|D^{\delta'}_{\xi'}(B^2\CM(x_N)\CM(y_N))\right| \\
&\quad \leq C_{\delta'}e^{-c(|\lambda|^{1/2}+|\xi'|)x_N}
e^{-(c/2)(|\lambda|^{1/2}+|\xi'|)y_N}
\left(|\lambda|^{1/2}+|\xi'|\right)^{-|\delta'|} \\
&\quad \leq C_{\delta'}
e^{-(c/2)(|\lambda|^{1/2}+|\xi'|)(x_N+y_N)}
|\xi'|^{-|\delta'|}.
\end{align*}
Moreover,  by the second inequality in $\eqref{est:2}$, we have
\begin{equation}
\left|D^{\delta'}_{\xi'} \left(m^2_0(\lambda, \xi')|\xi'|\right)\right| \leq C_{\delta'}|\xi|^{1-|\delta'|},
\end{equation}
which yields
\begin{equation}\label{est:8}
\left|D^{\delta'}_{\xi'}\left(m^2_0(\lambda, \xi')|\xi'|B^2\CM(x_N)\CM(y_N)\right)\right|
\leq C_{\delta'}|\xi'|^{1-|\delta'|} e^{-(c/2)(|\lambda|^{1/2}+|\xi'|)(x_N+y_N)}.
\end{equation}
By \eqref{est:8} and Proposition \ref{lem:3} we obtain
\begin{equation}\label{est:9}
\left|\CF^{-1}_{\xi'}\left[m^2_0(\lambda, \xi')|\xi'|B^2\CM(x_N)\CM(y_N))\right](x')\right| \leq C|x'|^{-N}.
\end{equation}
 By \eqref{est:8}  we also obtain
\begin{align*}
\left|\int_{\BR^{N-1}} e^{ix'\cdot\xi'}
m^2_0(\lambda, \xi')|\xi'|B^2\CM(x_N)\CM(y_N))\,d\xi'\right| 
 &\leq C\int_{\BR^{N-1}}|\xi'|e^{-(c/2)|\xi'|(x_N+y_N)}\, d\xi'\\
&= \frac{C}{(x_N+y_N)^N} \int_{\BR^{N-1}}|\eta'|e^{-(c/2)|\eta'|}\,d\eta',
\end{align*}
where we have replaced $\eta'$ by $\xi'(x_N+y_N) = \eta'$.
Thus, we have
\begin{equation}\label{est:10}
\left|\CF^{-1}_{\xi'}\left[m^2_0(\lambda, \xi')|\xi'|B^2\CM(x_N)\CM(y_N))\right](x')\right| 
\leq \frac{C}{\left(x_N+y_N\right)^N}.
\end{equation}
From \eqref{est:9} and \eqref{est:10}, we deduce that
\begin{equation}
\left|\CF^{-1}_{\xi'}\left[m^2_0(\lambda, \xi')|\xi'|B^2\CM(x_N)\CM(y_N))\right](x')\right| 
\leq \frac{C}{\left(|x'|+x_N+y_N\right)^N}.
\end{equation}
In a similar way as in \eqref{est:6}, we arrive at
\begin{equation}
\|L^{(2)}_1(\lambda)f\|_{L_q(\BR^{N-1})} \leq C\int^\infty_0 
\frac{\|f(\,\cdot\,,  y_N)\|_{L_q(\BR^{N-1})}}{x_N+y_N}\,dy_N,
\end{equation}
which together with Proposition \ref{lem:2} implies that $L^{(2)}_1$ is a bounded linear operator on $L_q(\HS)$. The proof is complete.
 \end{proof}
\begin{proof}[Proof of Proposition \ref{prop:3}]
	Here, we only consider the estimate for $ P_1(\lambda)f $ and $ P_5(\lambda)f $, since the others may be proved in a similar way.
	First, by Taylor formula, we obtain
	$$\CM(x_N) = x_N\int^1_0 e^{-(A+\theta(B-A))x_N}\,d\theta.$$
	Thus, we have
	$$\pd_\lambda\CM(x_N) = -x_N^2\int^1_0 (\pd_\lambda A +
(\pd_\lambda B - \pd_\lambda A)\theta)
	e^{-(A+\theta(B-A))x_N}\,d\theta.$$
	Since we know that $\pd_\lambda B = 1/2\alpha B$, $\pd_\lambda A = p'(\lambda)/2A$, and $|p'(\lambda)| \leq C$ for $\lambda \in \Lambda_{\epsilon, \lambda_0}$, 
using Lemma \ref{halfsymbol:3.1} with $s=-1$, we see that for any $\delta' \in \BN_0^{N-1}$
\begin{equation}\begin{aligned}\label{est.3.6.1}
|D_{\xi'}^{\delta'}(B^3\pd_\lambda \CM(x_N))| &\leq C_{\delta'}x_N^2(|\lambda|^{1/2}+|\xi'|)^{2-|\delta'|}
e^{-2c(|\lambda|^{1/2}+|\xi'|)x_N} \\
&\leq C_{\delta'}(|\lambda|^{-1/2}+|\xi'|)^{-|\delta'|}e^{-c(|\lambda|^{1/2}+|\xi'|)x_N}
\end{aligned}\end{equation}
with some positive constants $C_{\delta'}$ and $c$.  Here and in the following
 $c$ denotes a constant independent of $\delta'$. 
Writing 
\begin{align*}
B\pd_\lambda(B^3 \CM(x_N)\CM(y_N)) 
&= 3 B^3(\pd_\lambda B) \CM(x_N) \CM(y_N) + B^4(\pd_\lambda \CM(x_N))\CM(y_N) \\
&+ B^4\CM(x_N)(\pd_\lambda \CM(y_N)),
\end{align*}
 and using Lemma \ref{lem:3.4} and \eqref{est.3.6.1},  
we see that for any $\delta' \in \BN_0^{N-1}$ 
	\begin{align*}
		\left|D^{\delta'}_{\xi'}\left(B\pd_\lambda(B^3 \CM(x_N)\CM(y_N))\right)\right| 
		\leq C_{\delta'}(|\lambda|+|\xi'|)^{-|\delta'|}e^{-c(|\lambda|^{1/2}+|\xi'|)(x_N+y_N)}.
	\end{align*}
By the similar method as Proposition \ref{prop:2}, we can derive $\|P_1(\lambda)f\|_{L_q(\HS)}
\leq C\|f\|_{L_q(\HS)}$.\par 
As for $P_5(\lambda)f$,  writing $\pd_\lambda e^{-Ax_N} = -(\pd_\lambda A)x_N e^{-Ax_N}$,
using Lemma \ref{lem:3.4}, we see that for any $\delta' \in \BN_0^{N-1}$  
\begin{equation}\label{est.3.6.2}
|D_{\xi'}^{\delta'}B^2(\pd_\lambda e^{-A_N})| \leq Ce^{-c(|\lambda|^{1/2}+|\xi'|)x_N}
\end{equation}
with some positive constants $C_{\delta'}$ and $c$. 
 Writing
\begin{align*}
B\pd_\lambda(B^2 e^{-Ax_N}\CM(y_N)) &= 2B^2(\pd_\lambda B)e^{-Ax_N}\CM(y_N) 
-B^3(\pd_\lambda e^{-AxN}) \CM(x_N)\\
& +B^3 e^{-Ax_N}(\pd_\lambda \CM(y_N))\end{align*}
using \eqref{est.3.6.1}, \eqref{est.3.6.2} and Lemma \ref{lem:3.4}, we see that 
 for any $\delta' \in \BN_0^{N-1}$,
$$\left|D_{\xi'}^{\delta'}\left(B\pd_\lambda (B^2e^{-Ax_N}\CM(y_N))\right)\right| \\
	\leq C_{\delta'}\left(|\lambda|^{1/2}+|\xi'|\right)^{-|\delta'|}e^{-c(|\lambda|^{1/2}+|\xi'|)x_N}
$$
with some positive constants $C_{\delta'}$ and $c$. 
By the similar method as Proposition \ref{prop:2} again, we can derive 
$\|P_5(\lambda)f\|_{L_q(\HS)} \leq C\|f\|_{L_q(\HS)}$. The proof is completed.
\end{proof}
\section{Solution formula} \label{sec:4}
In this section, we shall discuss solution formulas of equations \eqref{Eq:Resolvent}.
From the first equation in \eqref{Eq:Resolvent}, we have $\rho = \lambda^{-1}(f-\gamma\dv\bu)$, 
and inserting this formula into the second equation in \eqref{Eq:Resolvent} implies the complex 
Lam\'e equations
\begin{equation}\label{cl:0}
\lambda\bu-\alpha\Delta\bu - \eta_\lambda\nabla\dv\bu = \bg - \gamma\lambda^{-1}\nabla f
\quad\text{in $\HS$}, \quad \bu|_{\pd\HS} = 0.
\end{equation}
Here, we have set $\eta_\lambda = \beta + \gamma^2\lambda^{-1}.$  \par
If we find a solution $\bu$ of equations \eqref{cl:0} and if we set $\rho = \lambda^{-1}(f-\gamma\dv\bu)$,
then $\rho$ and $\bu$ are solutions of equations \eqref{Eq:Resolvent}.  Thus, in this section,    
we shall drive solution formulas of  the complex Lam\'e equations
\begin{equation}\label{cl:1} 
\lambda\bu-\alpha\Delta\bu - \eta_\lambda\nabla\dv\bu = \bg
\quad\text{in $\HS$}, \quad \bu|_{\pd\HS} = 0.
\end{equation}
\subsection{Whole space case}
For $\epsilon \in (0, \pi/2)$ and $\lambda_0 > 0$ let $\lambda\in \Lambda_{\epsilon,\lambda_0}$ be
the resolvent parameter, where $\lambda_0$ is assumed to be sufficiently large if necessary.
In this subsection, we derive the representation of the solution formula for the following model problem in $\BR^N$:
\begin{equation}\label{cl:2}
\lambda\bu - \alpha\Delta\bu - \eta_\lambda \nabla\dv\bu 
 =\bg\quad \text{in $\BR^N$},
\end{equation}
where  $\bg \in B^{s}_{q,1}(\BR^N)^N$,
with $1<q<\infty$ and $-1+1/q<s<1/q$. Applying the divergence to  equation \eqref{cl:2} yields
\begin{equation}\label{cl:3}
\lambda\dv \bu - (\alpha+\eta_\lambda)\Delta\dv\bu = \dv \bg \quad\text{in $\BR^N$}.
\end{equation}
Applying Fourier transform to \eqref{cl:3} yields
$$(\lambda + (\alpha+\eta_\lambda)|\xi|^2)\CF[\dv\bu](\xi) = i\xi\cdot\CF[\bg](\xi).
$$
Applying Fourier transform to equation \eqref{cl:2} yields
$$(\lambda+\alpha|\xi|^2)\hat\bu  -\eta_\lambda i\xi\CF[\dv\bu] = \hat \bg.
$$
Thus, 
\begin{align*}
\hat\bu(\xi) & = (\lambda+\alpha|\xi|^2)^{-1}(\hat \bg(\xi) +\eta_\lambda i\xi
(\lambda+(\alpha+\eta_\lambda)|\xi|^2)^{-1}i\xi\cdot\hat\bg(\xi))\\
& = \frac{\hat \bg(\xi)}{\lambda+\alpha|\xi|^2}
+ \frac{\eta_\lambda (i\xi\otimes i\xi)\hat \bg(\xi)}
{(\lambda+\alpha|\xi|^2)(\lambda+(\alpha+\eta_\lambda)|\xi|^2)} \\
& = \frac{1}{\alpha}\frac{\hat \bg(\xi)}{\alpha^{-1}\lambda + |\xi|^2}
+ \frac{\eta_\lambda}{\alpha(\alpha+\eta_\lambda)}\frac{(i\xi\otimes i\xi)\hat \bg(\xi)}
{(\alpha^{-1}\lambda+|\xi|^2)(p(\lambda) +|\xi|^2)},
\end{align*}
where we have set 
$$p(\lambda) = \frac{\lambda}{\alpha + \eta_\lambda} = \frac{\lambda^2}{(\alpha+\beta)\lambda
+\gamma^2}.
$$
Applying the Fourier inverse transform implies that 
$$\begin{aligned}
\bu &= \CF^{-1}\Biggl[\frac{\hat \bg(\xi)}{\lambda + \alpha|\xi|^2}\Biggr]
-\frac{\beta\lambda+\gamma^2}{(\alpha+\beta)\lambda+\gamma^2}
\CF^{-1}\Biggl[\frac{(\xi\otimes\xi)\hat \bg(\xi)}{(\lambda+\alpha|\xi|^2)
	(p(\lambda)+|\xi|^2)}\Biggr]. 
\end{aligned}$$
Thus, for the later use, we define an operator $\CS_0(\lambda)$ by 
\begin{equation}\label{solform:4.1}\begin{aligned}
\CS^0(\lambda)\bg & = 
\CF^{-1}\Biggl[\frac{\hat \bg(\xi)}{\lambda + \alpha|\xi|^2}\Biggr]
-\frac{\beta\lambda+\gamma^2}{(\alpha+\beta)\lambda+\gamma^2}
\CF^{-1}\Biggl[\frac{(\xi\otimes\xi)\hat \bg(\xi)}{(\lambda+\alpha|\xi|^2)
	(p(\lambda)+|\xi|^2)}\Biggr], 
\end{aligned}\end{equation}
which is a solution operator of equation \eqref{cl:2}. 

\subsection{Half space case}
Let $\epsilon \in (0, \pi/2)$ and $\nu_0 > 0$.  Let $\gamma_0 > 0$ be a large number such that
$\Sigma_\epsilon + \gamma_0 \subset K_\epsilon \cap \Sigma_\epsilon \cap \{\lambda \in \BC \mid
|\lambda| \geq \nu_0\}$. 
In this subsection, we derive the representation of the solution 
formula for equations \eqref{cl:1}.  To this end, 
 we extend $\bg=(g_1,\cdots ,g_N)$ by
$$g_j^e(x) = \begin{cases} g_j(x) &\quad \text{for $x_N > 0$}, \\  g_j(x',- x_N)& \quad \text{for $x_N < 0$}, 
\end{cases}\quad
g_N^o(x) = \begin{cases} g_{N}(x) &\quad \text{for $x_N > 0$}, \\  -g_{N}(x',- x_N)& \quad \text{for $x_N < 0$}.
\end{cases}
$$
Here and in the sequel  $j$ and $k$ run from $1$ through $N-1$. 
We now set  $\bG:=(g^e_1,\cdots, g^e_{N-1}, g^o_N)$.  Let $\bu$ be a solution of equations \eqref{cl:1} and let
$\bw = \bu - \CS^0(\lambda)\bG$, and then $\bw$ should satisfy the equations
\begin{equation}\label{cl:4}
\lambda\bw- \alpha \Delta \bw - \eta_\lambda\nabla\dv\bw = 0\quad\text{in $\HS$}, \quad
\bw|_{\pd\HS} = -\CS^0(\lambda)\bG|_{\pd\HS}.
\end{equation}
In view of \eqref{solform:4.1}, we may have 
\begin{equation}\label{sol:whole}\begin{aligned}
\CS^0(\lambda)\bG&= \CF^{-1}\Biggl[\frac{\hat \bG(\xi)}{\lambda + \alpha|\xi|^2}\Biggr]
-\frac{\beta\lambda+\gamma^2}{(\alpha+\beta)\lambda+\gamma^2}
\CF^{-1}\Biggl[\frac{(\xi\otimes\xi)\hat \bG(\xi)}{(\lambda+\alpha|\xi|^2)
(p(\lambda)+|\xi|^2)}\Biggr]. 
\end{aligned}\end{equation}	
Let $\bw = (w_1, \ldots, w_N)$, and  we shall investigate the formula of
the partial Fourier transform $\CF'[w_j](\xi', x_N)$ of $w_j$.  Applying the partial 
Fourier transform $\CF'$ to equations \eqref{cl:4}, we have the ordinary differential equations
in $x_N$ variable, which reads as
\begin{equation}\label{cl:5}\left\{\begin{aligned}
&(\lambda + \alpha|\xi'|^2)\CF'[w_j]( x_N) - \alpha \pd_N^2\CF'[w_j]( x_N)\\
&\phantom{(\lambda}- \eta_\lambda i\xi_j (i\xi'\cdot\CF'[\bw'](x_N) + \pd_N\CF'[w_N]( x_N)) = 0, &&\quad&\text
{for $x_N>0$}, \\
&(\lambda + \alpha|\xi'|^2)\CF'[w_N]( x_N) - \alpha \pd_N^2 \CF'[w_N]( x_N)\\
&\phantom{(\lambda}- \eta_\lambda \pd_N (i\xi'\cdot\CF'[\bw'](x_N) + \pd_N\CF'[w_N](x_N)) = 0,  
&&\quad&\text{for $x_N>0$}, \\
	&\CF'[\bw](0) = -\CF'[\CS^0(\lambda)\bG]( 0).
\end{aligned}\right.\end{equation}
Here, we have set $\CF'[f](\xi', x_N) = \CF'[f](x_N)$, $i\xi'\cdot\CF'[\bw'](x_N)
= \sum_{j=1}^{N-1}i\xi_j\CF'[w_j](x_N)$.\par 
To obtain $\CF'[w_j](\xi', x_N)$, 
first we  derive the representation of $\CF'[\CS^0(\lambda)\bG](0)$. Notice that 
\begin{align*}
&\CF'[\CS^0(\lambda)\bG](0) \\
&= \frac{1}{\alpha}\frac{1}{2\pi}\int_{\BR}\frac{\hat \bG(\xi)}{\lambda\alpha^{-1} + |\xi|^2}\,d\xi_N
-\frac{\beta\lambda+\gamma^2}{\alpha((\alpha+\beta)\lambda + \gamma^2)}
\frac{1}{2\pi}\int_{\BR} \frac{(\xi\otimes\xi)\hat \bG(\xi)}{(\lambda\alpha^{-1}+|\xi|^2)
(p(\lambda)+|\xi|^2)}\,d\xi_N.
\end{align*}
Notice  that $\alpha^{-1}\lambda+|\xi|^2 = (\xi_N + iB)(\xi_N-iB)$ and 
$p(\lambda) + |\xi|^2 =(\xi_N + iA)(\xi_N-iA)$. By the residue theorem in the
theory of one complex variable, we have 
\allowdisplaybreaks
\begin{align*}
h^{(1)}_j :&= \frac{1}{2\pi}\int_{\BR}\frac{\hat g_j^e(\xi)}{\alpha^{-1}\lambda+|\xi|^2}\,d\xi_N \\
&=
i\int^\infty_0 \CF'[g_j](\xi', y_N)
\Bigl( \frac{1}{2\pi i}\int_{\BR}\frac{ e^{-iy_N\xi_N} + e^{iy_N\xi_N}}{(\xi_N+iB)(\xi_N-iB)}\,d\xi_N
\Bigr)\,dy_N\\ 
&= i \int^\infty_0\CF'[g_j](\xi', y_N)\Bigl(-\frac{e^{-y_NB}}{-2i B} + \frac{e^{-y_NB}}{2iB}\Bigr)\,dy_N
\\
& = \int^\infty_0\frac{e^{-y_NB}}{B}\CF'[g_j](\xi', y_N)\,dy_N;\\
h^{(1)}_N : &= \frac{1}{2\pi}\int_{\BR}\frac{\hat g_N^o(\xi)}{\alpha^{-1}\lambda +|\xi|^2}\,d\xi_N \\
&= 
i\int^\infty_0 \CF'[g_N](\xi', y_N)
\Bigl( \frac{1}{2\pi i}\int_{\BR}\frac{ e^{-iy_N\xi_N} - e^{iy_N\xi_N}}{(\xi_N+iB)(\xi_N-iB)}\,d\xi_N
\Bigr)\,dy_N\\
&= i \int^\infty_0\CF'[g_N](\xi', y_N)\Bigl(-\frac{e^{-y_NB}}{-2i B} - \frac{e^{-y_NB}}{2iB}\Bigr)\,dy_N
=0.
\end{align*}
Likewise, for each $1\leq j,k\leq N-1$, we also have
\allowdisplaybreaks 
\begin{align*}
h^{(2)}_{jk}:&= \frac{1}{2\pi}\int_{\BR}\frac{\xi_j\xi_k\hat g^e_k(\xi)}{(\alpha^{-1}\lambda+|\xi|^2)
(p(\lambda)+|\xi|^2)}\,d\xi_N\\
&=i\int^\infty_0 \xi_j\xi_k\CF'[g_k](\xi', y_N)\Bigl(\frac{1}{2\pi i}
\int_{\BR} \frac{ e^{-iy_N\xi_N} + e^{iy_N\xi_N}}{(A^2+\xi_N^2)(B^2+\xi_N^2)}\,d\xi_N\Bigr)\,dy_N\\
&= i\int^\infty_0\CF'[g_k](\xi', y_N)\xi_j\xi_k\Bigl(\frac{e^{-y_NB}}{(A^2-B^2)iB} + 
\frac{e^{-y_NA}}{(B^2-A^2)iA}\Bigr)\,dy_N \\
&= \int^\infty_0\Bigl(\frac{e^{-Ay_N}}{A}-\frac{e^{-By_N}}{B}\Bigr)\frac{\xi_j\xi_k}{B^2-A^2}
\CF'[g_k](\xi', y_N)\,dy_N;
\end{align*}
while for each $1\leq j\leq N-1$, we obtain
\begin{align*}
h^{(2)}_{jN} & =\frac{1}{2\pi}\int_{\BR}\frac{\xi_j\xi_N\hat g^o_N(\xi)}{(\alpha^{-1}\lambda+|\xi|^2)
(p(\lambda)+|\xi|^2)}\, d\xi_N\\
&=i\int^\infty_0 \CF'[g_N](\xi', y_N)\xi_j\Bigl(\frac{1}{2\pi i}
\int_{\BR} \frac{ \xi_N(e^{-iy_N\xi_N}-e^{iy_N\xi_N})}{(A^2+\xi_N^2)(B^2+\xi_N^2)}\,d\xi_N\Bigr)\,dy_N\\
& =i\int^\infty_0 \CF'[g_N](\xi', y_N)\xi_j\Bigl(
-\frac{e^{-y_NB}(-iB)}{(A^2-B^2)(-2iB)}-\frac{e^{-By_N}(iB)}{(A^2-B^2)(2iB)} \\
&\hskip1cm- \frac{e^{-Ay_N}(-iA)}{(B^2-A^2)(-2iA)}-\frac{e^{-Ay_N}(iA)}{(B^2-A^2)(2iA)}\Bigr)\,dy_N\\
& = i\int^\infty_0 \frac{e^{-Ay_N} - e^{-By_N}}{A^2-B^2} \xi_j\CF'[g_N](\xi', y_N)\,dy_N.
\end{align*}
In addition, 
for the case $1\leq j\leq N-1$ we see that
\begin{align*}
h^{(2)}_{Nk} & =\frac{1}{2\pi}\int_{\BR}\frac{\xi_N\xi_k\hat g^e_k(\xi)}{(\alpha^{-1}\lambda+|\xi|^2)
(p(\lambda)+|\xi|^2)}\,d\xi_N\\
&=i\int^\infty_0 \CF'[g_k](\xi', y_N)\xi_k\Bigl(\frac{1}{2\pi i}
\int_{\BR} \frac{ \xi_N(e^{-iy_N\xi_N}+e^{iy_N\xi_N})}{(A^2+\xi_N^2)(B^2+\xi_N^2)}\,d\xi_N\Bigr)\,dy_N\\
& =i\int^\infty_0 \CF'[g_k](\xi', y_N)\xi_k\Bigl(
-\frac{e^{-y_NB}(-iB)}{(A^2-B^2)(-2iB)}+\frac{e^{-By_N}(iB)}{(A^2-B^2)(2iB)} \\
&\hskip1cm- \frac{e^{-Ay_N}(-iA)}{(B^2-A^2)(-2iA)}+\frac{e^{-Ay_N}(iA)}{(B^2-A^2)(2iA)}\Bigr)\,dy_N\\
&=0,
\end{align*}
as well as 
\begin{align*}
h^{(2)}_{NN} & = \frac{1}{2\pi}\int_{\BR}\frac{\xi_N^2\hat g^o_N(\xi)}{(\alpha^{-1}\lambda+|\xi|^2)
(p(\lambda)+|\xi|^2)}\, d\xi_N\\
&= i\int^\infty_0\CF'[g_N](\xi', y_N)\Bigl(\frac{1}{2\pi i}\int_{\BR}
\frac{\xi_N^2(e^{-iy_N\xi_N} - e^{iy_N\xi_N})}{(\xi_N^2+A^2)(\xi_N^2+B^2)}\Bigr)\,dy_N\\
& =i\int^\infty_0 \CF'[g_N](\xi', y_N)\Bigl(
-\frac{e^{-y_NB}(-iB)^2}{(A^2-B^2)(-2iB)}-\frac{e^{-By_N}(iB)^2}{(A^2-B^2)(2iB)} \\
&\hskip1cm- \frac{e^{-Ay_N}(-iA)^2}{(B^2-A^2)(-2iA)}-\frac{e^{-Ay_N}(iA)^2}{(B^2-A^2)(2iA)}\Bigr)\,dy_N\\
&=0.
\end{align*}
We write 
\begin{align*}
\Bigl(\frac{e^{-Ay_N}}{A}- \frac{e^{-By_N}}{B}\Bigr)\frac{1}{B^2-A^2} 
	= -\frac{\CM(y_N)}{A(A+B)} + \frac{e^{-By_N}}{AB(A+B)}, \quad 
\frac{e^{-Ay_N}-e^{-By_N}}{A^2-B^2} = \frac{\CM(y_N)}{A+B}.
\end{align*}
Let $h_j$ be the $j$-th component of $-\CF'[\CS^0(\lambda)\bG](0)$, and then we have
\begin{align*}
h_j &:=-\frac{1}{\alpha}h^{(1)}_j +\frac{\beta\lambda+\gamma^2}{\alpha((\alpha+\beta)\lambda+\gamma^2)}
\sum_{k=1}^Nh^{(2)}_{jk}\\
& =- \frac{1}{\alpha B}\int_{0}^{\infty}\CF'[g_j](\xi', y_N)e^{-y_NB}\,dy_N \\
&\quad -\frac{\beta\lambda+\gamma^2}{\alpha((\alpha+\beta)\lambda+\gamma^2)}
\int^\infty_0\CM(y_N)\frac{\xi_j}{A+B}
(\sum_{k=1}^{N-1}\frac{\xi_k}{A}\CF'[g_k](\xi', y_N)- i\CF'[g_N](\xi', y_N))\,dy_N \\
&\quad + \frac{\beta\lambda+\gamma^2}{\alpha((\alpha+\beta)\lambda+\gamma^2)}
\int^\infty_0e^{-By_N}\frac{i\xi_j}{B(A+B)}
\sum_{k=1}^{N-1}\frac{\xi_k}{A}\CF'[g_k](\xi', y_N)\,dy_N,
\end{align*}
for $j=1, \ldots, N-1$ and $h_N=0$. 
According to \cite[(4.9)]{ES13}, we have
\begin{equation}\label{4.9}
\CF'[w_j](\xi', x_N) = h_je^{-Bx_N} - \frac{i\xi_j\eta_\lambda}{K}\CM(x_N)i\xi'\cdot h', 
\quad \CF'[w_N](\xi', x_N) = \frac{A\eta_\lambda}{K}\CM(x_N)i\xi'\cdot h',
\end{equation}
where $K = (\alpha+\eta_\lambda)A + \alpha B$ and $i\xi'\cdot h' = \sum_{j=1}^{N-1}i\xi_j h_j$. \par
We calculate the right hand side. For notational simplicity, we write 
$\CF'[g_j](\xi', y_N) = \CF'[g_j]$, namely, $(\xi', y_N)$ is omitted. We have
\begin{align*}
h_j & = \frac{-1}{\alpha }\int_{0}^{\infty}\frac{e^{-y_NB}}{B}\CF'[g_j]\,dy_N\\
&- \frac{\beta\lambda+\gamma^2}{\alpha((\alpha+\beta)\lambda + \gamma^2)}
\int^\infty_0\CM(y_N)(\sum_{k=1}^{N-1}\frac{\xi_j\xi_k}{(A+B)A}\CF'[g_k] 
-\frac{i\xi_j}{A+B}\CF'[g_N])\,dy_N \\
&+  \frac{\beta\lambda+\gamma^2}{\alpha((\alpha+\beta)\lambda + \gamma^2)}\int^\infty_0
\frac{e^{-By_N}}{B}
\sum_{k=1}^{N-1}\frac{i\xi_j \xi_k}{(A+B)A}\CF'[g_k]\,dy_N; \\
i\xi'\cdot h' & = \frac{-1}{\alpha}\int_{0}^{\infty}\frac{e^{-By_N}}{B}\, i\xi'\cdot\CF'[\bg']e^{-y_NB}\,dy_N\\
&+ \frac{\beta\lambda+\gamma^2}{\alpha((\alpha+\beta)\lambda + \gamma^2)}
\int^\infty_0\CM(y_N)(\sum_{k=1}^{N-1}\frac{|\xi'|^2 \xi_k}{i(A+B)A}\CF'[g_k] 
-\frac{|\xi'|^2}{A+B}\CF'[g_N])\,dy_N \\
&- \frac{\beta\lambda+\gamma^2}{\alpha((\alpha+\beta)\lambda + \gamma^2)}\int^\infty_0\frac{e^{-By_N}}{B}
\sum_{k=1}^{N-1}\frac{|\xi'|^2 \xi_k}{(A+B)A}\CF'[g_k]\,dy_N. 
\end{align*}
Thus,  we have
\allowdisplaybreaks
\begin{align}
&\CF'[w_j](\xi', x_N)  =-\int_{0}^{\infty}Be^{-(x_N+y_N)B} \frac{1}{\alpha B^2}\CF'[g_j]\,dy_N
\nonumber\\
&- \int^\infty_0B^2e^{-Bx_N}\CM(y_N)\frac{\beta\lambda+\gamma^2}{\alpha((\alpha+\beta)\lambda + \gamma^2)}
(\sum_{k=1}^{N-1}\frac{\xi_j\xi_k}{(A+B)AB^2}\CF'[g_k] 
-\frac{i\xi_j}{(A+B)B^2}\CF'[g_N])\,dy_N \nonumber\\
&+ \int^\infty_0Be^{-B(x_N+y_N)}
\frac{\beta\lambda+\gamma^2}{\alpha((\alpha+\beta)\lambda + \gamma^2)}
\sum_{k=1}^{N-1}\frac{i\xi_j \xi_k}{(A+B)AB^2}\CF'[g_k]\,dy_N \\
& +\int^\infty_0B^2\CM(x_N)e^{-By_N}\frac{i\xi_j\eta_\lambda}{K} \frac{1}{\alpha B^3}
i\xi'\cdot \CF'[\bg']\,dy_N\nonumber \\
&- \int^\infty_0B^3\CM(x_N)\CM(y_N)
\frac{\xi_j\eta_\lambda}{K}\frac{\beta\lambda+\gamma^2}{\alpha((\alpha+\beta)\lambda + \gamma^2)}
(\sum_{k=1}^{N-1}\frac{|\xi'|^2 \xi_k}{(A+B)AB^3}\CF'[g_k] \nonumber\\
&\hskip9cm 
-\frac{i|\xi'|^2}{(A+B)B^3}\CF'[g_N])\,dy_N \nonumber\\
&+\int^\infty_0 B^2\CM(x_N)e^{-By_N}
\frac{i\xi_j\eta_\lambda}{K}\frac{\beta\lambda+\gamma^2}{\alpha((\alpha+\beta)\lambda + \gamma^2)}
\sum_{k=1}^{N-1}\frac{|\xi'|^2 \xi_k}{(A+B)AB^3}\CF'[g_k]\,dy_N, \nonumber\\
&\CF'[w_N ](\xi', x_N) =
 -\int^\infty_0B^2\CM(x_N)e^{-By_N}\frac{A\eta_\lambda}{K} \frac{1}{\alpha B^3}
i\xi'\cdot \CF'[\bg']\,dy_N \nonumber\\
&+\int^\infty_0B^3\CM(x_N)\CM(y_N)
\frac{A\eta_\lambda}{K}\frac{\beta\lambda+\gamma^2}{\alpha((\alpha+\beta)\lambda + \gamma^2)}
(\sum_{k=1}^{N-1}\frac{|\xi'|^2 \xi_k}{(A+B)AB^3}\CF'[g_k]\nonumber \\
&\hskip9cm 
-\frac{|\xi'|^2}{(A+B)B^3}\CF'[g_N])\,dy_N\nonumber\\
&-\int^\infty_0 B^2\CM(x_N)e^{-By_N}
\frac{A\eta_\lambda}{K}\frac{\beta\lambda+\gamma^2}{\alpha((\alpha+\beta)\lambda + \gamma^2)}
\sum_{k=1}^{N-1}\frac{|\xi'|^2 \xi_k}{(A+B)AB^3}\CF'[g_k]\,dy_N.
\label{sol:form5}
\end{align}

\section{Estimates of solution operators  in the whole space} \label{sec:5}
\label{sec-5}
In this section, we shall estimate the solution operator $\CS^0(\lambda)$ defined 
in \eqref{solform:4.1}.  To this end, we use the Fourier multiplier theorem of Mihlin-H\"ormander
type \cite{Min, Hor}. Let $m(\xi)$ be a complex-valued function defined on 
$\BR^N\setminus\{0\}$ which satisfies the multiplier conditions:
\begin{equation}\label{multiplier:5.1}
|D^{\delta}_{\xi} m(\xi)| \leq C_{\delta}|\xi|^{-|\delta|}
\end{equation}
for any multi-index $\delta \in \BN_0^N$ with some constant $C_\delta$ depending on
$\delta$. We say that $m(\xi)$ is a multiplier. Then, the Fourier multiplier operator
with kernel function $m(\xi)$ is defined by
\begin{equation}\label{multiplier:5.2}
T_m f = \CF^{-1}_\xi[m(\xi)\CF[f](\xi)] = \frac{1}{(2\pi)^N}\int_{\BR^N}
e^{ix\cdot \xi}m(\xi)\CF[f](\xi)\,d\xi \quad\text{for $f \in \CS(\BR^N)$}. 
\end{equation}
Then, we have the following theorem called  the Fourier multiplier theorem.
\begin{thm}\label{FMT} Let $1 < q < \infty$ and $m(\xi)$ be a multiplier.  Then,
the Fourier multiplier $T_m$ is an $L_q(\BR^N)$ bounded operator, that is
there exists a constant depending on $q$ and $N$ such that
$$\|T_m f\|_{L_q(\BR^N)} \leq C(\max_{|\delta| \leq [N/2]+1} C_\delta)\|f\|_{L_q(\BR^N)}.
$$
Here, $[N/2]$ denotes the integer part of $N/2$. \par
$T_m$ is extended uniquely to an operator on $L_q(\BR^N)$, which is also written
by $T_m$.
\end{thm}
To estimate solution operators, we use the following lemma.
\begin{lem}\label{lem:5.1}
 Let $1 < q < \infty$, $1 \leq r \leq \infty$ and $s$, $\sigma$ 
be two real numbers. Let $m(\xi)$ be a complex valued $C^\infty$ 
function defined on 
$\BR^N\setminus\{0\}$ satisfying \eqref{multiplier:5.1} and
let $T_m$ be an operator defined by \eqref{multiplier:5.2}.  Then, 
for any $f \in B^s_{q,r}(\BR^N)$, there holds 
\begin{equation}\label{multiplier:5.3}
\|T_mf\|_{B^s_{q,r}(\BR^N)} \leq C_{s,q,r}(\max_{|\delta| \leq [N/2]+1}C_\delta)
\|f\|_{B^s_{q,r}(\BR^N)}.
\end{equation}
Here, $C_\alpha$ are constants appearing in \eqref{multiplier:5.1}. 

Moreover, let 
$$\langle D \rangle^\sigma f = \CF^{-1}[(1+|\xi|^2)^{\frac{\sigma}{2}}\CF[f](\xi)]
=\frac{1}{(2\pi)^N} \int_{\BR^N}  e^{ix\cdot \xi}(1+|\xi|^2)^{\frac{\sigma}{2}}\CF[f](\xi)\,d\xi.
$$
Then, 
\begin{equation}\label{multiplier:5.4}
\|\langle D \rangle^\sigma f\|_{B^s_{q,r}(\BR^N)} \leq C\|f\|_{B^{s+\sigma}_{q,r}(\BR^N)}.
\end{equation}
\end{lem}
\begin{proof}
Let $\psi$, $\phi$, and $\phi_k$ be functions given in \eqref{little:1}. Since $ m$ satisfies the condition 
\eqref{multiplier:5.1}, 
by Theorem \ref{FMT}  we have 
\begin{align*}
\|\psi * T_mf\|_{L_q(\BR^N)}
&= \|\CF^{-1}_\xi[m(\xi)
\psi(\xi)\CF[f](\xi)]\|_{L_q(\BR^N)} 
\leq CD\|\CF^{-1}_\xi[\psi(\xi)\CF[f](\xi)]\|_{L_q(\BR^N)},
\end{align*}
where $D=\max_{|\delta| \leq [N/2] + 1} C_\delta$. 
Likewise, we have
\begin{align*}
\|\phi_k * T_mf\|_{L_q(\BR^N)}
&= \|\CF^{-1}_\xi[m(\xi)
\phi_k(\xi)\CF[f](\xi)]\|_{L_q(\BR^N)} 
\leq CD\|\CF^{-1}_\xi[\phi_k(\xi)\CF[f](\xi)]\|_{L_q(\BR^N)}.
\end{align*}
Thus, from the definition \eqref{pld:2} we have
$$
\|T_mf\|_{B^s_{q,r}(\BR^N)} 
\leq CD\|f\|_{B^s_{q,r}(\BR^N)}.
$$
\par
To prove \eqref{multiplier:5.4}, we choose two $C^\infty_0(\BR^N)$
functions $\tilde\phi$ and $\tilde\psi$ such that 
$\tilde\phi(\xi) = 1$ on ${\rm supp}\, \psi$, $\tilde\psi(\xi)=1$ on
${\rm supp}\, \phi$, and $\tilde\phi$ vanishes outside of 
$\{\xi \in \BR^N \mid 1/4 \leq |\xi| \leq 4\}$. We see that 
\begin{align*}
|D^{\delta}_{\xi} ((1+|\xi|^2)^{\frac{\sigma}{2}}\tilde \psi)| &\leq C_\delta|\xi|^{-|\delta|},
\\
|D^{\delta}_{\xi} (2^{-\sigma k}((1+|\xi|^2)^{\frac{\sigma}{2}}\tilde\phi(2^{-k}\xi))
| &\leq C_\delta|\xi|^{-|\delta|}
\end{align*}
for any multi-index $\delta \in \BN_0^N$.  By Theorem \ref{FMT}, we have
\begin{align*}
\|\psi*\langle D \rangle^\sigma f\|_{L_q(\BR^N)} &= \|\CF^{-1}[\CF[\psi](\xi)(1+|\xi|^2)^{\frac{\sigma}{2}}
\tilde\psi(\xi)\CF[f](\xi)]\|_{L_q(\BR^N)} \\
&\leq C_\sigma \|\CF^{-1}[\CF[\psi](\xi)\CF[f](\xi)]\|_{L_q(\BR^N)}
=  C_\sigma \|\psi*f\|_{L_q(\BR^N)}.
\end{align*}
Likewise, 
\begin{align*}
2^{sk}\|\phi_k*\langle D \rangle^\sigma f\|_{L_q(\BR^N)} 
&=2^{sk} 2^{\sigma k}\|\CF^{-1}[\CF[\phi_k](\xi)(1+|\xi|^2)^{\frac{\sigma}{2}}
2^{-\sigma k}\tilde\phi(2^{-k}\xi)\CF[f](\xi)]\|_{L_q(\BR^N)} \\
&\leq C_\sigma 2^{(s+\sigma)k}\|\CF^{-1}[\CF[\phi_k](\xi)\CF[f](\xi)]\|_{L_q(\BR^N)}
=  C_\sigma 2^{(s+\sigma)k}\|\phi_k*f\|_{L_q(\BR^N)}.
\end{align*}
Thus, from the definition \eqref{pld:2} we have
$$
\|\langle D \rangle^\sigma f\|_{B^s_{q,r}(\BR^N)} 
\leq C_\sigma \|f\|_{B^{s+\sigma}_{q,r}(\BR^N)}.
$$
This completes the proof of Lemma \ref{lem:5.1}.
\end{proof}
Now, we shall prove the following theorem which is a main result of this section. 
\begin{thm}\label{thm:5.1} Let $1 < q <\infty$, $1 \leq r \leq \infty$, 
$-1+1/q < s < 1/q$, and $\epsilon \in (0, \pi/2)$. 
Let $\CS^0(\lambda)$ be the operator defined in \eqref{solform:4.1}.
Then, there exists a large constant $\omega_0 > 0$ such that 
for any $\lambda \in \Sigma_{\epsilon, \omega_0}$ and $\bg \in B^s_{q,r}(\BR^N)^N$,
 there hold
\begin{align}\label{thm:5.1.1}
\|(\lambda, \lambda^{\frac{1}{2}}\nabla, \nabla^2)\CS^0(\lambda)\bg\|_{B^s_{q,r}(\BR^N)}
&\leq C\|\bg\|_{B^s_{q,r}(\BR^N)}, \\
\label{thm:5.1.2}
\|(\lambda, \lambda^{\frac{1}{2}}\nabla, \nabla^2)\pd_\lambda \CS^0(\lambda)\bg\|_{B^s_{q,r}(\BR^N)}
&\leq C|\lambda|^{-1}\|\bg\|_{B^s_{q,r}(\BR^N)}.
\end{align}
\par
Moreover, let $\sigma>0$ be a small number such that $-1+1/q < s-\sigma < s < s+\sigma < 1/q$.
Then, there exist a large number $\omega_1 \geq\omega_0$ and 
two operators $\CT^0_1(\lambda)$ and $\CT^0_2(\lambda)$ which are holomorphic on
$\Lambda_{\epsilon, \omega_1}$ such that $\CS^0(\lambda) = \CT^0_1(\lambda)
+\CT^0_2(\lambda)$ and 
 for any $\bg \in C^\infty_0(\BR^N)$ and $\lambda \in \Lambda_{\epsilon, \omega_1}$, there hold
\begin{align}\label{thm:5.1.3}
\|(\lambda^{\frac{1}{2}}\nabla, \nabla^2)\CT^0_1(\lambda)\bg\|_{B^s_{q,r}(\BR^N)}
&\leq C|\lambda|^{-\frac{\sigma}{2}}\|\bg\|_{B^{s+\sigma}_{q,r}(\BR^N)}, \\
\label{thm:5.1.4}
\|(\lambda^{\frac{1}{2}}\nabla, \nabla^2)\pd_\lambda\CT^0_1(\lambda)\bg\|_{B^s_{q,r}(\BR^N)}
&\leq C|\lambda|^{-(1-\frac{\sigma}{2})}\|\bg\|_{B^{s-\sigma}_{q,r}(\BR^N)}
\end{align} as well as for any $\lambda \in \Lambda_{\epsilon, \omega_1}$ and 
$\bg \in B^s_{q,1}(\BR^N)$, there hold
\begin{align}
\label{thm:5.1.5}
\|(\lambda^{\frac{1}{2}}\nabla, \nabla^2)\CT^0_2(\lambda)\bg\|_{B^s_{q,r}(\BR^N)}
&\leq C|\lambda|^{-\frac{\sigma}{2}}\|\bg\|_{B^{s}_{q,r}(\BR^N)}, \\
\label{thm:5.1.6}
\|(\lambda^{\frac{1}{2}}\nabla, \nabla^2)\pd_\lambda\CT^0_2(\lambda)\bg\|_{B^s_{q,r}(\BR^N)}
&\leq C|\lambda|^{-(1-\frac{\sigma}{2})}\|\bg\|_{B^{s}_{q,r}(\BR^N)}.
\end{align}
\end{thm}
\begin{proof} To prove the theorem, we divide $\CS^0(\lambda)$ as 
\begin{equation}\label{solform:4.4} \CS^0(\lambda)=
\frac{1}{\alpha}\CS^0_1(\lambda) 
+  \frac{\eta_\lambda}{\alpha(\alpha+\eta_\lambda)}\CS^0_2(\lambda),
\end{equation}
where we have defined $\CS^0_j(\lambda)$ ($j=1,2$)
by
\begin{align*}
\CS^0_1(\lambda)\bg &= \frac{1}{\alpha}\CF^{-1}_\xi\Bigl[\frac{\CF[\bg](\xi)}
{\lambda\alpha^{-1}+|\xi|^2}\Bigr], \\
\CS^0_2(\lambda)\bg & =
\CF^{-1}_\xi\Bigl[\frac{(i\xi\otimes i\xi)\CF[\bg](\xi))}{(\lambda\alpha^{-1}+|\xi|^2)
(p(\lambda) + |\xi|^2)}\Bigr].
\end{align*}
By Lemma \ref{lem:3.2} with $s=-1$ we have for any multi-index $\delta \in \BN_0^N$ 
there exists a constant $C_\delta$ such that 
\begin{align}
\Bigl|D^{\delta}_{\xi} \frac{(\lambda, \lambda^{\frac{1}{2}}i\xi, (i\xi)^2)}
{\lambda\alpha^{-1}+|\xi|^2}\Bigr|
\leq C_\delta |\xi|^{-|\delta|},\label{pr:3.0} \\
\Bigl|D^{\delta}_\xi \frac{(i\xi\otimes i\xi)(\lambda, \lambda^{\frac{1}{2}}i\xi, (i\xi)^2)}
{(p(\lambda)+|\xi|^2)^{-1}(\lambda\alpha^{-1}+|\xi|^2)}\Bigr|
\leq C_\delta |\xi|^{-|\delta|}.  \label{pr:3-1}
\end{align}
Here and in the following, we denote $i\xi = (i\xi_1, \ldots, i\xi_N)$ ($N$-vector), and 
$(i\xi)^2 = (i\xi_j i \xi_k \mid j, k=1, \ldots, N)$ ($N^2$-vector). In particular, $i\xi$ and $(i\xi)^2$
are corresponding to $\nabla$ and $\nabla^2$ through the Fourier transform. 
By Lemma \ref{lem:5.1}, we have
\begin{align}\label{pr:3.1}
\|(\lambda, \lambda^{\frac{1}{2}}\nabla, \nabla^2)\CS^0_1(\lambda)\bg\|_{B^s_{q,r}(\BR^N)}
&=\Bigl\|\CF^{-1}_\xi\Bigl[\frac{(\lambda, \lambda^{\frac{1}{2}}i\xi, (i\xi)^2)\CF[\bg](\xi)}
{\lambda\alpha^{-1}+|\xi|^2}\Bigr]\Bigr\|_{B^s_{q,r}(\BR^N)} \\
&\leq C\|\bg\|_{B^s_{q,r}(\BR^N)}, \\
\label{pr:3.2}
\|(\lambda, \lambda^{\frac{1}{2}}\nabla, \nabla^2)\CS^0_2(\lambda)\bg\|_{B^s_{q,r}(\BR^N)}
&= \Bigl\|\CF^{-1}_\xi\Bigl[\frac{(i\xi\otimes i\xi)(\lambda, \lambda^{\frac{1}{2}}i\xi, (i\xi)^2)
\CF[\bg](\xi)}
{(p(\lambda)+|\xi|^2)^{-1}(\lambda\alpha^{-1}+|\xi|^2)}\Bigr]
\Bigr\|_{B^s_{q,r}(\BR^N)} \\
& \leq C\|\bg\|_{B^s_{q,r}(\BR^N)}.
\end{align}
Note that
$$\Bigl|\frac{\eta_\lambda}{\alpha(\alpha+\eta_\lambda)}\Bigr| \leq C
$$
for any $\lambda \in \Sigma_\epsilon + \omega_1$ as follows from 
Lemma \ref{lem:3.1} $(3)$.  Combining these estimates gives \eqref{thm:5.1.1}.\par

Now, we estimate $\pd_\lambda \CS^0(\lambda)$. 
Noting that 
\begin{align*}
\pd_\lambda \Bigl(\frac{\eta_\lambda}{\alpha(\alpha+\eta_\lambda)}\Bigr)
&= -\frac{\gamma^2\lambda^{-2}}{(\alpha+\eta_\lambda)^2}, \quad 
\pd_\lambda p(\lambda) = -\frac{\alpha+\beta+2\gamma^2\lambda^{-1}}
{(\alpha+\eta_\lambda)^2},
\end{align*}
 we have
\begin{equation}\label{solform:4.7}
\pd_\lambda \CS^0(\lambda)\bg = \frac1\alpha\pd_\lambda \CS^0_1(\lambda)\bg 
-\frac{\gamma^2\lambda^{-2}}{(\alpha+\eta_\lambda)^2} \CS^0_2(\lambda)\bg
+ \frac{\eta_\lambda}{\alpha(\alpha+\eta_\lambda)}
\pd_\lambda \CS^0_2(\lambda)\bg.
\end{equation}
Notice that
\begin{align*}
\pd_\lambda \CS^0_1(\lambda)\bg 
& = -\frac{1}{\alpha^2}\CF^{-1}_\xi\Bigl[\frac{\CF[\bg](\xi)}{(\lambda\alpha^{-1}+|\xi|^2)^2}\Bigr], \\
\pd_\lambda \CS^0_2(\lambda)\bg
& = -\pd_\lambda p(\lambda)\CF^{-1}_\xi\Bigl[\frac{\CF[\bg](\xi)}{(p(\lambda)+|\xi|^2)^2}\Bigr]
-\frac{1}{\alpha^2}\CF^{-1}_\xi\Bigl[\frac{(i\xi\otimes i\xi)\CF[\bg](\xi)}{
(p(\lambda)+|\xi|^2)(\lambda\alpha^{-1}+|\xi|^2)^2}\Bigr].
\end{align*}
By Lemma \ref{lem:3.2} with $s=-2$, we have 
\begin{align*}
\Bigl|D^{\delta}_\xi \frac{\lambda(\lambda, \lambda^{\frac{1}{2}}i\xi, (i\xi)^2)}
{(\lambda\alpha^{-1}+|\xi|^2)^2}\Bigr|
\leq C_\delta |\xi|^{-|\delta|}, \\
\Bigl|D^{\delta}_\xi \frac{(i\xi\otimes i\xi)\lambda(\lambda, \lambda^{\frac{1}{2}}i\xi, (i\xi)^2)}
{((p(\lambda)+|\xi|^2)^{-1})^2(\lambda\alpha^{-1}+|\xi|^2)}\Bigr|
\leq C_\delta |\xi|^{-|\delta|}, \\
\Bigl|D^{\delta}_\xi \frac{(i\xi\otimes i\xi)\lambda(\lambda, \lambda^{\frac{1}{2}}i\xi, (i\xi)^2)}
{(p(\lambda)+|\xi|^2)(\lambda\alpha^{-1}+|\xi|^2)^2}\Bigr|
\leq C_\delta |\xi|^{-|\delta|}.
\end{align*}
Thus, by Lemma \ref{lem:5.1}, we have
\begin{equation}\label{pr:3.3}
\|\lambda(\lambda, \lambda^{\frac{1}{2}}\nabla, \nabla^2)\pd_\lambda
\CS^0_\ell(\lambda)\bg\|_{B^s_{q,1}(\BR^N)}
 \leq C\|\bg\|_{B^s_{q,r}(\BR^N)}\quad\text{for $\ell=1,2$}.
\end{equation}
 \par
Moreover, by Lemma \ref{lem:3.1} $(3)$, we have
\begin{align*}
\Bigl|\frac{\gamma^2\lambda^{-2}}{(\alpha+\eta_\lambda)^2}\Bigr| \leq C|\lambda|^{-2}
\end{align*}
for any $\lambda \in \Lambda_{\epsilon, \lambda_0}$. Thus, by \eqref{pr:3.3} we have
\begin{align*}
\|\lambda(\lambda, \lambda^{\frac{1}{2}}\nabla, \nabla^2)\frac{\gamma^2\lambda^{-2}}{(\alpha+\eta_\lambda)}
\CS^0_2(\lambda)\bg\|_{B^s_{q,r}(\BR^N)}
&\leq C|\lambda|^{-1}\|(\lambda, \lambda^{\frac{1}{2}}\nabla, \nabla^2)\CS^0_2(\lambda)\bg\|_{B^s_{q,r}(\BR^N)}\\
&\leq C|\lambda|^{-1}\|\bg\|_{B^s_{q,r}(\BR^N)}.
\end{align*}
Combining these estimates gives \eqref{thm:5.1.2}. \par

Now, we shall prove \eqref{thm:5.1.3}-\eqref{thm:5.1.6}. To this end, we write
$$\frac{\eta_\lambda}{\alpha(\alpha+\eta_\lambda)} =\frac{\eta_\lambda}{\alpha(\alpha+\beta)}
(1+\gamma^2(\alpha+\beta)^{-1}\lambda^{-1})^{-1}
= \frac{\beta+\gamma^2\lambda^{-1}}{\alpha(\alpha+\beta)}\sum_{\ell=0}^\infty
\Bigl(-\frac{\gamma^2}{(\alpha+\beta)\lambda}\Bigr)^\ell.
$$
Thus, choosing $\lambda_0> 0$ in such a way that 
$\frac{\gamma^2}{(\alpha+\beta)\lambda_0} < 1$, we have
\begin{equation}\label{zeta:4.1}
\frac{\eta_\lambda}{\alpha(\alpha+\eta_\lambda)}
= \frac{\beta}{\alpha(\alpha+\beta)} + \lambda^{-1}\zeta(\lambda^{-1}),
\end{equation}
where $\zeta(\tau)$ be a $C^\infty$ function defined for $|\tau| <\tau_0 
:= (\alpha+\beta)\gamma^{-2}$. Thus, we set 
\begin{equation}\label{solform:4.5}
\CT^0_1(\lambda) = \CS^0_1(\lambda) + \frac{\beta}{\alpha(\alpha+\beta)}\CS^0_2(\lambda),
\quad
\CT^0_2(\lambda) = \lambda^{-1}\zeta(\lambda^{-1})\CS^0_2(\lambda).
\end{equation}
Obviously, $\CS^0(\lambda) = \CT^0_1(\lambda) + \CT^0_2(\lambda)$. 
By \eqref{pr:3.2}, we have
$$\|(\lambda, \lambda^{\frac{1}{2}}\nabla, \nabla^2)\CT^0_2(\lambda)\bg\|_{B^s_{q,r}(\BR^N)}
\leq C|\lambda|^{-1}\|\bg\|_{B^s_{q,r}(\BR^N)} 
\leq C\lambda_0^{-1+\frac{\sigma}{2}}|\lambda|^{-\frac{\sigma}{2}}\|\bg\|_{B^s_{q,r}(\BR^N)}.
$$
\par
We write 
$$\pd_\lambda \CT^0_2(\lambda)\bg = \{\pd_\lambda (\lambda^{-1}\zeta(\lambda^{-1}))\}
\CS^0_2(\lambda)\bg + \lambda^{-1}\zeta(\lambda^{-1})\pd_\lambda \CS^0_2(\lambda)$$
and then, applying \eqref{pr:3.2} and \eqref{pr:3.3} gives 
$$\|(\lambda, \lambda^{\frac{1}{2}}\nabla, \nabla^2)\pd_\lambda \CT^0_2(\lambda)\bg
\|_{B^s_{q,r}(\BR^N)} \leq C|\lambda|^{-1}\|\bg\|_{B^s_{q,r}(\BR^N)}
\leq C\lambda_0^{-\frac{\sigma}{2}}|\lambda|^{-(1-\frac{\sigma}{2})}\|\bg\|_{B^s_{q,r}(\BR^N)}.
$$
Combining these two estimates gives \eqref{thm:5.1.5} and \eqref{thm:5.1.6}. \par 
Finally, we shall prove \eqref{thm:5.1.3} and \eqref{thm:5.1.4}.  Let $\bg \in C^\infty_0(\BR^N)$. 
Writing
\begin{align*}\lambda^{\frac{\sigma}{2}}(\lambda^{\frac{1}{2}}\nabla, \nabla^2)\CS^0_1(\lambda)\bg
&= \CF^{-1}_{\xi}\Bigl[\frac{\lambda^{\frac{\sigma}{2}}(\lambda^{\frac{1}{2}}(i\xi), (i\xi)^2)}
{(\lambda\alpha^{-1}+|\xi|^2)(1+|\xi|^2)^{\frac{\sigma}{2}}}
(1+|\xi|^2)^{\frac{\sigma}{2}}\CF[\bg](\xi)\Bigr], \\
\lambda^{\frac{\sigma}{2}}(\lambda^{\frac{1}{2}}\nabla, \nabla^2)\CS^0_1(\lambda)\bg
&= \CF^{-1}_{\xi}\Bigl[\frac{\lambda^{\frac{\sigma}{2}}(\lambda^{\frac{1}{2}}(i\xi), (i\xi)^2)(i\xi\otimes i\xi)}
{(p(\lambda) + |\xi|^2)(\lambda\alpha^{-1}+|\xi|^2)(1+|\xi|^2)^{\frac{\sigma}{2}}}
(1+|\xi|^2)^{\frac{\sigma}{2}}\CF[\bg](\xi)\Bigr]
\end{align*}
and observing that 
\begin{align*}\Bigl|D^{\delta}_\xi\frac{\lambda^{\frac{\sigma}{2}}(\lambda^{\frac{1}{2}}(i\xi), (i\xi)^2)}
{(\lambda\alpha^{-1}+|\xi|^2)(1+|\xi|^2)^{\frac{\sigma}{2}}}\Bigr|
\leq C_\delta|\xi|^{-|\delta|}, \\
\Bigl|D^{\delta}_\xi\frac{\lambda^{\frac{\sigma}{2}}(\lambda^{\frac{1}{2}}(i\xi), (i\xi)^2)(i\xi\otimes i\xi)}
{(p(\lambda) + |\xi|^2)(\lambda\alpha^{-1}+|\xi|^2)(1+|\xi|^2)^{\frac{\sigma}{2}}}\Bigr|
\leq C_\delta|\xi|^{-|\delta|}
\end{align*}
for any multi-index $\delta \in \BN^N_0$, by Lemma \ref{lem:5.1} we have
\begin{align*}
\|\lambda^{\frac{\sigma}{2}}(\lambda^{\frac{1}{2}}\nabla, \nabla^2)
\CS^0_1(\lambda)\bg\|_{B^s_{q,r}(\BR^N)}
&\leq C\|\CF^{-1}_\xi[(1+|\xi|^2)^{\frac{\sigma}{2}}\CF[\bg]]\|_{B^s_{q,r}(\BR^N)}
\leq C\|\bg\|_{B^{s+\sigma}_{q,r}(\BR^N)}, \\
\|\lambda^{\frac{\sigma}{2}}(\lambda^{\frac{1}{2}}\nabla, \nabla^2)
\CS^0_2(\lambda)\bg\|_{B^s_{q,r}(\BR^N)}
&\leq C\|\CF^{-1}_\xi[(1+|\xi|^2)^{\frac{\sigma}{2}}\CF[\bg]]\|_{B^s_{q,r}(\BR^N)}
\leq C\|\bg\|_{B^{s+\sigma}_{q,r}(\BR^N)}.
\end{align*}
Combining these two estimates gives \eqref{thm:5.1.3}. \par 
Write 
$$\pd_\lambda \CT^0_1(\lambda)\bg = \pd_\lambda \CS^0_1(\lambda)\bg 
+ \frac{\beta}{\alpha(\alpha+\beta)}\pd_\lambda \CS^0_2(\lambda)\bg.
$$
Writing
\begin{align*}
&\lambda^{1-\frac{\sigma}{2}}(\lambda, \lambda^{\frac{1}{2}}\nabla, \nabla^2)
\pd_\lambda \CS^0_1(\lambda)\bg \\
&\quad = -\frac{2}{\alpha^2}\CF^{-1}_\xi\Bigl[\frac{\lambda^{1-\frac{\sigma}{2}}
(\lambda, \lambda^{\frac{1}{2}}(i\xi), (i\xi)^2)
(1+|\xi|^2)^{\frac{\sigma}{2}}}{(\lambda\alpha^{-1}+|\xi|^2)^2}
(1+|\xi|^2)^{-\frac{\sigma}{2}}\CF[\bg](\xi)\Bigr],
\end{align*}
and observing that
$$\Bigl|D^{\delta}_\xi\frac{\lambda^{1-\frac{\sigma}{2}}(\lambda, \lambda^{\frac{1}{2}}i\xi, (i\xi)^2)
(1+|\xi|^2)^{\frac{\sigma}{2}}}{(\lambda\alpha^{-1}+|\xi|^2)^2}\Bigr| \leq C|\xi|^{-|\delta|}
$$
for any multi-index $\delta \in \BN_0^N$, by Lemma \ref{lem:5.1} we have
\begin{align*}
\|\lambda^{1-\frac{\sigma}{2}}(\lambda, \lambda^{\frac{1}{2}}\nabla, \nabla^2)\pd_\lambda 
\CS^0_1(\lambda)\bg\|_{B^s_{q,r}(\BR^N)}
&\leq C\|\CF^{-1}_\xi[(1+|\xi|^2)^{-\frac{\sigma}{2}}\CF[\bg](\xi)]\|_{B^s_{q,r}(\BR^N)}\\
&\leq C\|\bg\|_{B^{s-\sigma}_{q,r}(\BR^N)}.
\end{align*}
Writing
\begin{align*}
&\lambda^{1-\frac{\sigma}{2}}(\lambda, \lambda^{\frac{1}{2}}\nabla, \nabla^2)
 \CS^0_2(\lambda)\bg \\
&\quad = -\pd_\lambda p(\lambda)\CF^{-1}_\xi\Bigl[\frac{\lambda^{1-\frac{\sigma}{2}}
(\lambda, \lambda^{\frac{1}{2}}(i\xi), (i\xi)^2)(i\xi\otimes i\xi)
(1+|\xi|^2)^{\frac{\sigma}{2}}}{(p(\lambda)+|\xi|^2)^2(\lambda\alpha^{-1}+|\xi|^2)}
(1+|\xi|^2)^{-\frac{\sigma}{2}}\CF[\bg](\xi)\Bigr] \\
&\quad  - \frac{1}{\alpha}\CF^{-1}_\xi\Bigl[\frac{\lambda^{1-\frac{\sigma}{2}}
(\lambda, \lambda^{\frac{1}{2}}(i\xi), (i\xi)^2)(i\xi\otimes i\xi)
(1+|\xi|^2)^{\frac{\sigma}{2}}}{(p(\lambda)+|\xi|^2)(\lambda\alpha^{-1}+|\xi|^2)^2}
(1+|\xi|^2)^{-\frac{\sigma}{2}}\CF[\bg](\xi)\Bigr],
\end{align*}
and observing that
$$\Bigl|D^{\delta}_\xi\frac{\lambda^{1-\frac{\sigma}{2}}(\lambda, \lambda^{\frac{1}{2}}i\xi, (i\xi)^2)
(i\xi\otimes i\xi)(1+|\xi|^2)^{\frac{\sigma}{2}}}
{(p(\lambda) + |\xi|^2)^{2-\ell}(\lambda\alpha^{-1}+|\xi|^2)^{1+\ell}}\Bigr| \leq C|\xi|^{-|\delta|}
\quad \text{for $\ell=0,1$}
$$ 
for  any multi-index $\delta \in \BN_0^N$, by Lemma \ref{lem:5.1} we have
\begin{align*}
\|\lambda^{1-\frac{\sigma}{2}}(\lambda, \lambda^{\frac{1}{2}}\nabla, \nabla^2)\pd_\lambda 
\CS^0_2(\lambda)\bg\|_{B^s_{q,r}(\BR^N)}
&\leq C\|\CF^{-1}_\xi[(1+|\xi|^2)^{-\frac{\sigma}{2}}\CF[\bg](\xi)]\|_{B^s_{q,r}(\BR^N)}\\
&\leq C\|\bg\|_{B^{s-\sigma}_{q,r}(\BR^N)}.
\end{align*}
Combining these two estimates gives \eqref{thm:5.1.4},  which completes the proof of 
Theorem \ref{thm:5.1}. 
\end{proof}

\section{Estimates of solution formulas of complex Lam\'e equations}
\label{sec:6}

Let $\CS^b(\lambda)=(\CS^b_1(\lambda), \ldots, \CS^b_N(\lambda))$ be the solution operator corresponding to
equations \eqref{cl:2} defined by 
\begin{equation}\label{defsolop.6.1} 
\CS^b_J(\lambda)\bg = w_J
\end{equation}
for $J=1, \ldots, N$, where the partial Fourier transform $\CF'[w_J]$ of $w_J$ are
defined by \eqref{sol:form5}. In this section, we shall estimate $\CS^b_J(\lambda)$. 
Namely, we shall prove the following theorem. 
\begin{thm}\label{thm:6.1} Let $1 < q <\infty$, $1 \leq r \leq \infty$, 
$-1+1/q < s < 1/q$, $\epsilon \in (0, \pi/2)$ and 
$\lambda_0 >0$. 
Then, for any $\lambda \in \Lambda_{\epsilon, \lambda_0}$ and $\bg \in B^s_{q,r}(\HS)^N$,
 there hold
\begin{align}\label{thm:6.1.1-1}
\|(\lambda, \lambda^{\frac{1}{2}}\nabla, \nabla^2)\CS^b(\lambda)\bg\|_{B^s_{q,r}(\HS)}
&\leq C\|\bg\|_{B^s_{q,r}(\HS)}, \\
\label{thm:6.1.1-2}
\|(\lambda, \lambda^{\frac{1}{2}}\nabla, \nabla^2)\pd_\lambda \CS^b(\lambda)\bg\|_{B^s_{q,r}(\HS)}
&\leq C|\lambda|^{-1}\|\bg\|_{B^s_{q,r}(\HS)}.
\end{align}
\par
Moreover, let $\sigma>0$ be a small number such that $-1+1/q < s-\sigma < s < s+\sigma < 1/q$.
Then, there exist a large number $\lambda_0>0$ and 
two operators $\CT^b_1(\lambda)$ and $\CT^b_2(\lambda)$ which are holomorphic on
$\Lambda_{\epsilon, \lambda_0}$ such that $\CS^b(\lambda) = \CT^b_1(\lambda)
+\CT^b_2(\lambda)$ and 
 for any $\bg \in C^\infty_0(\HS)^N$, there hold
\begin{align}\label{thm:6.1.3}
\|(\lambda, \lambda^{\frac{1}{2}}\bar\nabla, \bar\nabla^2)\CT^b_1(\lambda)\bg\|_{B^s_{q,r}(\HS)}
&\leq C|\lambda|^{-\frac{\sigma}{2}}\|\bg\|_{B^{s+\sigma}_{q,r}(\HS)}, \\
\label{thm:6.1.4}
\|(\lambda, \lambda^{\frac{1}{2}}\bar\nabla, \nabla^2)\pd_\lambda\CT^b_1(\lambda)\bg\|_{B^s_{q,r}(\HS)}
&\leq C|\lambda|^{-(1-\frac{\sigma}{2})}\|\bg\|_{B^{s-\sigma}_{q,r}(\HS)}, \\
\label{thm:6.1.5}
\|(\lambda, \lambda^{\frac{1}{2}}\bar\nabla, \bar\nabla^2)\CT^b_2(\lambda)\bg\|_{B^s_{q,r}(\HS)}
&\leq C|\lambda|^{-1}\|\bg\|_{B^{s}_{q,r}(\HS)}, \\
\label{thm:6.1.6}
\|(\lambda, \lambda^{\frac{1}{2}}\bar\nabla, \bar\nabla^2)\pd_\lambda\CT^b_2(\lambda)\bg\|_{B^s_{q,r}(\HS)}
&\leq C|\lambda|^{-2}\|\bg\|_{B^{s}_{q,r}(\HS)}.
\end{align}
\end{thm}
To prove Theorem \ref{thm:6.1}, the argument based on interpolation theory due to 
Shibata \cite{SS23} and Shibata and Watanabe \cite{SK24} play an important role. We  quote this in the following subsection.
\subsection{Spectral analysis based on interpolation theory}\label{subsec.6.1.2}
Below, we assume that $1 < q < \infty$ and $-1+1/q < s < 1/q$, $\epsilon \in (0, \pi/2)$, 
and  $\omega>0$. 
Let $T(\lambda)$ be an operator valued holomorphic function acting on $f \in C^\infty_0(\HS)$ defined for 
 $\lambda \in \Lambda_{\epsilon, \omega}$.  In this subsection, we shall show our strategy how to obtain 
the estimates of $T(\lambda)$ as an operator from one Besov space into another Besov space.  
The following argument is due to Shibata \cite{SS23}, see also Shibata and Watanabe \cite{SK24}. Since this gives one of main ideas, 
for the convenience of readers we record  arguments there. 
\par

We consider  two operator valued holomorphic functions $T_i(\lambda)$ defined on
$\Lambda_{\epsilon, \lambda_0}$ acting on $f \in C^\infty_0(\HS)$. We denote the dual operator
of $T_i(\lambda)$ by $T_i(\lambda)^*$, namely, $T_i(\lambda)^*$ satisfies the equality:
$$(T_i(\lambda)f, \varphi)= (f, T_i(\lambda)^*\varphi) \quad (i=1,2)$$
for any $f$, $\varphi \in C^\infty_0(\HS)$.  Here, $(f, g) = \int_{\HS}f(x)g(x)\,dx$.  Namely, we do not 
take the complex conjugate. \par 
We consider the following two cases. 
\begin{assumption}\label{assump:4.1}
 Let $1 < q < \infty$, $\epsilon \in (0, \pi/2)$,
 and $\omega>0$.  We  assume that the starting evaluations hold as follows{\color{red}{.}} \par
For any $f \in C^\infty_0(\HS)$ and $\lambda \in \Lambda_{\epsilon, \omega}$, 
 the following estimates hold:
\begin{align}
\label{f1} \|T_1(\lambda)f\|_{H^1_q(\HS)} &\leq C\|f\|_{H^1_q(\HS)},  \\
\label{f2} \|T_1(\lambda)f\|_{L_q(\HS)} &\leq C\|f\|_{L_q(\HS)}, \\
\label{f3} \|T_1(\lambda)f\|_{L_q(\HS)} & \leq C|\lambda|^{-\frac{1}{2}}\|f\|_{H^1_q(\HS)},\\
\label{d1*} \|T_1(\lambda)^*f\|_{L_{q'}(\HS)} &\leq C\|f\|_{L_{q'}(\HS)},  \\
\label{d2*} \|T_1(\lambda)^*f\|_{H^1_{q'}(\HS)} &\leq C\|f\|_{H^1_{q'}(\HS)}, \\
\label{d3*} \|T_1(\lambda)^*f\|_{L_{q'}(\HS)} & \leq C|\lambda|^{-\frac{1}{2}}\|f\|_{H^1_{q'}(\HS)}.
\end{align}
\end{assumption}
\begin{assumption}\label{assump:4.2}
 Let $1 < q < \infty$, $\epsilon \in (0, \pi/2)$,
 and $\omega>0$.  We  assume that the starting evaluations hold as follows{\color{red}{.}} \par
For any $f \in C^\infty_0(\HS)$ and $\lambda \in \Lambda_{\epsilon, \omega}$, 
 the following estimates hold:
\begin{align}
\label{d1} \|T_2(\lambda)f\|_{H^1_q(\HS)} &\leq C|\lambda|^{-1}\|f\|_{H^1_q(\HS)},  \\
\label{d2} \|T_2(\lambda)f\|_{L_q(\HS)} &\leq C|\lambda|^{-1}\|f\|_{L_q(\HS)}, \\
\label{d3} \|T_2(\lambda)f\|_{H^1_q(\HS)} & \leq C|\lambda|^{-\frac{1}{2}}\|f\|_{L_q(\HS)}, \\
\label{e1} \|T_2(\lambda)^*f\|_{L_{q'}(\HS)} &\leq C|\lambda|^{-1}\|f\|_{L_{q'}(\HS)},  \\
\label{e2} \|T_2(\lambda)^*f\|_{H^1_{q'}(\HS)} &\leq C|\lambda|^{-1}\|f\|_{H^1_{q'}(\HS)}, \\
\label{e3} \|T_2(\lambda)^*f\|_{H^1_{q'}(\HS)} & \leq C|\lambda|^{-\frac{1}{2}}\|f\|_{L_{q'}(\HS)}.
\end{align}
\end{assumption}
Then, we have the following theorems.
\begin{thm}\label{spectralthm:1}
Let $1 < q < \infty$, $1 \leq r \leq \infty$,  $-1+1/q < s < 1/q$,  
 $\epsilon \in (0, \pi/2)$, and $\omega>0$.
Let $\sigma>0$ be a  small number such that $-1+1/q < s-\sigma < s < s+\sigma < 1/q$.
 Assume that Assumptions \ref{assump:4.1} and \ref{assump:4.2}
hold.  Then, for any $f \in C^\infty_0(\HS)$ and  $\lambda \in \Lambda_{\epsilon, \omega}$, there hold
\begin{align*}
\|T_1(\lambda)f\|_{B^s_{q,r}(\HS)} &\leq C\|f\|_{B^s_{q,r}(\HS)}, \\
\|T_1(\lambda)f\|_{B^s_{q,r}(\HS)} &\leq C|\lambda|^{-\frac{\sigma}{2}}\|f\|_{B^{s+\sigma}_{q,r}(\HS)}, \\
\|T_2(\lambda)f\|_{B^s_{q,r}(\HS)} & \leq C|\lambda|^{-1}\|f\|_{B^s_{q,r}(\HS)}, \\
\|T_2(\lambda)f\|_{B^s_{q,r}(\HS)} &\leq C|\lambda|^{-(1-\frac{\sigma}{2})}
\|f\|_{B^{s-\sigma}_{q,r}(\HS)}
\end{align*}
for some constant $C$. 
\end{thm}
 We divide the proof into the case $0 < s < 1/q$ and the case 
$-1+1/q < s < 0$. The $s=0$ case follows from the real interpolation between 
$B^{-\nu_1}_{q,r}$ and $B^{\nu_2}_{q,r}$ for some small $\nu_i>0$ ($i=1,2$). 
\begin{lem}\label{lem.s.1} Assume that Assumption \ref{assump:4.1} above holds. 
Let $q$, $\epsilon$,  and $\omega$ be the same as in Assumption \ref{assump:4.1}. Let $1 \leq r \leq \infty$. 
Let $0 < s < 1/q$ and let $\sigma>0$ be  numbers such that  $0 <  s+\sigma <1/q$. Then, for any 
$ \lambda \in \Lambda_{\epsilon, \omega}$ and 
$f \in C^\infty_0(\HS)$, there hold
\begin{align}
\|T_1(\lambda)f\|_{B^s_{q,r}(\HS)} & \leq C\|f\|_{B^s_{q,r}(\HS)}, 
\label{A.1}\\
\|T_1(\lambda)f\|_{B^s_{q,r}(\HS)} & \leq C|\lambda|^{-\frac{\sigma}{2}}
\|f\|_{B^{s+\sigma}_{q,r}(\HS)}. \label{A.2}
\end{align}
\end{lem}
\begin{proof} Below, we always assume that $f \in C^\infty_0(\HS)$ and  
$\lambda \in \Lambda_{\epsilon, \omega}$. 
Choose $\mu$ and  $\mu'$ in such a way that
$0 < s < s+\sigma < \mu' <\mu < 1/q$.  Estimates \eqref{f1}, \eqref{f2}, and \eqref{f3} are interpolated with
complex interpolation method to obtain  
\begin{align}
\label{1}
\|T_1(\lambda)f\|_{L_q(\HS)} & \leq C\|f\|_{L_q(\HS)}, \\
\label{2}
\|T_1(\lambda)f\|_{H^\mu_q(\HS)} & \leq C\|f\|_{H^\mu_q(\HS)},
\\
\label{2'}
\|T_1(\lambda)f\|_{H^{\mu'}_q(\HS)} 
& \leq C\|f\|_{H^{\mu'}_q(\HS)}, \\
\label{3}
\|T_1(\lambda)f\|_{L_q(\HS)} & \leq C|\lambda|^{-\frac{\mu}{2}}\|f\|_{H^\mu_q(\HS)}.
\end{align}
By interpolating \eqref{1} and  \eqref{2} with real interpolation method, 
\begin{equation}\label{4}
\|T_1(\lambda)f\|_{B^s_{q,r}(\HS)} \leq C\|f\|_{B^s_{q,r}(\HS)}.
\end{equation}
Choosing $\theta=s/\mu'$ and  setting $A = \mu(1-s/\mu')$, by \eqref{2'} and \eqref{3} with
real interpolation method, 
\begin{equation}\label{5}
\|T_1(\lambda)f\|_{B^s_{q,r}(\HS)} \leq C|\lambda|^{-\frac{A}{2}}
\|f\|_{B^{s+A}_{q,r}(\HS)}.
\end{equation}
Now, we choose $\mu$ and $\mu'$ in such a way that $s < s + \sigma < s+A$,
that is,  we choose $\mu$ and $\mu'$ in such a way that 
 $\sigma/\mu + s/\mu' < 1$ and 
$s+\sigma < \mu' <\mu < 1/q$. Thus, choosing  
$\theta \in (0,1)$ in such a way that  $s+\sigma = (1-\theta) s + \theta(s+A)$,
that is,  $\theta = \sigma/A$, by 
\eqref{4} and \eqref{5} we have 
\begin{equation}\label{6}
\|T_1(\lambda)f\|_{B^{s}_{q,r}(\HS)} \leq C|\lambda|^{-\frac{\sigma}{2}}
\|f\|_{B^{s+\sigma}_{q,r}(\HS)}.
\end{equation}
Therefore, we have \eqref{A.1} and \eqref{A.2}.  This completes the proof of Lemma  \ref{lem.s.1}.
\end{proof} 
\begin{lem}\label{lem.s.2} Assume that Assumption \ref{assump:4.1} above holds. 
Let $q$, $\epsilon$,  and $\omega$ be the same as in Assumption \ref{assump:4.1}. Let $1 \leq r \leq \infty$. 
Let $-1+1/q < s < 0$ and let $\sigma>0$ be  a number such that  $-1+1/q < s+\sigma <0$. 
Then, for any 
$ \lambda \in \Lambda_{\epsilon, \omega}$ and 
$f \in C^\infty_0(\HS)$, there hold
\begin{align}
\|T_1(\lambda)f\|_{B^s_{q,r}(\HS)} & \leq C\|f\|_{B^s_{q,r}(\HS)}, 
\label{A.1*}\\
\|T_1(\lambda)f\|_{B^s_{q,r}(\HS)} & \leq C|\lambda|^{-\frac{\sigma}{2}}
\|f\|_{B^{s+\sigma}_{q,r}(\HS)}. \label{A.2*}
\end{align}
\end{lem}
\begin{proof} Since $-1+1/q < s < 0$, we have $0 < |s| < 1-1/q = 1/q'$. 
Let $\mu$, $\mu'$ and $\sigma$ be positive numbers such that
\begin{equation}\label{number.1} 0 < \mu' < |s| -\sigma < |s|  < \mu < 1/q'.\end{equation}
Using the complex interpolation method, 
by \eqref{d1*}, \eqref{d2*}, and \eqref{d3*}, we have
\begin{align*}
\|T_1(\lambda)^*\varphi\|_{L_{q'}(\HS)} & \leq C\|\varphi\|_{L_{q'}(\HS)}, \\
\|T_1(\lambda)^*\varphi\|_{H^\mu_{q'}(\HS)} & \leq C\|\varphi\|_{H^\mu_{q'}(\HS)},
\\
\|T_1(\lambda)^*\varphi\|_{H^{\mu'}_{q'}(\HS)} 
& \leq C\|\varphi\|_{H^{\mu'}_{q'}(\HS)}, \\
\|T_1(\lambda)^*\varphi\|_{L_{q'}(\HS)} & \leq C|\lambda|^{-\frac{\mu}{2}}\|\varphi\|_{H^\mu_{q'}(\HS)}.
\end{align*}
By the duality argument, we have 
\begin{align}
\label{d4*}
\|T_1(\lambda)f\|_{L_q(\HS)} & \leq C\|f\|_{L_q(\HS)}, \\
\label{d5*}
\|T_1(\lambda)f\|_{H^{-\mu}_q(\HS)} & \leq C\|f\|_{H^{-\mu}_q(\HS)},
\\
\label{d10*}
\|T_1(\lambda)f\|_{H^{-\mu'}_q(\HS)} 
& \leq C\|f\|_{H^{-\mu'}_q(\HS)}, \\
\label{d6*}
\|T_1(\lambda)f\|_{H^{-\mu}_q(\HS)} & \leq C|\lambda|^{-\frac{\mu}{2}}\|f\|_{L_q(\HS)}.
\end{align}
In fact, note that $H^{-\mu}_q(\HS) = (H^\mu_{q', 0}(\HS))'$. 
For any $f$ and $\varphi \in C^\infty_0(\HS)$, by the dual argument we have 
\begin{align*}
|(T_1(\lambda)f, \varphi)| &= |(f, T_1(\lambda)^*\varphi)| \\
&\leq \|f\|_{H^{-\mu}_q(\HS)}\|T_1(\lambda)^*\varphi\|_{H^{\mu}_{q'}(\HS)}
\\
& \leq \|f\|_{H^{-\mu}_q(\HS)}C\|\varphi\|_{H^{\mu}_{q'}(\HS)},
\end{align*}
which implies \eqref{d5*}.  Likewise, we have \eqref{d10*} and \eqref{d4*}.  And also,  
\begin{align*}
|(T_1(\lambda)f, \varphi)| &= |(f, T_1(\lambda)^*\varphi)| \\
&\leq \|f\|_{L_q(\HS)}\|T_1(\lambda)^*\varphi\|_{L_{q'}(\HS)}
\\
& \leq \|f\|_{L_q(\HS)}C|\lambda|^{-\frac{\mu}{2}}\|\varphi\|_{H^{\mu}_{q'}(\HS)},
\end{align*}
which implies \eqref{d6*}.

Now, we shall prove \eqref{A.1*} and \eqref{A.2*}. 
Combining \eqref{d4*} and  \eqref{d5*} with real interpolation method, we have
\begin{equation}\label{d9*}
\|T_1(\lambda)f\|_{B^{-|s|}_{q,r}(\HS)} \leq C\|f\|_{B^{-|s|}_{q,r}(\HS)},
\end{equation}
which shows \eqref{A.1*}. \par
Next, recall that 
$0 < \mu' < |s|-\sigma < |s| < \mu <1/q'$ as follows from \eqref{number.1}. 
 Choose $\theta \in (0, 1)$ in such a way that 
$-|s|= -\mu(1-\theta) - \mu'\theta$, that is $\theta = \dfrac{\mu-|s|}{\mu-\mu'}$.
Combining 
\eqref{d10*} and \eqref{d6*} with real interpolation method which implies that
$$\|T_1(\lambda)f\|_{B^{-|s|}_{q,r}(\HS)}\leq
 C|\lambda|^{-\frac{\mu}{2}(1-\theta)}\|f\|_{B^{(-\mu')\theta}_{q,r}}.
$$
Therefore, we have
\begin{equation}\label{d11*}
\|T_1(\lambda)f\|_{B^{-|s|}_{q,1}(\HS)} \leq 
C|\lambda|^{-\frac{\mu}{2}\frac{|s|-\mu'}{\mu-\mu'}}
\|f\|_{B^{-\frac{\mu'(\mu-|s|)}{\mu-\mu'}}_{q,r}}.
\end{equation}
Since $0 < \mu' < |s|-\sigma$ and $0 < \mu-|s| < \mu-\mu'$,  we have  
$$-|s| < -|s| + \sigma < -\frac{\mu'(\mu-|s|)}{\mu - \mu'}.$$
Choose  $\theta \in (0, 1)$ in such a way that 
$$-|s| + \sigma =(1-\theta)(-|s|) + \theta(-\frac{\mu'(\mu-|s|)}{\mu-\mu'})
$$
Combining \eqref{d9*} and \eqref{d11*} with real interpolation method implies that 
$$\|T_1(\lambda)f\|_{B^{-|s|}_{q,r}(\HS)} 
\leq C|\lambda|^{-\frac{\mu}{2}\frac{|s|-\mu'}{\mu-\mu'}\theta}
\|f\|_{B^{-|s|+\sigma}_{q,r}(\HS)}.
$$
Inserting $\theta = \dfrac{(\mu-\mu')\sigma}{\mu(|s|-\mu')}$, we have 
$$
\|T_1(\lambda)f\|_{B^{-|s|}_{q,r}(\HS)}
\leq C{\color{red}{|}}\lambda|^{-\frac{\sigma}{2}}\|f\|_{B^{-|s|+\sigma}_{q,r}(\HS)},
$$
which shows \eqref{A.2*}.
\end{proof}
\begin{lem}\label{lem.s.3} Assume that Assumption \ref{assump:4.2} holds. 
Let $q$, $\epsilon$,  and $\omega$ be the same as in Assumption \ref{assump:4.2}. Let $1 \leq r \leq \infty$. 
Let $0< s < 1/q $ and let $\sigma>0$ be  numbers such that  $0  < s -\sigma <1/q$. 
Then, for any 
$ \lambda \in \Lambda_{\epsilon, \omega}$ and 
$f \in C^\infty_0(\HS)$, there hold
\begin{align}
\|T_2(\lambda)f\|_{B^s_{q,r}(\HS)} & \leq C|\lambda|^{-1}\|f\|_{B^s_{q,r}(\HS)}, 
\label{A.3}\\
\|T_2(\lambda)f\|_{B^s_{q,r}(\HS)} & \leq C|\lambda|^{-(1-\frac{\sigma}{2})}
\|f\|_{B^{s-\sigma}_{q,r}(\HS)}. \label{A.4}
\end{align}
\end{lem}
\begin{proof}
 Let $\mu$ be a number such that $0 < s < s+\sigma < \mu < 1/q$. 
Combining 
\eqref{d1} and \eqref{d2}, and \eqref{d1} and \eqref{d3} with complex interpolation method,
implies that 
\begin{align}
\label{d4} \|T_2(\lambda)f\|_{L_q(\HS)}& \leq C|\lambda|^{-1}
\|f\|_{L_q(\HS)}, \\
\label{d5} \|T_2(\lambda)f\|_{H^{\mu}_q(\HS)}
&
\leq C|\lambda|^{-1}\|f\|_{H^{\mu}_q(\HS)}, \\
\label{d6}
\|T_2(\lambda)f\|_{H^\mu_q(\HS)} & \leq 
C|\lambda|^{-(1-\frac{\mu}{2})}\|f\|_{L_q(\HS)}.
\end{align}
Combining \eqref{d4} and  \eqref{d5} with real interpolation method yields 
\begin{equation}\label{d7}
\|T_2(\lambda)f\|_{B^s_{q,r}(\HS)}
\leq C|\lambda|^{-1}\|f\|_{B^s_{q,r}(\HS)},
\end{equation}
which shows \eqref{A.3}. \par
Now, choosing $\mu'$  and $\theta$ in such a way that  $0 < \mu' < \mu$ and $\theta = \mu'/ \mu \in (0, 1)$
and combining \eqref{d5} and \eqref{d6} with complex interpolation, we have 
\begin{equation}\label{d6'}
\|T_2(\lambda)f\|_{H^{\mu}_{q}(\HS)}
\leq C|\lambda|^{-(1-\frac{1}{2}(\mu-\mu'))}\|f\|_{H^{\mu'}_{q}(\HS)},
\end{equation}
as follows from 
$\theta + (1-\mu/2)(1-\theta) = 1-(\mu/2)(1-\theta) = 
1-(\mu/2)(1-\mu'/\mu) = 1-(1/2)(\mu-\mu')$. 
\par
Next, we will combine \eqref{d4} and \eqref{d6'} with real interpolation method for  
$s= \theta \mu${\color{red}{.}} Namely, we choose  $\theta = s/\mu
\in (0, 1)$ and so  $\theta\mu' = (\mu'/\mu)s$, and so 
$$(1-\frac{1}{2}(\mu-\mu'))\theta + (1-\theta)=
1-\frac{\theta}{2}(\mu-\mu') =  (1-\frac{s}{2\mu}(\mu-\mu')).
$$
Thus, we have 
\begin{equation}\label{d8}
\|T_2(\lambda)f\|_{B^s_{q,r}(\HS)}
\leq C |\lambda|^{-(1-\frac{s}{2\mu}(\mu-\mu'))}
\|f\|_{B^{\frac{\mu'}{\mu}s}_{q,r}(\HS)} 
\end{equation}

Finally, we will combine \eqref{d7} and \eqref{d8} with real interpolation method. We choose 
$0 < \mu' < \mu$ in such a way that  $(\mu'/\mu)s < s-\sigma < s$, that is 
$0 < \mu' < (1-\sigma/s)\mu$.  And, we choose 
$\theta \in (0, 1)$ in such a way that 
$s-\sigma = (1-\theta)s + \theta(\mu'/\mu)s$, that is  $\theta = \sigma/A$
with $A = s(1-\mu'/\mu)$.  In this case, we have
$$(1-\theta) + \theta(1-\frac{s}{2\mu}(\mu-\mu'))
= 1-\frac{s}{2}(1-\frac{\mu'}{\mu})\theta
=1 - \frac{s}{2}\frac{A}{s}\frac{\sigma}{A} = 1-\frac{\sigma}{2}.
$$
Thus, by \eqref{d7} and \eqref{d8}, we have 
$$
\|T_2(\lambda)f\|_{B^s_{q,r}(\HS)}
	\leq C|\lambda|^{-(1-\frac{\sigma}{2})}\|f\|_{B^{s-\sigma}_{q,r}
(\HS)},
$$
which shows \eqref{A.4}. 
Therefore, we have proved Lemma \ref{lem.s.3}.  
\end{proof}
\begin{lem}\label{lem.s.4} Assume that Assumption \ref{assump:4.2} holds. 
Let $q$, $\epsilon$,  and $\omega$ be the same as in Assumption \ref{assump:4.2}. Let $1 \leq r \leq \infty$. 
Let $-1+1/q < s < 0$ and let $\sigma>0$ be  numbers such that  $-1+1/q <  s-\sigma  < 0$. 
Then, for any 
$ \lambda \in \Lambda_{\epsilon, \omega}$ and 
$f \in C^\infty_0(\HS)$, there hold
\begin{align}
\|T_2(\lambda)f\|_{B^s_{q,r}(\HS)} & \leq C|\lambda|^{-1}\|f\|_{B^s_{q,r}(\HS)}, 
\label{A.3*}\\
\|T_2(\lambda)f\|_{B^s_{q,r}(\HS)} & \leq C|\lambda|^{-(1-\frac{\sigma}{2})}
\|f\|_{B^{s-\sigma}_{q,r}(\HS)}. \label{A.4*}
\end{align}
\end{lem}
\begin{proof}
Combining \eqref{e1}, \eqref{e2} and \eqref{e3} with complex interpolation method
for $|s| < \mu, \mu' < 1-1/q=1/q'$,  we have 
\begin{align}
\label{g1}
\| T_2(\lambda)^*\varphi\|_{L_{q'}(\HS)} 
& \leq C|\lambda|^{-1}\|\varphi\|_{L_{q'}(\HS)}, \\
\label{g2}
\|T_2(\lambda)^*\varphi\|_{H^\mu_{q'}(\HS)} 
& \leq C|\lambda|^{-1}\|\varphi\|_{H^\mu_{q'}(\HS)},
\\
\label{g2*}
\|T_2(\lambda)^*\varphi\|_{H^{\mu'}_{q'}(\HS)} 
& \leq C|\lambda|^{-1}\|\varphi\|_{H^{\mu'}_{q'}(\HS)}, \\
\label{g3}
\|T_2(\lambda)^*\varphi\|_{H^\mu_{q'}(\HS)} 
& \leq C{\color{red}{|}}\lambda|^{-(1-\frac{\mu}{2})}\|\varphi\|_{L_{q'}(\HS)}.
\end{align}
Thus, by the duality argument, we have 
\begin{align}
\label{h1}
\|T_2(\lambda)f\|_{L_{q}(\HS)} 
& \leq C|\lambda|^{-1}\|f\|_{L_{q}(\HS)}, \\
\label{h2}
\|T_2(\lambda)f\|_{H^{-\mu}_q(\HS)}
&\leq C|\lambda|^{-1}\|f\|_{H^{-\mu}_q(\HS)}, \\
\label{h2*} 
\|T_2(\lambda)f\|_{H^{-\mu'}_{q}(\HS)} 
& \leq C|\lambda|^{-1}\|f\|_{H^{-\mu'}_{q}(\HS)},
\\
\label{h3}
\| T_2(\lambda)f\|_{L_q(\HS)} 
& \leq C|\lambda|^{-(1-\frac{\mu}{2})}\|f\|_{H^{-\mu}_{q}(\HS)}.
\end{align}
Noting that $-1 +1/q < -\mu < -|s| < 0$ and combining \eqref{h1} and \eqref{h2} with real interpolation method implies 
\begin{equation}\label{h4} 
\|T_2(\lambda)f\|_{B^{-|s|}_{q,1}(\HS)}
\leq C|\lambda|^{-1}\|f\|_{B^{-|s|}_{q,1}(\HS)},
\end{equation}
which shows \eqref{A.3*}. \par
Choosing $\theta \in (0, 1)$ in such a way that $|s| = \mu'\theta$
and combining \eqref{h2*} and \eqref{h3} with real interpolation method, 
we have 
$$
\|T_2(\lambda)f\|_{B^{-|s|}_{q,r}(\HS)}
\leq C|\lambda|^{-a}
\|f\|_{B^{c}_{q,r}(\HS)}.
$$
Here, 
\begin{align*}
a& = -\theta - (1-\theta)(1-\frac{\mu}{2})
= -1+\frac{\mu}{2}(1-\frac{|s|}{\mu'}),\\
c&=-\mu'\theta - \mu(1-\theta)
= -\mu'\frac{|s|}{\mu'}-\mu(1-\frac{|s|}{\mu'})
= -|s|-\mu(1-\frac{|s|}{\mu'})
= -(|s|+\mu(1-\frac{|s|}{\mu'})).  
\end{align*}
Thus, we have obtained
\begin{equation}\label{h5}
\| T_2(\lambda)f\|_{B^{-|s|}_{q,r}(\HS)}
\leq C|\lambda|^{-(1-\frac{\mu}{2}(1-\frac{|s|}{\mu'}))}
\|f\|_{B^{-(|s|+\mu(1-\frac{|s|}{\mu'}))}_{q,r}(\HS)}.
\end{equation}
Now, we choose $\mu' \in (0, 1)$ in such a way that 
$$-|s| > -|s|-\sigma > -|s|-\mu(1-\frac{|s|}{\mu'}),$$
that is 
\begin{equation}\label{const:1} \frac{\mu|s|}{\mu-\sigma} < \mu' < 1-\frac{1}{q}. \end{equation}
Since $\sigma>0$ may be chosen so small that $\mu/(\mu-\sigma)$
is very close to $1$,  we can choose $\mu'$ in such a way that  $|s| < \mu'$ and \eqref{const:1} holds. \par

We choose $\theta \in (0, 1)$ in such a way that 
$$-|s|-\sigma = -|s|\theta  -(|s|+\mu(1-\frac{|s|}{\mu'}))(1-\theta).$$
Combining \eqref{h4} and \eqref{h5} with real interpolation method implies that 
$$\| T_2(\lambda)f\|_{B^{-|s|}_{q,r}(\HS)}
\leq C|\lambda|^{-d}\|f\|_{B^{-|s|-\sigma}_{q,r}(\HS)},
$$
where 
$$d = \theta + (1-\theta)(1-\frac{\mu}{2}(1-\frac{|s|}{\mu'})) = 1-\frac{\sigma}{2}.$$
Thus, we have 
$$\|T_2(\lambda)f\|_{B^{-|s|}_{q,r}(\HS)}
\leq C|\lambda|^{-(1-\frac{\sigma}{2})}\|f\|_{B^{-|s|-\sigma}_{q,r}(\HS)}.
$$
Namely, we have  \eqref{A.4*}, which completes the proof of Lemma \ref{lem.s.4}.
\end{proof}
\begin{proof}[The end of the proof of Theorem \ref{spectralthm:1}.] In view of Lemmas \ref{lem.s.1}--\ref{lem.s.4}, it suffices to prove the case $s=0$. 
Let  $0 < \sigma < 1/q$ and let $\nu_1$ and $\nu_2$ be positive numbers such that  
$-1+1/q < -\nu_1 < -\nu_1 + \sigma < 0 <  \nu_2 < \sigma + \nu_2  < 1/q$. 
 By Lemmas \ref{lem.s.1}
and \ref{lem.s.2}, we have
$$\|T_1(\lambda)f\|_{B^{-\nu_1}_{q,r}(\HS)} \leq C\|f\|_{B^{ -\nu_1}_{q,r}(\HS)},
\quad \|T_1(\lambda)f\|_{B^{\nu_2}_{q,r}(\HS)} \leq C\|f\|_{B^{ \nu_2}_{q,r}(\HS)}.$$
Let $\theta$ be a number  $\in (0, 1)$ such that $0 = (1-\theta)(-\nu_1) + \theta \nu_2$. 
Since 
$$B^0_{q,r}(\HS) = (B^{-\nu_1}_{q,r}(\HS), B^{\nu_2}_{q,r}(\HS))_{\theta, r}, $$
by  real interpolation we have 
$$\|T_1(\lambda)f\|_{B^{0}_{q,r}(\HS)} \leq C\|f\|_{B^{0}_{q,r}(\HS)}. $$
Moreover,  by Lemmas \ref{lem.s.1}
and \ref{lem.s.2}, we have
$$\|T_1(\lambda)f\|_{B^{-\nu_1}_{q,r}(\HS)} \leq C|\lambda|^{-\frac{\sigma}{2}}
\|f\|_{B^{-\nu_1+\sigma}_{q,r}(\HS)}, \quad
\|T_1(\lambda)f\|_{B^{\nu_2}_{q,r}(\HS)} \leq C|\lambda|^{-\frac{\sigma}{2}}
\|f\|_{B^{\nu_2+\sigma}_{q,r}(\HS)}.
$$
Since $B^\sigma_{q,r}(\HS) = (B^{-\nu_1+\sigma}_{q,r}(\HS), B^{\nu_2+\sigma}_{q,r}(\HS))_{\theta, r}$, 
by real interpolation, we have 
$$\|T_1(\lambda)f\|_{B^{0}_{q,r}(\HS)} \leq C|\lambda|^{-\frac{\sigma}{2}}\|f\|_{B^{\sigma}_{q,r}(\HS)}. $$
Analogously, we have 
\begin{align*}
\|T_2(\lambda)f\|_{B^{0}_{q,r}(\HS)} &\leq C|\lambda|^{-1}\|f\|_{B^{0}_{q,r}(\HS)}, \\
\|T_2(\lambda)f\|_{B^{0}_{q,r}(\HS)} &\leq C|\lambda|^{-(1-\frac{\sigma}{2})}\|f\|_{B^{-\sigma}_{q,r}(\HS)}. 
\end{align*}
This completes the proof of Theorem \ref{spectralthm:1}.
\end{proof}

\subsection{A proof of Theorem \ref{thm:6.1}.}

In this subsection, we shall prove Theorem \ref{thm:6.1}.
First, we divide $\dfrac{\eta_\lambda}{K}$ and 
$\dfrac{\eta_\lambda}{K}\dfrac{\beta \lambda+\gamma^2}{(\alpha+\beta)\lambda+\gamma^2}$
appearing in \eqref{sol:form5}. To this end, we start with the following 
lemma.	
\begin{lem}\label{lem:6.1} Let $\epsilon \in (0, \pi/2)$ and $\lambda_0 > 0$.  Set 
$K_1 = (\alpha+\beta)A + \alpha B$.  Then, there exists a constant $c_3 > 0$ depending on
$\alpha$, $\beta$, $\epsilon$ and $\epsilon'$ appearing in Lemma \ref{lem:3.1} such that 
for any $\lambda \in \Lambda_{\epsilon, \lambda_0}$ and $\xi' \in \BR^{N-1}\setminus\{0\}$, 
there holds 
\begin{equation}\label{lem:6.1.1} |K_1| \geq c_3(|\lambda|^{1/2} + |\xi'|).
\end{equation}
Moreover, for any multi-index $\delta' \in \BN_0^{N-1}$,  
$\lambda \in \Lambda_{\epsilon, \lambda_0}$ and $\xi' \in \BR^{N-1}\setminus\{0\}$
there holds
\begin{equation}\label{lem:6.1.2} \begin{aligned}
|D^{\delta'}_{\xi'}K_1^{-1}| &\leq C_{\delta'}(|\lambda|^{1/2} + |\xi'|)^{-1-|\delta'|}, \\
|D^{\delta'}_{\xi'}\pd_\lambda K_1^{-1}| &\leq C_{\delta'}(|\lambda|^{1/2} + |\xi'|)^{-3-|\delta'|}.
\end{aligned}\end{equation}
\end{lem}
\begin{proof}
By Lemma \ref{lem:3.1}, we have 
$$|\arg(\alpha + \beta) A| \leq \frac{\pi-\epsilon'}{2}, \quad |\arg \alpha B| \leq \frac{\pi-\epsilon}{2}.
$$
Moreover,  we see that 
$$(\alpha+\beta)|A| \geq \frac{\alpha+\beta}{\sqrt{2}}\sqrt{c_1}(|\lambda|^{1/2}+|\xi'|),\quad
\alpha |B| \geq \frac{\alpha}{\sqrt{2}}
\sqrt{\sin(\epsilon/2)}(|\lambda|^{1/2}\alpha^{-1/2}+|\xi'|).
$$
From geometric interpretation of the sum of complex numbers 
 we see that 
\eqref{lem:6.1.1} holds with 
$$c_3 = (\sin\frac{\min(\sigma, \epsilon)}{2})\min(\frac{(\alpha+\beta)}{\sqrt{2}}\sqrt{c_1},
\frac{\alpha}{\sqrt{2}}\sqrt{\sin(\epsilon/2)}\min(\alpha^{-1/2}, 1))$$
\par
By Bell's formula, 
\begin{equation}\label{lem:6.1.3}
|D^{\delta'}_{\xi'}K_1^{-1}| \leq C_{\delta'}\sum_{\ell=1}^{|\delta'|}
|K_1|^{-(\ell+1)}\sum_{|\delta'_1| + \cdots+|\delta'_\ell| = |\delta'|  \atop |\delta'_i| \geq 1}
|D^{\delta'_1}_{\xi'}K_1|\cdots |D^{\delta'_\ell}_{\xi'}K_1|.
\end{equation}
By Lemma \ref{lem:3.2} with $s=1$, we have
$$|D^{\delta'_1}_{\xi'}K_1|\cdots |D^{\delta'_\ell}_{\xi'}K_1| \leq C_{\delta'}(|\lambda|^{1/2}
+|\xi'|)^{\ell-|\delta'|},$$
which, combined with \eqref{lem:6.1.3}, implies \eqref{lem:6.1.2}. \par
Writing 
$\pd_\lambda K_1 = (\alpha+\beta)\pd_\lambda A + \alpha \pd_\lambda B$,
 by Lemma \ref{lem:3.3} we have 
$$
|D^{\delta'}_{\xi'} (\pd_\lambda K_1 )| \leq C_{\delta'}(|\lambda|^{1/2}+|\xi'|)^{-1 -|\delta'|}.
$$
Since $\pd_\lambda K_1^{-1} = -(\pd_\lambda K_1)K_1^{-2}$,  \eqref{lem:6.1.2}
follows from $\pd_\lambda K_1 \in \BM_{-1}$ and 
$K_1^{-1} \in \BM_{-1}$.
This completes the proof of Lemma \ref{lem:6.1}. 
\end{proof}
Recall that $K = (\alpha+\eta_\lambda)A  + \alpha B = K_1 + \gamma^2\lambda^{-1} A$
 (cf. \eqref{halfsymbol:3.1}).  In particular, 
$$|A| \leq \max(c_2^{-1/2}, 1)(|\lambda|^{1/2} + |\xi'|).$$ 
By \eqref{lem:6.1.1} $|\gamma^2AK_1^{-1}| \leq \gamma^2\max(c_2^{-1/2},1)c_3^{-1}$. 
Setting $\lambda_1 = 2\gamma^2\max(c_2^{-1/2},1)c_3^{-1}$, 
we have
\begin{equation}\label{6.1.4} |\gamma^2AK_1^{-1}\lambda^{-1}| \leq 1/2
\end{equation}
for any $\lambda \in \Lambda_{\epsilon, \lambda_1}$ and $\xi' \in \BR^{N-1}\setminus\{0\}$. 
Thus, we have
\begin{equation}\label{6.1.5}
K^{-1} = \frac{1}{K_1}-\frac{1}{K_1}\frac{\gamma^2AK_1^{-1}\lambda^{-1}}{1+\gamma^2AK_1^{-1}\lambda^{-1}}.
\end{equation}
From this observation, it follows that
$$\frac{\eta_\lambda}{K} = \frac{\beta}{K_1} -\frac{1}{\lambda}\frac{\beta}{K_1}\frac{\gamma^2 AK_1^{-1}}
{1 + \gamma^2 AK_1^{-1}\lambda^{-1}} + \frac{\gamma^2}{\lambda}\frac{\beta}{K}.
$$
Thus, setting 
\begin{equation}\label{6.1.6}
K_2(\lambda) =-\frac{\beta}{K_1}\frac{\gamma^2 AK_1^{-1}}
{1 + \gamma^2 AK_1^{-1}\lambda^{-1}} + \frac{\gamma^2 \beta}{K},
\end{equation}
we may write
\begin{equation}\label{6.1.7}
\frac{\eta_\lambda}{K} = \frac{\beta}{K_1} + \lambda^{-1}K_2(\lambda).
\end{equation}
\begin{lem}\label{lem:6.2} Let $\epsilon \in (0, \pi/2)$ and let $\lambda_1$ be a positive number
defined in \eqref{6.1.4}.  Let $K_2$ be the function defined in \eqref{6.1.6}.
Then, for any $\lambda \in \Lambda_{\epsilon, \lambda_1}$, $\xi' \in \BR^{N-1}\setminus\{0\}$,
and multi-index $\delta' \in \BN^{N-1}_0$, there hold
\begin{align}
\label{lem:6.2.1}
|D^{\delta'}_{\xi'}K_2(\lambda)| &\leq C_{\delta'}(|\lambda|^{1/2}+|\xi'|)^{-1-|\delta'|}, \\
\label{lem:6.2.2}
|D^{\delta'}_{\xi'}(\pd_\lambda K_2(\lambda))| &\leq C_{\delta'}(|\lambda|^{1/2}+|\xi'|)^{-1-|\delta'|}
|\lambda|^{-1}
\end{align}
with some constant $C_{\delta'}$. 
\end{lem}
\begin{proof} In what follows, we assume that  $\lambda \in \Lambda_{\epsilon, \lambda_1}$ and $\xi' \in 
\BR^{N-1}\setminus\{0\}$.  By \eqref{lem:6.1.2} and Lemma \ref{lem:3.2}, we have 
\begin{equation}\label{6.1.8*}
|D_{\xi'}^{\delta'}(AK_1^{-1})| \leq C_{\delta'}(|\lambda|^{1/2}+|\xi'|)^{-|\delta'|}
\end{equation}
for any multi-index $\delta' \in \BN_0^{N-1}$. We may assume that 
$|1+\gamma^2AK_1^{-1}\lambda^{-1}|^{-1} \leq 2$ from \eqref{6.1.4}, and so 
by  Bell's formula, \eqref{6.1.4},  and \eqref{6.1.8*} 
\begin{align}
&|D^{\delta'}_{\xi'}(1+ \gamma^2AK_1^{-1}\lambda^{-1})^{-1}| \nonumber\\
&\leq C_{\delta'}\sum_{\ell=1}^{|\delta'|}|1+\gamma^2AK_1^{-1}\lambda^{-1}|^{-(\ell+1)}
\sum_{\delta'_1+ \cdots+\delta'_\ell = \delta' \atop |\delta'_i| \geq 1}
|D_{\xi'}^{\delta'_1}(\gamma^2AK_1^{-1}\lambda^{-1})|\cdots 
|D_{\xi'}^{\delta'_\ell}(\gamma^2AK_1^{-1}\lambda^{-1})| \nonumber \\
&\leq C_{\delta'}\big\{\sum_{\ell=1}^{|\delta'|} (\gamma^2|\lambda|^{-1})^\ell 2^{\ell+1}\big\}(|\lambda|^{1/2}
+|\xi'|)^{-|\delta'|} \leq C_{\delta'}\frac{4\gamma^2|\lambda|^{-1}}{(1-2\gamma^2|\lambda|^{-1})}
(|\lambda|^{1/2}+|\xi'|)^{-|\delta'|}
\label{6.1.10}
\end{align}
provided that $2\gamma^2|\lambda|^{-1} < 1$. 
Combining Lemma \ref{lem:3.3} with $M=K$, \eqref{lem:6.1.2}, \eqref{6.1.8*} and \eqref{6.1.10}
gives \eqref{lem:6.2.1}. \par
To prove \eqref{lem:6.2.2}, we write 
\begin{align*}
\pd_\lambda K_2(\lambda)& = -\beta\pd_\lambda K_1^{-1} \frac{\gamma^2 AK_1^{-1}}{1 + \gamma^2 A K_1^{-1}\lambda^{-1}}
\\
&+\frac{\beta}{K_1}\frac{\gamma^2\lambda^{-1}\pd_\lambda(AK_1^{-1})
\gamma^2 A K_1^{-1}}{(1 + \gamma^2 A K_1^{-1}\lambda^{-1})^2}
- \frac{\gamma^2\pd_\lambda(AK_1^{-1})}{1+\gamma^2A K_1^{-1}\lambda^{-1}}
-\pd_\lambda K^{-1}.
\end{align*}
By \eqref{lem:6.1.2}, we have $\pd_\lambda K_1^{-1} \in \BM_{-3}$. 
Since $\pd_\lambda K = (\alpha + \beta + \gamma^2\lambda^{-1})\pd_\lambda A
+\beta \pd_\lambda B -\gamma^2\lambda^{-2}A$, by  
Lemma \ref{lem:3.3}, we have
\begin{equation}\label{6.1.15}
|D^{\delta'}_{\xi'} (\pd_\lambda K )| \leq C_{\delta'}\{(|\lambda|^{1/2}+|\xi'|)^{-1 -|\delta'|}
+ |\lambda|^{-2}(|\lambda|^{1/2}+|\xi'|)^{1-|\delta'|}\}.
\end{equation}
Writing $\pd_\lambda K^{-1} = -K^{-2}\pd_\lambda K$ and using Lemma \ref{lem:3.3} and \eqref{6.1.15},
we have 
\begin{equation}\label{6.1.13}
|D^{\delta'}_{\xi'} (\pd_\lambda K^{-1} )| \leq C_{\delta'}|\lambda|^{-1}(|\lambda|^{1/2}+|\xi'|)^{-1 -|\delta'|}.
\end{equation}
Writing 
$\pd_\lambda(AK_1^{-1}) = (\pd_\lambda A)K_1^{-1} - AK_1^{-2}\pd_\lambda K_1$, 
by \eqref{6.1.10}, Lemma \ref{lem:3.3} and \eqref{lem:6.1.2} we have
\begin{equation}\label{6.1.14}
|D^{\delta'}_{\xi'}\pd_\lambda(AK_1^{-1})| \leq C_{\delta'}(|\lambda|^{1/2} + |\xi'|)^{-2 -|\delta'|}.
\end{equation}
Thus, by \eqref{lem:6.1.2}, \eqref{6.1.10}, \eqref{6.1.8}, \eqref{lem:6.1.2}, 
\eqref{6.1.13}, and \eqref{6.1.14},  we have \eqref{lem:6.2.2}.
This completes the proof of Lemma \ref{lem:6.2}.
\end{proof}
\par
We write 
\begin{equation}\label{div:6.1}\begin{aligned}
&\frac{\beta\lambda+\gamma^2}{(\alpha+\beta)\lambda+ \gamma^2} = \frac{\beta}{\alpha+\beta}
+ \frac{q(\lambda)}{\lambda}, \quad
\frac{\eta_\lambda}{K} = \frac{\beta}{K_1} + \frac1\lambda K_2(\lambda),\\
&\frac{\eta_\lambda}{K}\frac{\beta\lambda+\gamma^2}{(\alpha+\beta)\lambda+\gamma^2}
= \frac{\beta^2}{\alpha+\beta}\frac{1}{K_1} + \frac1\lambda K_3(\lambda)
\end{aligned}\end{equation}
where 
\begin{equation}\label{6.1.8}\begin{aligned}
q(\lambda) &=  \frac{\alpha\gamma^2 \lambda}{(\alpha+\beta)((\alpha+\beta)\lambda+ \gamma^2)}, \\
 K_3(\lambda)& = \frac{\alpha\beta\gamma^2}{(\alpha+\beta)K_1}
\frac{\lambda}{(\alpha+\beta)\lambda+\gamma^2}+ K_2(\lambda)\frac{\beta\lambda+\gamma^2}
{(\alpha+\beta)\lambda+\gamma^2}.
\end{aligned}\end{equation}
Notice that $|(\alpha+\beta)\lambda + \gamma^2| \geq (1/2)(\alpha+\beta)|\lambda|$ for 
$|\lambda| \geq 2\gamma^2(\alpha+\beta)^{-1}$. 
\begin{cor}\label{lem:cor:6.2}
Let $\epsilon \in (0, \pi/2)$ and let $\lambda_1$ be a positive number
defined in \eqref{6.1.4}. Set $\lambda_2= \max(\lambda_1, 2\gamma^2/(\alpha+\beta))$. 
Let $K_3$ be the function defined in \eqref{6.1.8}.
Then, for any $\lambda \in \Lambda_{\epsilon, \lambda_2}$, $\xi' \in \BR^{N-1}\setminus\{0\}$,
and multi-index $\delta' \in \BN^{N-1}_0$, there hold
\begin{align}
\label{lem:6.2.4}
|D^{\delta'}_{\xi'}K_3(\lambda)| &\leq C_{\delta'}(|\lambda|^{1/2}+|\xi'|)^{-1-|\delta'|}, \\
\label{lem:6.2.4}
|D^{\delta'}_{\xi'}(\pd_\lambda K_3(\lambda))| &\leq C_{\delta'}(|\lambda|^{1/2}+|\xi'|)^{-1-|\delta'|}
|\lambda|^{-1}
\end{align}
with some constant $C_{\delta'}$. 
\end{cor}
Applying the decompositions \eqref{div:6.1} to the formulas in  \eqref{sol:form5}, we define $\CT^b_{1J}(\lambda)$ and 
$\CT^b_{2J}(\lambda)$ by 
\allowdisplaybreaks
\begin{align*}
&\CT^b_{1j}(\lambda)\bg   =\int_{0}^{\infty}\CF^{-1}_{\xi'}\Bigl[Be^{-(x_N+y_N)B} \frac{1}{\alpha B^2}\CF[g_j]
(\xi', y_N)\Bigr](x')\,dy_N
\nonumber\\
&+ \int^\infty_0\CF^{-1}_{\xi'}\Bigl[B^2e^{-Bx_N}\CM(y_N)\frac{\beta}{\alpha+\beta}
(\sum_{k=1}^{N-1}\frac{i\xi_j \xi_k}{(A+B)AB^2}\CF'[g_k](\xi', y_N) \\
&\hskip6cm-\frac{i\xi_j}{(A+B)B^2}\CF'[g_N](\xi', y_N))\Bigr](x')\,dy_N \\
&-  \int^\infty_0\CF^{-1}_{\xi'}\Bigl[Be^{-B(x_N+y_N)}
\frac{\beta}{\alpha+\beta}
\sum_{k=1}^{N-1}\frac{i\xi_j \xi_k}{AB^2(A+B)}\CF'[g_k](\xi', y_N)\Bigr](x')\,dy_N \\
& -\int^\infty_0\CF^{-1}_{\xi'}\Bigl[B^2\CM(x_N)e^{-By_N}
\sum_{k=1}^{N-1}\frac{\beta i\xi_j i\xi_k}{\alpha K_1B^3} 
\CF[\bg_k](\xi', y_N)\Bigr](x')\,dy_N  \\
&+ \int^\infty_0\CF^{-1}_{\xi'}\Bigl[B^3\CM(x_N)\CM(y_N)
\frac{\beta^2 i\xi_j}{(\alpha+\beta)K_1}
(\sum_{k=1}^{N-1}\frac{|\xi'|^2 \xi_k}{(A+B)AB^3}\CF'[g_k](\xi', y_N) \\
&\hskip6cm -\frac{|\xi'|^2}{(A+B)B^3}\CF'[g_N](\xi', y_N))\Bigr](x')\,dy_N \\
&-\int^\infty_0\CF^{-1}_{\xi'}\Bigl[ B^2\CM(x_N)e^{-By_N}
\frac{\beta^2 i\xi_j}{(\alpha+\beta)K_1}
\sum_{k=1}^{N-1}\frac{|\xi'|^2 \xi_k}{AB^3(A+B)}\CF'[g_k](\xi', y_N)\Bigr](x')\,dy_N, \\
&\CT^b_{1N}(\lambda)\bg =
 \int^\infty_0\CF^{-1}_{\xi'}\Bigl[B^2\CM(x_N)e^{-By_N}
\sum_{k=1}^{N-1}\frac{\beta A i\xi_k}{\alpha K_1B^3} 
\CF[\bg_k](\xi', y_N)\Bigr](x')\,dy_N \\
&-\int^\infty_0\CF^{-1}_{\xi'}\Bigl[B^3\CM(x_N)\CM(y_N)
\frac{\beta^2A}{(\alpha+\beta)K_1}
(\sum_{k=1}^{N-1}\frac{|\xi'|^2 \xi_k}{(A+B)AB^3}\CF'[g_k](\xi', y_N) \\
&\hskip6cm-\frac{|\xi'|^2}{(A+B)B^3}\CF'[g_N](\xi', y_N)\Bigr](x')\,dy_N\nonumber\\
&+\int^\infty_0\CF^{-1}_{\xi'}\Bigl[ B^2\CM(x_N)e^{-By_N}
\frac{\beta^2A}{(\alpha+\beta)K_1}
\sum_{k=1}^{N-1}\frac{|\xi'|^2 \xi_k}{AB^3(A+B)}\CF'[g_k](\xi', y_N)\Bigr](x')\,dy_N, \\
&\CT^b_{2j}(\lambda)\bg \\
&=\int^\infty_0\CF^{-1}_{\xi'}\Bigl[B^2e^{-Bx_N}\CM(y_N)\frac{q(\lambda)}{\lambda}
(\sum_{k=1}^{N-1}\frac{i\xi_j \xi_k}{(A+B)AB^2}\CF'[g_k](\xi', y_N)\\
&\hskip7cm -\frac{i\xi_j}{(A+B)B^2}\CF'[g_N](\xi',y_N))\Bigr](x')\,dy_N \nonumber\\
&-  \int^\infty_0\CF^{-1}_{\xi'}\Bigl[Be^{-B(x_N+y_N)}
\frac{q(\lambda)}{\lambda}
\sum_{k=1}^{N-1}\frac{i\xi_j \xi_k}{AB^2(A+B)}\CF'[g_k](\xi', y_N)\Bigr](x')\,dy_N \\
& -\int^\infty_0\CF^{-1}_{\xi'}\Bigl[B^2\CM(x_N)e^{-By_N}
\frac{K_2(\lambda)}{\lambda}\frac{i\xi_j}{\alpha B^3}
i\xi'\cdot \CF[\bg'](\xi', y_N)\Bigr](x')\,dy_N\\
&+ \int^\infty_0\CF^{-1}_{\xi'}\Bigl[B^3\CM(x_N)\CM(y_N)
(\sum_{k=1}^{N-1}\frac{K_3(\lambda)}{\lambda}
\frac{i\xi_j \xi_k|\xi'|^2}{(A+B)AB^3}
\CF'[g_k](\xi', y_N)\nonumber\\
&\hskip7cm 
-\frac{K_3(\lambda)}{\lambda}\frac{i\xi_j |\xi'|^2 }{(A+B)B^3}\CF'[g_N](\xi', y_N))
\Bigr](x')\,dy_N \nonumber\\
&-\int^\infty_0\CF^{-1}_{\xi'}\Bigl[ B^2\CM(x_N)e^{-By_N}
\sum_{k=1}^{N-1}\frac{K_3(\lambda)}{\lambda}
\frac{i\xi_j |\xi'|^2 \xi_k}{AB^3(A+B)}\CF'[g_k](\xi', y_N)\Bigr](x')\,dy_N, \\
&\CT^b_{2N}(\lambda)\bg =
 \int^\infty_0\CF^{-1}_{\xi'}\Bigl[B^2\CM(x_N)e^{-By_N}\sum_{k=1}^{N-1}
\frac{K_2(\lambda)}{\lambda}\frac{i\xi_k A}{\alpha B^3}
\CF[\bg_k](\xi', y_N)\Bigr](x')\,dy_N \\
&-\int^\infty_0\CF^{-1}_{\xi'}\Bigl[B^3\CM(x_N)\CM(y_N)
(\sum_{k=1}^{N-1}\frac{K_3(\lambda)}{\lambda}
\frac{|\xi'|^2 \xi_k}{(A+B)B^3}\CF'[g_k](\xi', y_N)  \\
&\hskip7cm 
-\frac{K_3(\lambda)}{\lambda}\frac{|\xi'|^2A}{(A+B)B^3}\CF'[g_N](\xi', y_N)\Bigr](x')\,dy_N\nonumber\\
&+\int^\infty_0 \CF^{-1}_{\xi'}\Bigl[B^2\CM(x_N)e^{-By_N}
\sum_{k=1}^{N-1}\frac{K_3(\lambda)}{\lambda}\frac{|\xi'|^2 \xi_k}{B^3(A+B)}
\CF'[g_k](\xi', y_N)\Bigr](x')\,dy_N.
\end{align*}
And then, we have
\begin{equation}
\CS^b(\lambda)\bg = \CT^b_1(\lambda)\bg + \CT^b_2(\lambda)\bg
\end{equation}
where we have set $\CT^b_i(\lambda) = (\CT^b_{i1}(\lambda), \ldots, \CT^b_{iN}(\lambda))$
$(i=1,2)$. \par

To estimate $\CT^b_1(\lambda) = (\CT^b_{11}(\lambda), \ldots, \CT^b_{1N}(\lambda))$, we introduce
the multiplier class $\BN_k$ defined by 
$$\BN_k = \{m(\lambda, \xi') \in \BM_{k}(\Lambda_{\epsilon, \lambda_0}) \mid 
\pd_\lambda m(\lambda, \xi') \in \BM_{k-2}(\Lambda_{\epsilon, \lambda_0})\}.
$$
By Lemmas \ref{lem:3.3} and \ref{lem:6.1}, all the following symbols appearing in the definition of
$\CT^b_{1J}(\lambda)$ ($J=1, \ldots, N-1, N$):
\begin{gather*} 
\frac{1}{B^2}, \quad  \frac{i\xi_ji\xi_k}{(A+B)AB^2}, \quad  \frac{i\xi_j}{(A+B)B^2}, 
\quad \frac{i\xi_ji\xi_k}{K_1B^3}, \quad \frac{i\xi_j|\xi'|^2\xi_k}{K_1(A+B)AB^3}, \\
\frac{i\xi_j|\xi'|^2}{K_1(A+B)B^3}, \quad 
\frac{Ai\xi_k}{K_1B^3}, \quad  \frac{A|\xi'|^2\xi_k}{K_1(A+B)AB^3}, \quad \frac{A|\xi'|^2}{K_1(A+B)B^3}
\end{gather*}
 belong to $\BN_{-2}$. To represent 
$\CT^b_1(\lambda)$ in a little bit simple way, we define symbols $\CP_i$ ($i=1,2,3,4$) by 
\begin{gather*}
\CP_1(x_N, y_N) =B e^{-B(x_N+y_N)}, \quad \CP_2(x_N, y_N) = B^2 e^{-Bx_N}\CM(y_N), \\
\CP_3(x_N, y_N) = B^2 \CM(x_N)e^{-By_N}, \quad \CP_4(x_N, y_N) =B^3 \CM(x_N)\CM(y_N).
\end{gather*}
Then, we may assert that 
there exist four $N\times N$ matrices of $\BN_{-2}$ symbols $\bT^{b, 0}_{1j}(\lambda, \xi')$ 
such that 
$\CT^b_1(\lambda)$ is represented by
\begin{equation}\label{repr:6.1}
\CT^b_1(\lambda)\bg = \int^\infty_0 \CF^{-1}_{\xi'}\Bigl[ \Bigl(\sum_{j=1}^4 \CP_j(x_N, y_N) 
\bT^{b, 0}_{1j}(\lambda, \xi')\Bigr)\CF'[\bg](\xi', y_N)\Bigr](x')\, dy_N.
\end{equation}
First, we shall prove the $L_q$-$L_q$ estimate. 
Below, we write $\nabla' = (\pd_1, \ldots, \pd_{N-1})$, 
$\nabla'' = (\pd_j\pd_k \mid j, k=1, \ldots, N-1)$, and $\nabla''' = (\pd^\delta \mid
|\delta|=3)$.  Corresponding symbols are written 
by $\xi' = (\xi_1, \ldots, \xi_{N-1})$, $(i\xi')^{2} = (i\xi_j i\xi_k \mid j,k=1, \ldots, N-1)$,
and $(\xi')^3=(i\xi_ji\xi_ki\xi_\ell \mid j,k,\ell=1, \ldots, N-1)$.  Using the formulas:
\begin{equation}\label{diffm.1}
\pd_N^\ell \CM(x_N) = (-1)^\ell(A^\ell \CM(x_N) + \frac{A^\ell-B^\ell}{A-B}e^{-Bx_N})\quad(\ell \geq 1),
\end{equation}
we write 
\begin{equation}\label{repr:6.3}
\pd_N^\ell \CT^b_1(\lambda)\bg =  (-1)^\ell
\int^\infty_0 \CF^{-1}_{\xi'}\Bigl[ \Bigl(\sum_{j=1}^4 \CP_j(x_N, y_N) 
\bT^{b,\ell}_{1j}(\lambda, \xi')\Bigr)\CF'[\bg](\xi', y_N)\Bigr](x')\, dy_N
\end{equation}
for $\ell=1,2$, where we have set 
\begin{align*}
\bT^{b, \ell}_{1k}(\lambda, \xi')& = B^{\ell}\bT^{b, 0}_{1k}(\lambda, \xi')+\frac{A^\ell-B^{\ell}}{A-B}B
\bT^{b, 0}_{1k+2}, \quad \bT^{b, \ell}_{1m}(\lambda, \xi') = A^\ell \bT^{b,0}_{1m}(\lambda, \xi')
\end{align*}
for $k=1,2$ and $m=3,4$. We see that $\bT^{b, \ell}_{1j}(\lambda, \xi') \in \BN_{-2+\ell}$
for $\ell=1,2, 3$. 
Then, for $(\lambda, \lambda^{1/2}\nabla, \nabla^2)$, 
using  \eqref{repr:6.1} and \eqref{repr:6.3}, we may write
\begin{align}
&(\lambda, \lambda^{\frac{1}{2}}\nabla', \nabla'')\CT^{b}_1(\lambda)\bg \nonumber \\
&\quad = \int^\infty_0 \CF^{-1}_{\xi'}\Bigl[ \Bigl(\sum_{j=1}^4 \CP_j(x_N, y_N) 
(\lambda, \lambda^{\frac{1}{2}}i\xi', (i\xi')^2)\bT^{b, 0}_{1j}(\lambda, \xi')\Bigr)\CF'[\bg](\xi', y_N)\Bigr](x')\, dy_N,
\\
&(\lambda^{\frac{1}{2}}, \nabla')\pd_N\CT^b_1(\lambda)\bg  \nonumber\\
&\quad  =  (-1)
\int^\infty_0 \CF^{-1}_{\xi'}\Bigl[ \Bigl(\sum_{j=1}^4 \CP_j(x_N, y_N) 
(\lambda^{\frac{1}{2}}, i\xi')\bT^{b, 1}_{1j}(\lambda, \xi')\Bigr)\CF'[\bg](\xi', y_N)\Bigr](x')\, dy_N,  \nonumber\\
&\pd_N^2 \CT^b_1(\lambda)\bg 
=
\int^\infty_0 \CF^{-1}_{\xi'}\Bigl[ \Bigl(\sum_{j=1}^4 \CP_j(x_N, y_N) 
\bT^{b, 2}_{1j}(\lambda, \xi')\Bigr)\CF'[\bg](\xi', y_N)\Bigr](x')\, dy_N.
\label{repdiff.0}
\end{align}
Since $(\lambda, \lambda^{1/2}i\xi', (i\xi')^2)\bT^b_{1j}(\lambda, \xi')$, 
$(\lambda^{1/2}, i\xi')\bT^{b,1}_{1j}$ and $\bT^{b,2}_{1j}(\lambda, \xi')$
belong to $\BM_0$, by Proposition \ref{prop:2}, we have
\begin{equation}\label{est:6.1}
\|(\lambda, \lambda^{\frac{1}{2}}\nabla, \nabla^2)\CT^b_1(\lambda)\bg\|_{L_q(\HS)}
\leq C\|\bg\|_{L_q(\HS)}. 
\end{equation}
\par
Next, we consider $H^1_q$-$H^1_q$ estimate.   To this end, 
we use the formulas:
\begin{equation}\label{diffm:3}\begin{aligned}
&\CP_1(x_N, y_N) = -B^{-1}\pd_{y_N}\CP(x_N, y_N), 
\quad
\CP_2(x_N, y_N) = -A^{-1}\pd_{y_N}(\CP_2(x_N, y_N) - \CP_1(x_N, y_N)),\\
&\CP_3(x_N, y_N) = -B^{-1}\CP_3(x_N, y_N), \quad
\CP_4(x_N, y_N) = A^{-1}\pd_{y_N}(\CP_4(x_N, y_N) - \CP_3(x_N, y_N)),
\end{aligned}\end{equation}
which follows from 
\begin{equation}\label{diffm.2}
e^{-Bt_N} = \frac{1}{B}\pd_N e^{-B(y_N)}, \quad
\CM(y_N) = \pd_N\Bigl\{\frac{-1}{A} \CM(y_N) + \frac{1}{AB}e^{-By_N}\Bigr\}.
\end{equation}
Since  $\bg \in C^\infty_0(\HS)^N$, by integration by parts, 
we rewrite the formulas in \eqref{repr:6.1} and \eqref{repr:6.3} as follows: 
\begin{equation}
\label{repdiff.1}
\pd_{N}^\ell\CT^b_1(\lambda)\bg
= \int^\infty_0 \CF^{-1}_{\xi'}\Bigl[\Bigl(\sum_{j=1}^4 \CP_j(x_N, y_N)\tilde\bT^{b, \ell}_{1j}(\lambda, \xi')\Bigr)
\CF'[\pd_{N}\bg](\xi', y_N)\Bigr](x')\,dy_N.
\end{equation}
Here, we have set
\begin{align*}
&\tilde\bT^{b,\ell}_{11} = B^{-1}\bT^{b,\ell}_{11} - A^{-1}\bT^{b, \ell}_{12}, \quad
\tilde \bT^{b, \ell}_{12} = A^{-1}\bT^{b,2}_{12}, \nonumber \\
&\tilde\bT^{b,\ell}_{13} = B^{-1}\bT^{b, \ell}_{13} + A^{-1}\bT^{b, \ell}_{1 4}, \quad 
\tilde \bT^{b, \ell}_{14} = -A^{-1}\bT^{b, \ell}_{14}.
\end{align*}
Since $\bT^{b, \ell}_{1j} \in \BN_{-2+\ell}$, we see that $\tilde \bT^{b, \ell}_{1j} \in 
\BN_{-3+\ell}$ for $\ell=0,1,2,3$. \par
Using \eqref{repdiff.1}, we may write 
\allowdisplaybreaks
\begin{align}
&\nabla'(\lambda, \lambda^{\frac{1}{2}}\nabla', \nabla'')\CT^{b,0}_1(\lambda)\bg \nonumber \\
&= \int^\infty_0 \CF^{-1}_{\xi'}\Bigl[ \Bigl(\sum_{j=1}^4 \CP_j(x_N, y_N) 
i\xi'(\lambda, \lambda^{\frac{1}{2}}i\xi', (i\xi')^2)\tilde\bT^{b, 0}_{1j}(\lambda, \xi')\Bigr)\CF'[\pd_N\bg](\xi', y_N)\Bigr](x')\, dy_N,
\nonumber \\
&\pd_N(\lambda, \lambda^{\frac{1}{2}}\nabla', \nabla'')\CT^{b,0}_1(\lambda)\bg \nonumber \\
&= \int^\infty_0 \CF^{-1}_{\xi'}\Bigl[ \Bigl(\sum_{j=1}^4 \CP_j(x_N, y_N) 
(\lambda, \lambda^{\frac{1}{2}}i\xi', (i\xi')^2)\tilde\bT^{b, 1}_{1j}(\lambda, \xi')\Bigr)\CF'[\pd_N\bg](\xi', y_N)\Bigr](x')\, dy_N,\\
&\nabla'(\lambda^{\frac{1}{2}}, \nabla')\pd_N\CT^b_1(\lambda)\bg  \nonumber\\
& =  -
\int^\infty_0 \CF^{-1}_{\xi'}\Bigl[ \Bigl(\sum_{j=1}^4 \CP_j(x_N, y_N) 
i\xi'(\lambda^{\frac{1}{2}}, i\xi')\tilde\bT^{b, 1}_{1j}(\lambda, \xi')\Bigr)
\CF'[\pd_N\bg](\xi', y_N)\Bigr](x')\, dy_N,  \nonumber\\
&\pd_N(\lambda^{\frac{1}{2}}, \nabla')\pd_N\CT^b_1(\lambda)\bg  \nonumber\\
& =  -
\int^\infty_0 \CF^{-1}_{\xi'}\Bigl[ \Bigl(\sum_{j=1}^4 \CP_j(x_N, y_N) 
(\lambda^{\frac{1}{2}}, i\xi')\tilde\bT^{b, 2}_{1j}(\lambda, \xi')\Bigr)\CF'[\pd_N\bg](\xi', y_N)\Bigr](x')\, dy_N,  \nonumber\\
&\nabla'\pd_N^2 \CT^b_1(\lambda)\bg 
=
\int^\infty_0 \CF^{-1}_{\xi'}\Bigl[ \Bigl(\sum_{j=1}^4 \CP_j(x_N, y_N) 
i\xi'\tilde\bT^{b, 2}_{1j}(\lambda, \xi')\Bigr)\CF'[\pd_N\bg](\xi', y_N)\Bigr](x')\, dy_N,
\nonumber \\
&\pd_N\pd_N^2 \CT^b_1(\lambda)\bg 
=
\int^\infty_0 \CF^{-1}_{\xi'}\Bigl[ \Bigl(\sum_{j=1}^4 \CP_j(x_N, y_N) 
\tilde\bT^{b, 3}_{1j}(\lambda, \xi')\Bigr)\CF'[\pd_N\bg](\xi', y_N)\Bigr](x')\, dy_N.
\label{repdiff.2}
\end{align}
Since the following symbols:
\begin{gather*}
i\xi'(\lambda, \lambda^{\frac{1}{2}}i\xi', (i\xi')^2)\tilde\bT^{b, 0}_{1j}(\lambda, \xi'), \quad 
(\lambda, \lambda^{\frac{1}{2}}i\xi', (i\xi')^2)\tilde\bT^{b, 1}_{1j}(\lambda, \xi'), \quad 
i\xi'(\lambda^{\frac{1}{2}}, i\xi')\tilde\bT^{b, 1}_{1j}(\lambda, \xi'), \\
(\lambda^{\frac{1}{2}}, i\xi')\tilde\bT^{b, 2}_{1j}(\lambda, \xi'), \quad
i\xi'\tilde\bT^{b, 2}_{1j}(\lambda, \xi'), \quad \tilde\bT^{b, 3}_{1j}(\lambda, \xi')
\end{gather*}
belong to $\BM_0(\Lambda_{\epsilon, \lambda_0})$, by Proposition \ref{prop:2}, 
we have
\begin{equation}\label{est:6.2}
\|\nabla (\lambda, \lambda^{\frac{1}{2}}\nabla, \nabla^2)\CT^b_1(\lambda)\bg\|_{L_q(\HS)}
\leq C\|\pd_N\bg\|_{L_q(\HS)} \leq C\|\nabla \bg\|_{L_q(\HS)}
\end{equation}
for any $\lambda \in \Lambda_{\epsilon, \lambda_0}$ and $\bg \in C^\infty_0(\HS)^N$. \par
We also have
\begin{align}
&\lambda^{\frac{1}{2}}(\lambda, \lambda^{\frac{1}{2}}\nabla', \nabla'')\CT^{b,0}_1(\lambda)\bg \nonumber \\
&\quad = \int^\infty_0 \CF^{-1}_{\xi'}\Bigl[ \Bigl(\sum_{j=1}^4 \CP_j(x_N, y_N) 
\lambda^{\frac{1}{2}}(\lambda, \lambda^{\frac{1}{2}}i\xi', (i\xi')^2)\tilde\bT^{b, 0}_{1j}(\lambda, \xi')\Bigr)
\CF'[\pd_N\bg](\xi', y_N)\Bigr](x')\, dy_N,
\nonumber \\
&\lambda^{\frac{1}{2}}(\lambda^{\frac{1}{2}}, \nabla')\pd_N\CT^b_1(\lambda)\bg  \nonumber\\
&\quad  =  -
\int^\infty_0 \CF^{-1}_{\xi'}\Bigl[ \Bigl(\sum_{j=1}^4 \CP_j(x_N, y_N) 
\lambda^{\frac{1}{2}}(\lambda^{\frac{1}{2}}, i\xi')\tilde\bT^{b, 1}_{1j}(\lambda, \xi')\Bigr)
\CF'[\pd_N\bg](\xi', y_N)\Bigr](x')\, dy_N,  \nonumber\\
&\lambda^{\frac{1}{2}}\pd_N^2 \CT^b_1(\lambda)\bg 
=
\int^\infty_0 \CF^{-1}_{\xi'}\Bigl[ \Bigl(\sum_{j=1}^4 \CP_j(x_N, y_N) 
\lambda^{\frac{1}{2}}\tilde\bT^{b, 2}_{1j}(\lambda, \xi')\Bigr)\CF'[\pd_N\bg](\xi', y_N)\Bigr](x')\, dy_N. 
\label{repdiff.3}
\end{align}
Since the following symbols:
\begin{gather*}
\lambda^{\frac{1}{2}}(\lambda, \lambda^{\frac{1}{2}}i\xi', (i\xi')^2)\tilde\bT^{b, 0}_{1j}(\lambda, \xi'), \quad 
\lambda^{\frac{1}{2}}(\lambda^{\frac{1}{2}}, i\xi')\tilde\bT^{b, 1}_{1j}(\lambda, \xi'), \quad 
\lambda^{\frac{1}{2}}\tilde\bT^{b, 2}_{1j}(\lambda, \xi')
\end{gather*}
belong to $\BM_0(\Lambda_{\epsilon, \lambda_0})$, by Proposition \ref{prop:2}, 
we have
\begin{equation}\label{est:6.3}
\|(\lambda, \lambda^{\frac{1}{2}}\nabla, \nabla^2)\CT^b_1(\lambda)\bg\|_{L_q(\HS)}
 \leq C|\lambda|^{-\frac{1}{2}}\|\nabla \bg\|_{L_q(\HS)}
\end{equation}
for any $\lambda \in \Lambda_{\epsilon, \lambda_0}$ and $\bg \in C^\infty_0(\HS)^N$. \par 


Now, we consider the dual operator $\pd_N^\ell \CT^b_1(\lambda)^*$
 of $\pd_N^\ell\CT^b_1(\lambda)$ acting on
$\bh \in C^\infty_0(\HS)^N$, which satisfies the equality:
$|(\pd_N^\ell\CT^b_1(\lambda)^*\bg, \bh)| = |(\bg, \CT^b_1(\lambda)^*\bh)|$. In fact, 
from \eqref{repr:6.1} and \eqref{repr:6.3} by Fubini's theorem and 
Plancherel's theorem, we have
\begin{align*}
&(-1)^\ell(\pd_N^\ell \CT^b_1(\lambda)\bg, \bh)\\
& = \int^\infty_0\int_{\BR^{N-1}}
\Bigl(\int^\infty_0 \CF^{-1}_{\xi'}\Bigl[ \Bigl(\sum_{j=1}^4 \CP_j(x_N, y_N) 
\bT^{b, \ell}_{1j}(\lambda, \xi')\Bigr)\CF'[\bg](\xi', y_N)\Bigr](x')\, dy_N\Bigr)\bh(x', x_N)\,dx'dx_N\\
&=\int^\infty_0\int^\infty_0(\int_{\BR^{N-1}}\Bigl(\sum_{j=1}^4 \CP_j(x_N, y_N) 
\bT^{b, \ell}_{1j}(\lambda, \xi')\Bigr)\CF'[\bg](\xi', y_N)\CF^{-1}_{\xi'}[\bh](\xi', x_N)\,d\xi'\Bigr)\,dy_Ndx_N
\\
& = \int^\infty_0\int_{\BR^{N-s}} \bg(y', y_N)\Bigl(
\int^\infty_0\CF'\Bigl[\Bigl(\sum_{j=1}^4 \CP_j(x_N, y_N) 
\bT^{b, \ell}_{1j}(\lambda, \xi')\Bigr)\CF^{-1}_{\xi'}[\bh](\xi', x_N)\Bigr](y')\,dx_N\Bigr)\,dy'dy_N,
\end{align*}
which yields 
\begin{equation}\label{dual:6.1}\pd_N^\ell\CT^b_1(\lambda)^*\bh = 
\int^\infty_0\CF'\Bigl[\Bigl(\sum_{j=1}^4 \CP_j(x_N, y_N) 
\bT^{b, \ell}_{1j}(\lambda, \xi')\Bigr)\CF^{-1}_{\xi'}[\bh](\xi', x_N)\Bigr](y')\,dx_N.
\end{equation}
Namely, $\pd_N^\ell\CT^b_1(\lambda)^*\bh$ is obtained by exchanging $\CF^{-1}_{\xi'}$ and 
$\CF'$ in the representation of $\pd_N^\ell\CT^b_1(\lambda)$.  Thus, employing the completely 
same argument as in proving \eqref{est:6.1}, \eqref{est:6.2}, and \eqref{est:6.3}, 
we have 
\begin{equation}\label{est:6.4}\begin{aligned}
\|(\lambda, \lambda^{\frac{1}{2}}\nabla, \nabla^2)\CT^b_1(\lambda)^*\bh\|_{L_{q'}(\HS)}
&\leq C\|\bh\|_{L_{q'}(\HS)}, \\
\|(\lambda, \lambda^{\frac{1}{2}}\nabla, \nabla^2)\CT^b_1(\lambda)^*\bh\|_{H^1_{q'}(\HS)}
&\leq C\|\bh\|_{H^1_{q'}(\HS)}, \\
\|(\lambda, \lambda^{\frac{1}{2}}\nabla, \nabla^2)\CT^b_1(\lambda)^*\bh\|_{L_{q'}(\HS)}
&\leq C|\lambda|^{-\frac{1}{2}}\|\bh\|_{H^1_{q'}(\HS)}.
\end{aligned}\end{equation}
From \eqref{est:6.1}, \eqref{est:6.2}, \eqref{est:6.3}, and \eqref{est:6.4} it follows that 
$\bT^b_1(\lambda)$ satisfies Assumption \ref{assump:4.1}, and so by
Theorem \ref{spectralthm:1} we have obtained 
\begin{equation}\label{thm:6.1.7}\begin{aligned}
\|(\lambda, \lambda^{\frac{1}{2}}\nabla, \nabla^2)\CT^b_1(\lambda)\bg\|_{B^s_{q,r}(\HS)}
&\leq C\|\bg\|_{B^s_{q,r}(\HS)}, \\
\|(\lambda, \lambda^{\frac{1}{2}}\nabla, \nabla^2)\CT^b_1(\lambda)\bg\|_{B^s_{q,r}(\HS)}
&\leq C|\lambda|^{-\frac{\sigma}{2}}\|\bg\|_{B^{s+\sigma}_{q,r}(\HS)}
\end{aligned}\end{equation}
for any $\lambda \in \Lambda_{\epsilon, \lambda_0}$ and $\bg \in C^\infty_0(\HS)^N$. 
In particular, we have obtained \eqref{thm:6.1.3}. \par

Now, we consider $\pd_\lambda \bT^b_1(\lambda)$, which is represented by
\begin{align*}
\pd_\lambda\CT^b_1(\lambda)\bg &= \int^\infty_0 \CF^{-1}_{\xi'}\Bigl[ \Bigl(\sum_{j=1}^4
B^2(\pd_\lambda \CP_j(x_N, y_N) )
B^{-2}\bT^{b,0}_{1j}(\lambda, \xi')\Bigr)\CF'[\bg](\xi', y_N)\Bigr](x')\, dy_N \\
& + 
\int^\infty_0 \CF^{-1}_{\xi'}\Bigl[ \Bigl(\sum_{j=1}^4 \CP_j(x_N, y_N) 
(\pd_\lambda\bT^{b, 0}_{1j}(\lambda, \xi'))\Bigr)\CF'[\bg](\xi', y_N)\Bigr](x')\, dy_N
\end{align*}
as follows from \eqref{repr:6.1}.  Moreover, from \eqref{repr:6.3} and \eqref{repdiff.1}, 
we have
\begin{align*}
\pd_N^\ell\pd_\lambda  \CT^b_1(\lambda)\bg& =  (-1)^\ell
\int^\infty_0 \CF^{-1}_{\xi'}\Bigl[ \Bigl(\sum_{j=1}^4B^2(\pd_\lambda \CP_j(x_N, y_N) )
B^{-2}\bT^{b,\ell}_{1j}(\lambda, \xi')\Bigr)\CF'[\bg](\xi', y_N)\Bigr](x')\, dy_N \\
&+ (-1)^\ell
\int^\infty_0 \CF^{-1}_{\xi'}\Bigl[ \Bigl(\sum_{j=1}^4 \CP_j(x_N, y_N) )
(\pd_\lambda \bT^{b,\ell}_{1j}(\lambda, \xi'))\Bigr)\CF'[\bg](\xi', y_N)\Bigr](x')\, dy_N;
\\
\pd_{N}^\ell\pd_\lambda\CT^b_1(\lambda)\bg
&= \int^\infty_0 \CF^{-1}_{\xi'}\Bigl[\Bigl(\sum_{j=1}^4 B^2(\pd_\lambda\CP_j(x_N, y_N))
B^{-2}\tilde\bT^{b, \ell}_{1j}(\lambda, \xi')\Bigr)
\CF'[\pd_{N}\bg](\xi', y_N)\Bigr](x')\,dy_N \\
&+  \int^\infty_0 \CF^{-1}_{\xi'}\Bigl[\Bigl(\sum_{j=1}^4 
\CP_j(x_N, y_N)(\pd_\lambda \tilde\bT^{b, \ell}_{1j}(\lambda, \xi'))\Bigr)
\CF'[\pd_{N}\bg](\xi', y_N)\Bigr](x')\,dy_N.
\end{align*}
If we write
\begin{align*}
\lambda \pd_N^\ell\pd_\lambda  \CT^b_1(\lambda)\bg& =  (-1)^\ell
\int^\infty_0 \CF^{-1}_{\xi'}\Bigl[ \Bigl(\sum_{j=1}^4B^2(\pd_\lambda \CP_j(x_N, y_N) )
\lambda B^{-2}\bT^{b,\ell}_{1j}(\lambda, \xi')\Bigr)\CF'[\bg](\xi', y_N)\Bigr](x')\, dy_N \\
&+ (-1)^\ell
\int^\infty_0 \CF^{-1}_{\xi'}\Bigl[ \Bigl(\sum_{j=1}^4 \CP_j(x_N, y_N) )
\lambda (\pd_\lambda \bT^{b,\ell}_{1j}(\lambda, \xi'))\Bigr)\CF'[\bg](\xi', y_N)\Bigr](x')\, dy_N;\\
\lambda \pd_{N}^\ell\pd_\lambda \CT^b_1(\lambda)\bg
&= \int^\infty_0 \CF^{-1}_{\xi'}\Bigl[\Bigl(\sum_{j=1}^4 B^2(\pd_\lambda\CP_j(x_N, y_N))
\lambda B^{-2}\tilde\bT^{b, \ell}_{1j}(\lambda, \xi')\Bigr)
\CF'[\pd_{N}\bg](\xi', y_N)\Bigr](x')\,dy_N \\
&+  \int^\infty_0 \CF^{-1}_{\xi'}\Bigl[\Bigl(\sum_{j=1}^4 
\CP_j(x_N, y_N)\lambda (\pd_\lambda \tilde\bT^{b, \ell}_{1j}(\lambda, \xi'))\Bigr)
\CF'[\pd_{N}\bg](\xi', y_N)\Bigr](x')\,dy_N
\end{align*}
for $\ell=0, \ldots, 3$, then using the facts that 
$\lambda B^{-2}\bT^{b,\ell}_{1j}(\lambda, \xi') \in \BM_{-2+\ell}(\Lambda_{\epsilon, \lambda_0})$,  
$\lambda (\pd_\lambda \bT^{b, \ell}_{1j}(\lambda, \xi'))\in \BM_{-2+\ell}(\Lambda_{\epsilon, \lambda_0})$,
$\lambda B^{-2}\tilde \bT^{b,\ell}_{1j}(\lambda, \xi') \in \BM_{-3+\ell}(\Lambda_{\epsilon, \lambda_0})$,
and $\lambda (\pd_\lambda\tilde\bT^{b,\ell}_{1j}(\lambda, \xi'))
 \in \BM_{-3+\ell}(\Lambda_{\epsilon, \lambda_0})$
for $\ell=0,1,2,3$ and employing the same argument as in the proof of 
\eqref{est:6.1}, \eqref{est:6.2}, \eqref{est:6.3}, and \eqref{est:6.4},
by  Propositions \ref{prop:2} and \ref{prop:3}, 
we have
\begin{equation}\label{est:6.5}\begin{aligned}
\|(\lambda, \lambda^{\frac{1}{2}}\nabla, \nabla^2)\pd_\lambda\CT^b_1(\lambda)\bg\|_{L_q(\HS)}
\leq C|\lambda|^{-1}\|\bg\|_{L_q(\HS)}, \\
\|(\lambda, \lambda^{\frac{1}{2}}\nabla, \nabla^2)\pd_\lambda\CT^b_1(\lambda)\bg\|_{H^1_q(\HS)}
\leq C|\lambda|^{-1}\|\bg\|_{H^1_q(\HS)}.
\end{aligned}\end{equation}
Moreover, writing 
\begin{align*}
&\lambda^{\frac{1}{2}} \pd_N^\ell\pd_\lambda  \CT^b_1(\lambda)\bg \\
&\quad =  (-1)^\ell
\int^\infty_0 \CF^{-1}_{\xi'}\Bigl[ \Bigl(\sum_{j=1}^4B^2(\pd_\lambda \CP_j(x_N, y_N) )
\lambda^{\frac{1}{2}} B^{-2}\bT^{b,\ell}_{1j}(\lambda, \xi')\Bigr)\CF'[\bg](\xi', y_N)\Bigr](x')\, dy_N \\
&\quad + (-1)^\ell
\int^\infty_0 \CF^{-1}_{\xi'}\Bigl[ \Bigl(\sum_{j=1}^4 \CP_j(x_N, y_N) )
\lambda^{\frac{1}{2}} (\pd_\lambda \bT^{b,\ell}_{1j}(\lambda, \xi'))\Bigr)\CF'[\bg](\xi', y_N)\Bigr](x')\, dy_N
\end{align*}
and using the facts that 
$\lambda^{1/2} B^{-2}\bT^{b,\ell}_{1j}(\lambda, \xi') \in \BM_{-3+\ell}(\Lambda_{\epsilon, \lambda_0})$
and $\lambda^{1/2} (\pd_\lambda \bT^{b, \ell}_{1j}(\lambda, \xi'))
\in \BN_{-3+\ell}(\Lambda_{\epsilon, \lambda_0})$ for $\ell=0, 1, 2, 3$,
by Propositions \ref{prop:2} and \ref{prop:3}, we have
\begin{equation}\label{est:6.6} 
\|(\lambda, \lambda^{\frac{1}{2}}\nabla, \nabla^2)\pd_\lambda\CT^b_1(\lambda)\bg\|_{H^1_q(\HS)}
\leq C|\lambda|^{-\frac{1}{2}}\|\bg\|_{L_q(\HS)},
\end{equation}
for any $\lambda \in \Lambda_{\epsilon, \lambda_0}$ and $\bg \in C^\infty_0(\HS)^N$.
\par
Employing the same argument as in the proof of \eqref{dual:6.1}, we see that 
the dual operators $\pd_N^\ell\pd_\lambda \CT^b_1(\lambda)^*$  
of $\pd_N^\ell\pd_\lambda \CT^b_1(\lambda)$ are defined by 
\begin{align*}
 \pd_N^\ell\pd_\lambda  \CT^b_1(\lambda)^*\bh& = 
\int^\infty_0 \CF'\Bigl[ \Bigl(\sum_{j=1}^4B^2(\pd_\lambda \CP_j(x_N, y_N) )
 B^{-2}\bT^{b,\ell}_{1j}(\lambda, \xi')\Bigr)\CF^{-1}_{\xi'}[\bh](\xi', y_N)\Bigr](x')\, dy_N \\
&+ 
\int^\infty_0 \CF'\Bigl[ \Bigl(\sum_{j=1}^4 \CP_j(x_N, y_N) )
 (\pd_\lambda \bT^{b,\ell}_{1j}(\lambda, \xi'))\Bigr)\CF^{-1}_{\xi'}[\bh](\xi', y_N)\Bigr](x')\, dy_N.
\end{align*}
Employing the same argument as in the proof of \eqref{repdiff.1}, we have 
\begin{align*}\pd_{N}^\ell\pd_\lambda \CT^b_1(\lambda)^*\bh
&= \int^\infty_0 \CF'\Bigl[\Bigl(\sum_{j=1}^4 B^2(\pd_\lambda\CP_j(x_N, y_N))
B^{-2}\tilde\bT^{b, \ell}_{1j}(\lambda, \xi')\Bigr)
\CF^{-1}_{\xi'}[\pd_{N}\bh](\xi', y_N)\Bigr](x')\,dy_N \\
&+  \int^\infty_0 \CF'\Bigl[\Bigl(\sum_{j=1}^4 
\CP_j(x_N, y_N) (\pd_\lambda \tilde\bT^{b, \ell}_{1j}(\lambda, \xi'))\Bigr)
\CF^{-1}_{\xi'}[\pd_{N}\bh](\xi', y_N)\Bigr](x')\,dy_N.
\end{align*}
Thus, we have
\begin{equation}\label{est:6.7}\begin{aligned}
\|(\lambda, \lambda^{\frac{1}{2}}\nabla, \nabla^2)\pd_\lambda \CT^b_1(\lambda)^*\bh\|_{L_{q'}(\HS)}
\leq C|\lambda|^{-1}\|\bh\|_{L_{q'}(\HS)}, \\
\|(\lambda, \lambda^{\frac{1}{2}}\nabla, \nabla^2)\pd_\lambda \CT^b_1(\lambda)^*\bh\|_{H^1_{q'}(\HS)}
\leq C|\lambda|^{-1}\|\bh\|_{H^1_{q'}(\HS)}, \\
\|(\lambda, \lambda^{\frac{1}{2}}\nabla, \nabla^2)\pd_\lambda \CT^b_1(\lambda)^*\bh\|_{H^1_{q'}(\HS)}
\leq C|\lambda|^{-\frac{1}{2}}\|\bh\|_{L_{q'}(\HS)}.
\end{aligned}\end{equation}
From \eqref{est:6.5}, \eqref{est:6.6}, and \eqref{est:6.7} it follows that 
$\pd_\lambda \CT^b_1(\lambda)$ satisfies Assumption \ref{assump:4.2}, 
and so by Theorem \ref{spectralthm:1}, we have
\begin{equation}\label{thm:6.1.8}\begin{aligned}
\|\pd_\lambda \CT^b_1(\lambda)\bg\|_{B^s_{q,r}(\HS)} &\leq C|\lambda|^{-1}\|\bg\|_{B^s_{q,r}(\HS)},
\\
\|\pd_\lambda \CT^b_1(\lambda)\bg\|_{B^s_{q,r}(\HS)} 
&\leq C|\lambda|^{-(1-\frac{\sigma}{2})}\|\bg\|_{B^{s-\sigma}_{q,r}(\HS)}
\end{aligned}\end{equation}
for any $\lambda \in \Lambda_{\epsilon, \lambda_0}$ and $\bg \in C^\infty_0(\HS)^N$.
In particular, we have \eqref{thm:6.1.4}. \par

Now, we consider $\CT^b_2(\lambda)$ and we 
shall prove \eqref{thm:6.1.5} and \eqref{thm:6.1.6}. To this end, we introduce
the class of multipliers $\BN^d_k$ defined by 
\begin{align*} 
\BN_k^d  = &\bigl\{m(\lambda, \xi') \in \BM_k(\Lambda_{\epsilon, \lambda_0})
\mid \text{there hold} \\
&\hskip1cm
|D_{\xi'}^{\delta'} m(\lambda, \xi')| \leq C|\lambda|^{-1}(|\lambda|^{\frac{1}{2}}
+|\xi'|)^{k-|\delta'|}
\\
&\hskip1cm
|D_{\xi'}^{\delta'} (\pd_\lambda m(\lambda, \xi'))| \leq C|\lambda|^{-2}
(|\lambda|^{\frac{1}{2}}+|\xi'|)^{k-|\delta'|} 
\\
&\text{for any multi-index $\delta' \in \BN_0^{N-1}$, 
$\lambda \in \Lambda_{\epsilon, \lambda_0}$
and $\xi' \in \BR^{N-1}$} \bigr\}.
\end{align*}
For $m_1(\lambda, \xi') \in \BN_k$
and $m_2(\lambda, \xi') \in \BN^d_\ell$, we have $m_1(\lambda, \xi')m_2(\lambda, \xi')
\in \BN^d_{k+\ell}$. For $m(\lambda, \xi') \in \BN_{-2}$, we have 
$q(\lambda)\lambda^{-1}m(\lambda, \xi') \in \BN^d_{-2}$.  
From Lemma \ref{lem:6.2} and Corollary \ref{lem:cor:6.2}, it follows that 
$K_2(\lambda)\lambda^{-1}\in \BN^d_{-1}
$, $K_3(\lambda)\lambda^{-1} \in \BN^d_{-1}$, and so 
$K_2(\lambda)\lambda^{-1}m(\lambda, \xi') \in \BN^d_{-2}$ and 
$K_3(\lambda)\lambda^{-1}m(\lambda, \xi') \in \BN^d_{-2}$ for 
$m(\lambda, \xi') \in \BN_{-1}$. From these observations, 
we see that all the following symbols appearing in the definition of $\CT^b_{2J}(\lambda)$: 
\begin{gather*}
\frac{q(\lambda)}{\lambda}\frac{i\xi_ji\xi_k}{(A+B)AB^2}, \quad \frac{q(\lambda)}{\lambda}\frac{i\xi_j}{(A+B)B^2},
\quad \frac{K_2(\lambda)}{\lambda}\frac{i\xi_j}{B^3}, \quad \frac{K_3(\lambda)}{\lambda}
\frac{i\xi_j i\xi_k|\xi'|^2}{ (A+B)AB^3},
\\
\frac{K_3(\lambda)}{\lambda}\frac{i\xi_j|\xi'|^2}{ (A+B)B^3}, 
\quad \frac{K_2(\lambda)}{\lambda} \frac{i\xi_k A}{ B^3},
\quad 
\frac{K_3(\lambda)}{\lambda}\frac{|\xi'|^2A}{(A+B)B^3}
\end{gather*}
belong to $\BN^d_{-2}$. Thus, we may assert that 
there exist
four $N\times N$ matrices of $\BN^d_{-2}$ symbols $\bT^{b,0}_{2j}$ ($j=1, 2, 3, 4$)
such that $\CT^b_2(\lambda)$ is represented by 
\begin{equation}
\CT^b_2(\lambda)\bg = \int^\infty_0\CF^{-1}_{\xi'}\Bigl[\Bigl(
\sum_{j=1}^4 \CP_j(x_N, y_N)\bT^{b,0}_{2j}(\lambda, \xi')\Bigr)
\CF'[\bg](\xi', y_N)\Bigr](x')\,dy_N.
\label{repr:6.2}
\end{equation}
Employing the same arguments as in \eqref{repr:6.3} and \eqref{repdiff.1}, we have
\begin{equation}\label{repr:6.3.2}
\pd_N^\ell \CT^b_2(\lambda)\bg =  (-1)^\ell
\int^\infty_0 \CF^{-1}_{\xi'}\Bigl[ \Bigl(\sum_{j=1}^4 \CP_j(x_N, y_N) 
\bT^{b,\ell}_{2 j}(\lambda, \xi')\Bigr)\CF'[\bg](\xi', y_N)\Bigr](x')\, dy_N
\end{equation}
for $\ell=1,2$, where we have set 
\begin{align*}
\bT^{b, \ell}_{2k}(\lambda, \xi')& = B^{\ell}\bT^{b}_{1k}(\lambda, \xi')+\frac{A^\ell-B^{\ell}}{A-B}B
\bT^{b}_{2k+2}, \quad \bT^{b, \ell}_{2\ell}(\lambda, \xi') = A^\ell \bT^b_{2\ell}(\lambda, \xi')
\end{align*}
for $k=1,2$ and $\ell=3,4$,  and 
\begin{equation}
\label{repdiff.1.2}
\pd_{N}^\ell\CT^b_2(\lambda)\bg
= \int^\infty_0 \CF^{-1}_{\xi'}\Bigl[\Bigl(\sum_{j=1}^4 \CP_j(x_N, y_N)\tilde\bT^{b, \ell}_{2j}(\lambda, \xi')\Bigr)
\CF'[\pd_{N}\bg](\xi', y_N)\Bigr](x')\,dy_N.
\end{equation}
Here, we have set
\begin{align*}
&\tilde\bT^{b,\ell}_{21} = B^{-1}\bT^{b,\ell}_{21} - A^{-1}\bT^{b, \ell}_{22}, \quad
\tilde \bT^{b, \ell}_{22} = A^{-1}\bT^{b,2}_{22}, \nonumber \\
&\tilde\bT^{b,\ell}_{23} = B^{-1}\bT^{b, \ell}_{23} + A^{-1}\bT^{b, \ell}_{24}, \quad 
\tilde \bT^{b, \ell}_{24} = -A^{-1}\bT^{b, \ell}_{24}.
\end{align*}
Since $\bT^{b,\ell}_{2k}(\lambda, \xi') \in \BN^d_{-2+\ell}$ ($\ell=0,1,2,3$), applying  Proposition \ref{prop:2}
to the formulas in \eqref{repr:6.3.2} yields 
\begin{align}
\|(\lambda, \lambda^{\frac{1}{2}}\nabla', \nabla'')\CT^b_2(\lambda)\bg\|_{L_q(\HS)}
& \leq C|\lambda|^{-1}\|\bg\|_{L_q(\HS)}, \nonumber\\
\|(\lambda^{\frac{1}{2}}, \nabla')\pd_N\CT^b_2(\lambda)\bg\|_{L_q(\HS)}
& \leq C|\lambda|^{-1}\|\bg\|_{L_q(\HS)}, \nonumber \\
\|\pd_N^2\CT^b_2(\lambda)\bg\|_{L_q(\HS)}
& \leq C|\lambda|^{-1}\|\bg\|_{L_q(\HS)}. \label{est:6.1.2*}
\end{align}
Since $\tilde\bT^{b,\ell}_{2k}(\lambda, \xi') \in \BN^d_{-3+\ell}$ ($\ell=0,1,2,3$), 
applying Proposition \ref{prop:2} to the formulas in \eqref{repdiff.1.2} yields 
\begin{align}
\|\nabla'(\lambda, \lambda^{\frac{1}{2}}\nabla', \nabla'')\CT^b_2(\lambda)\bg\|_{L_q(\HS)}
& \leq C|\lambda|^{-1}\|\pd_N\bg\|_{L_q(\HS)}, \nonumber \\
\|\pd_N(\lambda, \lambda^{\frac{1}{2}}\nabla', \nabla'')\CT^b_2(\lambda)\bg\|_{L_q(\HS)}
& \leq C|\lambda|^{-1}\|\pd_N\bg\|_{L_q(\HS)}, \nonumber \\
\|\nabla'(\lambda^{\frac{1}{2}}, \nabla')\pd_N\CT^b_2(\lambda)\bg\|_{L_q(\HS)}
& \leq C|\lambda|^{-1}\|\pd_N\bg\|_{L_q(\HS)}, \nonumber \\
\|\pd_N(\lambda^{\frac{1}{2}}, \nabla')\pd_N\CT^b_2(\lambda)\bg\|_{L_q(\HS)}
& \leq C|\lambda|^{-1}\|\pd_N\bg\|_{L_q(\HS)}, \nonumber\\
\|\nabla'\pd_N^2\CT^b_2(\lambda)\bg\|_{L_q(\HS)}
& \leq C|\lambda|^{-1}\|\pd_N\bg\|_{L_q(\HS)}, \nonumber \\
\|\pd_N^3\CT^b_2(\lambda)\bg\|_{L_q(\HS)}
& \leq C|\lambda|^{-1}\|\pd_N\bg\|_{L_q(\HS)}.\label{est:6.1.3*}
\end{align}
Combining these estimates yields
\begin{align}
\|(\lambda, \lambda^{\frac{1}{2}}\nabla, \nabla^2)\CT^b_2(\lambda)\bg\|_{L_q(\HS)}
& \leq C|\lambda|^{-1}\|\bg\|_{L_q(\HS)}, \label{est:6.1.2} \\
\|(\lambda, \lambda^{\frac{1}{2}}\nabla, \nabla^2)\CT^b_2(\lambda)\bg\|_{H^1_q(\HS)}
& \leq C|\lambda|^{-1}\|\bg\|_{H^1_q(\HS)}. \label{est:6.1.3}
\end{align}
When $0 < s < 1/q$, applying real interpolation to \eqref{est:6.1.2} and \eqref{est:6.1.3}
yields
\begin{equation}\label{est:6.1.4}
\|(\lambda, \lambda^{\frac{1}{2}}\nabla, \nabla^2)\CT^b_2(\lambda)\bg\|_{B^s_{q,r}(\HS)}
\leq C|\lambda|^{-1}\|\bg\|_{B^s_{q,r}(\HS)}.
\end{equation}\par
The dual operator $(\lambda, \lambda^{1/2}\nabla, \nabla^2)\CT^b_2(\lambda)^*$ of 
$(\lambda, \lambda^{1/2}\nabla, \nabla^2)\CT^b_2(\lambda)$ is obtained by exchanging
$\CF^{-1}_{\xi'}$ and $\CF'$ in \eqref{repr:6.2} and  \eqref{repr:6.3.2}.  Thus, 
employing the same argument as in  \eqref{est:6.1.2} and \eqref{est:6.1.3}, we have
\begin{align*}
\|((\lambda, \lambda^{\frac{1}{2}}\nabla, \nabla^2)\CT^b_2(\lambda))^*\bh\|_{L_{q'}(\HS)}
& \leq C|\lambda|^{-1}\|\bh\|_{L_{q'}(\HS)}, \\
\|((\lambda, \lambda^{\frac{1}{2}}\nabla, \nabla^2)\CT^b_2(\lambda))^*\bh\|_{H^1_{q'}(\HS)}
& \leq C|\lambda|^{-1}\|\bh\|_{H^1_{q'}(\HS)}
\end{align*}
for any $\lambda \in \Lambda_{\epsilon, \lambda_0}$ and $\bh \in C^\infty_0(\HS)^N$. 
Thus, by duality argument, we have
\begin{align}
\|(\lambda, \lambda^{\frac{1}{2}}\nabla, \nabla^2)\CT^b_2(\lambda)\bg\|_{L_q(\HS)}
& \leq C|\lambda|^{-1}\|\bg\|_{L_q(\HS)}, \label{est:6.1.5} \\
\|(\lambda, \lambda^{\frac{1}{2}}\nabla, \nabla^2)\CT^b_2(\lambda)\bg\|_{H^{-1}_{q}(\HS)}
& \leq C|\lambda|^{-1}\|\bg\|_{H^{-1}_{q}(\HS)}. \label{est:6.1.6}
\end{align}
Applying real interpolation \eqref{est:6.1.5} and \eqref{est:6.1.6} yields 
\begin{equation}\label{est:6.1.7}
\|(\lambda, \lambda^{\frac{1}{2}}\nabla, \nabla^2)\CT^b_2(\lambda)\bg\|_{B^s_{q,r}(\HS)}
\leq C|\lambda|^{-1}\|\bg\|_{B^s_{q,r}(\HS)},
\end{equation}
provided that $-1+1/q < s < 0$. Finally, interpolating \eqref{est:6.1.4} and \eqref{est:6.1.7} 
yields
\begin{equation}\label{est:6.1.8}
\|(\lambda, \lambda^{\frac{1}{2}}\nabla, \nabla^2)\CT^b_2(\lambda)\bg\|_{B^0_{q,r}(\HS)}
\leq C|\lambda|^{-1}\|\bg\|_{B^0_{q,r}(\HS)}.
\end{equation}
Thus, we have obtained \eqref{thm:6.1.5}.\par
Now, we consider $\pd_\lambda \bT^b_2(\lambda)$, which is represented by
\begin{align*}
\pd_\lambda\CT^b_2(\lambda)\bg &= \int^\infty_0 \CF^{-1}_{\xi'}\Bigl[ \Bigl(\sum_{j=1}^4
B^2(\pd_\lambda \CP_j(x_N, y_N) )
B^{-2}\bT^{b,0}_{2j}(\lambda, \xi')\Bigr)\CF'[\bg](\xi', y_N)\Bigr](x')\, dy_N \\
& + 
\int^\infty_0 \CF^{-1}_{\xi'}\Bigl[ \Bigl(\sum_{j=1}^4 \CP_j(x_N, y_N) 
(\pd_\lambda\bT^{b, 0}_{2j}(\lambda, \xi'))\Bigr)\CF'[\bg](\xi', y_N)\Bigr](x')\, dy_N
\end{align*}
as follows from \eqref{repr:6.2}. Moreover, from \eqref{repr:6.3} and \eqref{repdiff.1}, 
we have 
\begin{align}
\pd_N^\ell\pd_\lambda  \CT^b_2(\lambda)\bg& =  (-1)^\ell
\int^\infty_0 \CF^{-1}_{\xi'}\Bigl[ \Bigl(\sum_{j=1}^4B^2(\pd_\lambda \CP_j(x_N, y_N) )
B^{-2}\bT^{b,\ell}_{2j}(\lambda, \xi')\Bigr)\CF'[\bg](\xi', y_N)\Bigr](x')\, dy_N \nonumber \\
&+ (-1)^\ell
\int^\infty_0 \CF^{-1}_{\xi'}\Bigl[ \Bigl(\sum_{j=1}^4 \CP_j(x_N, y_N) )
(\pd_\lambda \bT^{b,\ell}_{2j}(\lambda, \xi'))\Bigr)\CF'[\bg](\xi', y_N)\Bigr](x')\, dy_N;
\nonumber \\
\pd_{N}^\ell\pd_\lambda\CT^b_2(\lambda)\bg
&= \int^\infty_0 \CF^{-1}_{\xi'}\Bigl[\Bigl(\sum_{j=1}^4 B^2(\pd_\lambda\CP_j(x_N, y_N))
B^{-2}\tilde\bT^{b, \ell}_{2j}(\lambda, \xi')\Bigr)
\CF'[\pd_{N}\bg](\xi', y_N)\Bigr](x')\,dy_N \nonumber \\
&+  \int^\infty_0 \CF^{-1}_{\xi'}\Bigl[\Bigl(\sum_{j=1}^4 
\CP_j(x_N, y_N)(\pd_\lambda \tilde\bT^{b, \ell}_{2j}(\lambda, \xi'))\Bigr)
\CF'[\pd_{N}\bg](\xi', y_N)\Bigr](x')\,dy_N. \label{repr:6.2.1}
\end{align}
If we write
\begin{align}
\lambda \pd_N^\ell\pd_\lambda  \CT^b_2(\lambda)\bg& =  (-1)^\ell
\int^\infty_0 \CF^{-1}_{\xi'}\Bigl[ \Bigl(\sum_{j=1}^4B^2(\pd_\lambda \CP_j(x_N, y_N) )
\lambda B^{-2}\bT^{b,\ell}_{2j}(\lambda, \xi')\Bigr)\CF'[\bg](\xi', y_N)\Bigr](x')\, dy_N 
\nonumber \\
&+ (-1)^\ell
\int^\infty_0 \CF^{-1}_{\xi'}\Bigl[ \Bigl(\sum_{j=1}^4 \CP_j(x_N, y_N) )
\lambda (\pd_\lambda \bT^{b,\ell}_{2j}(\lambda, \xi'))\Bigr)\CF'[\bg](\xi', y_N)\Bigr](x')\, dy_N;
\nonumber \\
\lambda \pd_{N}^\ell\pd_\lambda \CT^b_2(\lambda)\bg
&= \int^\infty_0 \CF^{-1}_{\xi'}\Bigl[\Bigl(\sum_{j=1}^4 B^2(\pd_\lambda\CP_j(x_N, y_N))
\lambda B^{-2}\tilde\bT^{b, \ell}_{2j}(\lambda, \xi')\Bigr)
\CF'[\pd_{N}\bg](\xi', y_N)\Bigr](x')\,dy_N  \nonumber \\
&+  \int^\infty_0 \CF^{-1}_{\xi'}\Bigl[\Bigl(\sum_{j=1}^4 
\CP_j(x_N, y_N)\lambda (\pd_\lambda \tilde\bT^{b, \ell}_{2j}(\lambda, \xi'))\Bigr)
\CF'[\pd_{N}\bg](\xi', y_N)\Bigr](x')\,dy_N \label{repr:6.2.2}
\end{align}
for $\ell=0,1,2, 3$, then using the facts that 
$\lambda B^{-2}\bT^{b,\ell}_{2j}(\lambda, \xi') \in \BN^d_{-2+\ell}$,  
$\lambda (\pd_\lambda \bT^{b, \ell}_{2j}(\lambda, \xi'))\in \BN^d_{-2+\ell}$,
$\lambda B^{-2}\tilde \bT^{b,\ell}_{2j}(\lambda, \xi') \in \BN^d_{-3+\ell}$,
and $\lambda (\pd_\lambda\tilde\bT^{b,\ell}_{2j}(\lambda, \xi'))
 \in \BN^d_{-3+\ell}$
for $\ell=0,1,2,3$ and employing the same argument as in the proof of 
\eqref{est:6.1.2*} and \eqref{est:6.1.3*},
by  Propositions \ref{prop:2} and \ref{prop:3}, 
we have
\begin{align}\label{est:6.5.1}
\|(\lambda, \lambda^{\frac{1}{2}}\nabla, \nabla^2)\pd_\lambda\CT^b_2(\lambda)\bg\|_{L_q(\HS)}
\leq C|\lambda|^{-2}\|\bg\|_{L_q(\HS)}, \\
\|(\lambda, \lambda^{\frac{1}{2}}\nabla, \nabla^2)\pd_\lambda\CT^b_2(\lambda)\bg\|_{H^1_q(\HS)}
\leq C|\lambda|^{-2}\|\bg\|_{H^1_q(\HS)}. \label{est:6.5.2}
\end{align}
If $0 < s < 1/q$, interpolating \eqref{est:6.5.1} and \eqref{est:6.5.2} yields
\begin{equation}\label{est:6.5.3}
\|(\lambda, \lambda^{\frac{1}{2}}\nabla, \nabla^2)\pd_\lambda\CT^b_2(\lambda)\bg\|_{B^s_{q,r}(\HS)}
\leq C|\lambda|^{-2}\|\bg\|_{B^s_{q,r}(\HS)}.
\end{equation}
To consider the case where $-1+1/q < s <0$, we consider the dual operator
$(\lambda, \lambda^{1/2}\nabla, \nabla^2)\pd_\lambda\CT^b_2(\lambda)^*$ of 
$(\lambda, \lambda^{1/2}\nabla, \nabla^2)\pd_\lambda\CT^b_2(\lambda)$, which is 
obtained by exchanging  $\CF^{-1}_{\xi'}$ and $\CF'$ in \eqref{repr:6.2.1}.
Then, from \eqref{repr:6.2.2} we have 
\begin{align}
\lambda (\pd_N^\ell\pd_\lambda  \CT^b_2(\lambda))^*\bh& =  (-1)^\ell 
\int^\infty_0 \CF'\Bigl[ \Bigl(\sum_{j=1}^4B^2(\pd_\lambda \CP_j(x_N, y_N) )
\lambda B^{-2}\bT^{b,\ell}_{2j}(\lambda, \xi')\Bigr)\CF^{-1}_{\xi'}[\bh](\xi', y_N)\Bigr](x')\, dy_N 
\nonumber \\ 
&+ (-1)^\ell
\int^\infty_0 \CF'\Bigl[ \Bigl(\sum_{j=1}^4 \CP_j(x_N, y_N) )
\lambda (\pd_\lambda \bT^{b,\ell}_{2j}(\lambda, \xi'))\Bigr)\CF^{-1}_{\xi'}[\bh](\xi', y_N)\Bigr](x')\, dy_N;
\nonumber \\
\lambda (\pd_{N}^\ell\pd_\lambda \CT^b_2(\lambda))^*\bh
&= \int^\infty_0 \CF'\Bigl[\Bigl(\sum_{j=1}^4 B^2(\pd_\lambda\CP_j(x_N, y_N))
\lambda B^{-2}\tilde\bT^{b, \ell}_{2j}(\lambda, \xi')\Bigr)
\CF^{-1}_{\xi'}[\pd_{N}\bg](\xi', y_N)\Bigr](x')\,dy_N  \nonumber \\
&+  \int^\infty_0 \CF'\Bigl[\Bigl(\sum_{j=1}^4 
\CP_j(x_N, y_N)\lambda (\pd_\lambda \tilde\bT^{b, \ell}_{2j}(\lambda, \xi'))\Bigr)
\CF^{-1}_{\xi'}[\pd_{N}\bg](\xi', y_N)\Bigr](x')\,dy_N. \label{repr:6.2.3}
\end{align}
Since
\begin{gather*}
\lambda B^{-2}\bT^{b,\ell}_{2j}(\lambda, \xi') \in \BN^d_{-2+\ell}, \quad 
\lambda (\pd_\lambda \bT^{b, \ell}_{2j}(\lambda, \xi'))\in \BN^d_{-2+\ell}, \\
\lambda B^{-2}\tilde \bT^{b,\ell}_{2j}(\lambda, \xi') \in \BN^d_{-3+\ell}, \quad 
\lambda (\pd_\lambda\tilde\bT^{b,\ell}_{2j}(\lambda, \xi'))
 \in \BN^d_{-3+\ell}
\end{gather*}
for $\ell=0,1,2,3$,  employing the same argument as in the proof of 
\eqref{est:6.1.2*} and \eqref{est:6.1.3*},
by  Propositions \ref{prop:2} and \ref{prop:3}, 
we have
\begin{align*}
\|((\lambda, \lambda^{\frac{1}{2}}\nabla, \nabla^2)\pd_\lambda\CT^b_2(\lambda))^*\bh\|_{L_{q'}(\HS)}
\leq C|\lambda|^{-2}\|\bh\|_{L_{q'}(\HS)}, \\
\|((\lambda, \lambda^{\frac{1}{2}}\nabla, \nabla^2)\pd_\lambda\CT^b_2(\lambda))^*\bh\|_{H^1_q(\HS)}
\leq C|\lambda|^{-2}\|\bh\|_{H^1_{q'}(\HS)}.
\end{align*}
By duality argument, we have 
\begin{align*}
\|(\lambda, \lambda^{\frac{1}{2}}\nabla, \nabla^2)\pd_\lambda\CT^b_2(\lambda)\bg\|_{L_q(\HS)}
&\leq C|\lambda|^{-2}\|\bg\|_{L_q(\HS)}, \\
\|(\lambda, \lambda^{\frac{1}{2}}\nabla, \nabla^2)\pd_\lambda\CT^b_2(\lambda)\bg\|_{H^{-1}_q(\HS)}
&\leq C|\lambda|^{-2}\|\bg\|_{H^{-1}_q(\HS)}.
\end{align*}
Thus, by real interpolation, we have
\begin{equation}\label{est:6.5.4}
\||(\lambda, \lambda^{\frac{1}{2}}\nabla, \nabla^2)\pd_\lambda\CT^b_2(\lambda)\bg\|_{B^s_{q,r}(\HS)}
\leq C|\lambda|^{-2}\|\bg\|_{B^s_{q,r}(\HS)}
\end{equation}
provided that $-1+1/q < s < 0$.  Combining \eqref{est:6.5.3} and \eqref{est:6.5.4} 
yields 
\begin{equation}\label{est:6.5.5}
\|(\lambda, \lambda^{\frac{1}{2}}\nabla, \nabla^2)\pd_\lambda\CT^b_2(\lambda)\bg\|_{B^0_{q,r}(\HS)}
\leq C|\lambda|^{-2}\|\bg\|_{B^0_{q,r}(\HS)}
\end{equation}
Therefore, from \eqref{est:6.1.4}, \eqref{est:6.1.7}, \eqref{est:6.1.8}, \eqref{est:6.5.3},
\eqref{est:6.5.4}, and \eqref{est:6.5.5}, we have obtained \eqref{thm:6.1.5} and \eqref{thm:6.1.6}.
This completes the proof of Theorem \ref{thm:6.1}.
\section{Proof of Main Results} \label{sec:7}

In this section, first of all we construct solution operators of equations: 
\begin{equation}\label{eq:4}\left\{ \begin{aligned}
\lambda\rho + \gamma \dv\bu = f&&\quad&\text{in $\HS$}, \\
\lambda\bu - \alpha\Delta \bu - \beta\nabla \dv\bu + \gamma \nabla\rho
= \bg&&\quad&\text{in $\HS$}, \\
\bu=0&&\quad&\text{on $\pd\HS$}.
\end{aligned}\right.\end{equation}
First, from the first equation in \eqref{eq:4}, we set $\rho = \lambda^{-1}(f- \gamma\dv\bu)$,
and inserting this formula into the second equation in \eqref{eq:4}, we have the complex
Lam\'e equation:
\begin{equation}\label{eq:5} 
\lambda\bu - \alpha\Delta \bu - \eta_\lambda \nabla\dv\bu
= \bg- \gamma\lambda^{-1}\nabla f \quad\text{in $\HS$}, \quad
\bu|_{\pd\HS}=0.
\end{equation}
From Theorems \ref{thm:5.1} and  \ref{thm:6.1}, we have
$$\bu = \CS^0(\lambda)(\bg- \gamma\lambda^{-1}\nabla f) 
- \CS^b(\lambda)(\bg-\gamma\lambda^{-1}\nabla f).$$
Thus, defining $\rho$ by
$$\rho = \lambda^{-1}(f-\gamma\dv\bu) = 
\lambda^{-1}f -\gamma \lambda^{-1}\dv(\CS^0(\lambda)(\bg- \gamma\lambda^{-1}\nabla f) 
- \CS^b(\lambda)(\bg-\gamma\lambda^{-1}\nabla f)),$$
we see that $\bu$ and $\rho$ are solutions of equations \eqref{eq:4}. 
In view of Theorems \ref{thm:5.1} and \ref{thm:6.1}, we decompose $\bu$ as 
\begin{align*}
\bu = \CT^0_1(\lambda)\bg - \CT^b_1(\lambda)\bg 
+\CT^0_2(\lambda)\bg - \CT^b_2(\lambda)\bg -
\gamma\lambda^{-1}\CS^0(\lambda)\nabla f + \gamma\lambda^{-1}
\CS^b(\lambda)\nabla f.
\end{align*}
Summing up, there exist solution operators $\CS(\lambda)$, $\CS^1(\lambda)$, $\CS^2(\lambda)$ such
that $\bu=\CS(\lambda)(f, \bg)$, $\rho = \CR(\lambda)(f, \nabla\bg)$, and 
\begin{equation}\label{solop.7.1}\begin{aligned}
\CS^1(\lambda)\bg &= \CT^0_1(\lambda)\bg - \CT^b_1(\lambda)\bg, \\
\CS^2(\lambda)(f, \bg) & = \CT^0_2(\lambda)\bg -\CT^b_2(\lambda)\bg
-\gamma\lambda^{-1}\CS^0(\lambda)\nabla f 
+ \gamma\lambda^{-1}\CS^b(\lambda)\nabla f, \\
\CS(\lambda)(f, \bg) &= \CS^1(\lambda)\bg + \CS^2(\lambda)(f, \bg), \\
\CR(\lambda)(f, \bg) & =  
\lambda^{-1}f -\gamma \lambda^{-1}\dv\CS^0(\lambda)\bg
+ \gamma^2\lambda^{-2}\dv\CS^0(\lambda)\nabla f\\
&\quad - \gamma \lambda^{-1}\dv\CS^b(\lambda)\bg + \gamma^2\lambda^{-2}\dv\CS^b(\lambda)\nabla f.
\end{aligned}\end{equation}
We see easily that 
\begin{equation}\label{hol:7.1}\begin{aligned}
\CS(\lambda)& \in {\rm Hol}\, (\Lambda_{\epsilon, \lambda_0}, 
\CL(B^{s+1}_{q,r}(\HS)\times B^s_{q,r}(\HS), B^{s+2}_{q,r}(\HS))), \\
\CS^1(\lambda) &\in {\rm Hol}\, (\Lambda_{\epsilon, \lambda_0}, 
\CL(B^s_{q,r}(\HS), B^{s+2}_{q,r}(\HS))), \\
\CS^2(\lambda)& \in {\rm Hol}\, (\Lambda_{\epsilon, \lambda_0}, 
\CL(B^{s+1}_{q,r}(\HS)\times B^s_{q,r}(\HS), B^{s+2}_{q,r}(\HS))), \\
\CR(\lambda) & \in {\rm Hol}\, (\Lambda_{\epsilon, \lambda_0}, 
\CL(B^{s+1}_{q,r}(\HS)\times B^s_{q,r}(\HS), B^{s+1}_{q,r}(\HS))).
\end{aligned}\end{equation}
Moreover, by Theorems \ref{thm:5.1} and \ref{thm:6.1}, we see the following theorem.
\begin{thm}\label{thm:7.1} 
Let $1 < q <\infty$, $1 \leq r < \infty$, 
$-1+1/q < s < 1/q$, and $\epsilon \in (0, \pi/2)$.
Then, there exists a large number $\lambda_0 >0$ such that 
for any $\lambda \in \Lambda_{\epsilon, \lambda_0}$, $f \in 
B^{s+1}_{q,r}(\HS)$, and $\bg \in C^\infty_0(\HS)^N$,
 there hold
\begin{align*}
\|(\lambda, \lambda^{\frac{1}{2}}\nabla, \nabla^2)\CS(\lambda)(f, \bg)\|_{B^s_{q,r}(\HS)}
&\leq C\|(f, \bg)\|_{\CH^s_{q,r}(\HS)}, \\
\|(\lambda^{\frac{1}{2}}\nabla, \nabla^2)\CS^1(\lambda)\bg\|_{B^s_{q,r}(\HS)} 
&\leq C|\lambda|^{-\frac{\sigma}{2}}\|\bg\|_{B^{s+\sigma}_{q,r}(\HS)}, \\
\|(\lambda^{\frac{1}{2}}\nabla, \nabla^2)\pd_\lambda \CS^1(\lambda)\bg\|_{B^s_{q,r}(\HS)} 
&\leq C|\lambda|^{-(1-\frac{\sigma}{2})}\|\bg\|_{B^{s-\sigma}_{q,r}(\HS)}, \\
\|(\lambda^{\frac{1}{2}}\nabla, \nabla^2)\CS^2(\lambda)(f, \bg)\|_{ B^s_{q,r}(\HS)} &\leq C|\lambda|^{-1}\|(f, \bg)\|_{\CH^s_{q,r}(\HS)},\\
\|(\lambda^{\frac{1}{2}}\nabla, \nabla^2)\pd_\lambda\CS^2(\lambda)(f, \bg)\|_{ B^s_{q,r}(\HS)} &\leq C|\lambda|^{-2}\|(f, \bg)\|_{\CH^s_{q,r}(\HS)}, \\
\|\CR(\lambda)f\|_{B^{s+1}_{q,r}(\HS)} & \leq C|\lambda|^{-1}\|(f, \bg)
\|_{\CH^s_{q,r}(\HS)}, \\
\|\pd_\lambda \CR(\lambda)f\|_{B^{s+1}_{q,r}(\HS)} & \leq C|\lambda|^{-2}\|(f, \bg)
\|_{\CH^s_{q,r}(\HS)}.
\end{align*}
\end{thm}
Theorem \ref{thm:1} follows from Theorem \ref{thm:7.1} immediately. 
\par
Now, we consider an initial value problem:
\begin{equation}\label{eq:5}\left\{ \begin{aligned}
\pd_t \Pi + \gamma \dv\bU = 0&&\quad&\text{in $\HS\times \BR_+$}, \\
\pd_t\bU - \alpha\Delta \bU - \beta\nabla \dv\bU + \gamma \nabla \Pi
=0 &&\quad&\text{in $\HS\times \BR_+$}, \\
\bU=0&&\quad&\text{on $\pd\HS\times\BR_+$}, \\
(\Pi, \bU)|_{t=0} = (\Pi_0, \bU_0)&&\quad&\text{in $\HS$}.
\end{aligned}\right.\end{equation}
To formulate problem \eqref{eq:5} in the semigroup setting, we introduce spaces $\CH^s_{q,r}(\HS)$,
$\CD^s_{q,r}(\HS)$ and an operator $\CA_{q,r}$ defined in \eqref{fundspace:1} and  \eqref{Def:A}, 
respectively.  Then, as was seen in \eqref{semigroup:1}, equations \eqref{eq:5} are written as 
$$\pd_t(\Pi,  \bU) + \CA^s_{q,r}(\Pi, \bU) = (0, 0) \quad\text{for $t > 0$}, \quad
(\Pi, \bU)|_{t=0}=(\Pi_0, \bU_0) \in \CH^s_{q,r}.$$
And, the corresponding resolvent problem \eqref{eq:4} is written as
$$\lambda(\rho, \bu) + \CA^s_{q,r}(\rho, \bu) = (f, \bg)$$
for $(f, \bg) \in \CH^s_{q,r}(\HS)$ and $(\rho, \bu) \in \CD^s_{q,r}(\HS)$. 
From Theorem \ref{thm:7.1} it follows that the resolvent operator $(\lambda+\CA^s_{q,r})^{-1}$ 
exists for any $\lambda \in \Lambda_{\epsilon, \lambda_0}$ for sufficient large $\lambda_0>0$.
In fact, $(\lambda+\CA^s_{q,r})^{-1}(f, \bg) = (\CR(\lambda),\CS(\lambda))(f, \bg)$
for $(f, \bg) \in \CH^s_{q,r}(\HS)$. Thus, 
 the resolvent estimate: $\|\lambda(\lambda+\CA^s_{q,r})^{-1}\|_{\CL(\CH^s_{q,r})} \leq C$ holds
for any $\lambda \in \Lambda_{\epsilon , \lambda_0}$. \par
From these observations, 
 by theory of $C_0$ analytic semigroup (\cite{Yosida}), there exists a $C_0$ analytic
semigroup $\{T(t)\}_{t\geq0}$ associated with \eqref{eq:5} and $(\Pi, \bU) = T(t)(\Pi_0, \bU_0)$
is a unique solution of \eqref{eq:5}, which satisfies the regularity condition:
$$(\Pi, \bU) \in C^0([0, \infty), \CH^s_{q,r}(\HS))
\cap C^0((0, \infty), \CD^s_{q,r}(\HS)) \cap C^1((0, \infty), \CH^s_{q,r}(\HS))$$
as well as 
$$\lim_{t\to0}\|(\Pi(\cdot, t), \bU(\cdot, t))-(\Pi_0, \bU_0)\|_{\CH^s_{q,r}(\HS)} = (0, 0).
$$
Finally, we shall show the following theorem about the  maximal $L_1$ regularity of $\{T(t)\}_{t\geq0}$.
Obviously, combining the results about continuous analytic semigroup theory  mentioned above 
and the following theorem completes  the proof of Theorem \ref{thm:0}.
\begin{thm}\label{thm:7.2} Let $1 < q < \infty$ and $-1+1/q < s < 1/q$.  Then,
there exists $\omega>0$ such that for any $(\Pi_0, \bU_0) \in 
\CH^s_{q,1}(\HS)$, there holds
$$\int^\infty_0 e^{-\omega t}(\|\pd_tT(t)(\Pi_0, \bU_0)\|_{\CH^s_{q,1}(\HS)}
+ \|T(t)(\Pi_0, \bU_0)\|_{\CD^s_{q,1}(\HS)})\,dt \leq C\|(\Pi_0, \bU_0)\|_{\CH^s_{q,1}(\HS)}.
$$
\end{thm}
In the sequel, we shall prove Theorem \ref{thm:7.2}. We start with the following lemma.
\begin{prop}
\label{prop-real-interpolation}
Let $X_0$ and $X_1$ be Banach spaces which are an  interpolation couple,
and $Y$ be another Banach space. 
Assume that $0 < \sigma_0, \sigma_1, \theta < 1$ satisfy
$1 = (1 - \theta)(1- \sigma_0) + \theta(1+ \sigma_1)$. Let $\omega \ge 0$.
For $t > 0$ let $T (t) \colon Y \to  X_0+X_1$ 
be a bounded linear operator such that
\begin{equation}
\label{assumption-decay}
\begin{split}
\lVert T (t) f \rVert_Y 
& \le C e^{\omega t} t^{- 1 + \sigma_0} \lVert f \rVert_{X_0}, 
\qquad f \in X_0, \\
\lVert T (t) f \rVert_Y
& \le C e^{\omega t} t^{- 1 - \sigma_1} \lVert f \rVert_{X_1}, 
\qquad f \in X_1.
\end{split}
\end{equation}
Then, there holds
\begin{equation}
\int^\infty_0 e^{-\omega t}\lVert T (t) f \rVert_Y\,d t
\le C \lVert f \rVert_{(X_0, X_1)_{\theta, 1}}
\end{equation}
with a constant $C > 0$ independent of $\omega$.
\end{prop}
\begin{proof}
The proof is based on real interpolation. For $k \in \BZ$ set
\begin{equation}
b_k (f) := \sup_{t \in [2^k, 2^{k + 1}]} e^{- \omega t} 
\lVert T (t) f \rVert_Y.
\end{equation}
We observe that
\begin{equation}
\label{est-B2}
\int^\infty_0 e^{-\omega t}\lVert T (t) f \rVert_Y\,d t
= \sum_{k \in \BZ} \int_{2^k}^{2^{k + 1}} e^{- \omega t} 
\lVert T (t) f \rVert_Yd t 
\le \sum_{k \in \BZ} 2^k b_k (f).
\end{equation}
Then we infer from the assumptions \eqref{assumption-decay} that
\begin{alignat}2
b_k (f) & \le C \sup_{t \in [2^k, 2^{k + 1}]} t^{- 1 + \sigma_0} 
\lVert f \rVert_{X_0} \le C 2^{- k (1 - \sigma_0)} \lVert f \rVert_{X_0}, 
& \qquad & f \in X_0, \\
b_k (f) & \le C \sup_{t \in [2^k, 2^{k + 1}]} t^{- 1 - \sigma_1} 
\lVert f \rVert_{X_1} \le C 2^{- k (1 + \sigma_1)} \lVert f \rVert_{X_1}, 
& \qquad & f \in X_1.
\end{alignat}
Namely, there hold
\begin{alignat}2
\lVert (b_k)_{k \in \BZ} \rVert_{\ell^{1 - \sigma_0}_\infty (\BZ)}
& \le C \lVert f \rVert_{X_0},
& \qquad & f \in X_0, \\
\lVert (b_k)_{k \in \BZ} \rVert_{\ell^{1 + \sigma_1}_\infty (\BZ)}
& \le C \lVert f \rVert_{X_1},
& \qquad & f \in X_1.
\end{alignat}
Since $(\ell^{1 - \sigma_0}_\infty (\BZ), \ell^{1 + \sigma_1}_\infty (\BZ))_{\theta, 1}
= \ell^1_1 (\BZ)$ due to \cite[Thm. 5.6.1]{BLbook}, it follows that
\begin{equation}
\label{est-B3}
\sum_{k \in \BZ} 2^k b_k (f) = \lVert b_k (f)_{k \in \BZ} \rVert_{\ell^1_1 (\BZ)}
\le C \lVert f \rVert_{(X_0, X_1)_{\theta, 1}}.
\end{equation}
Thus, the desired estimates follows from \eqref{est-B2} and \eqref{est-B3}.
\end{proof}
\begin{proof}[A Proof of Theorem \ref{thm:7.2}.] Let $\omega>0$ be a large number such that
$\Sigma_\epsilon + \omega \subset \Lambda_{\epsilon, \lambda_0}$. 
Let $\Gamma$ be a contour in $\BC$ defined by $\Gamma=\Gamma_+ \cup \Gamma_-$
with 
$$\Gamma_\pm = \{\lambda = re^{\pm(\pi-\epsilon)} \mid r \in (0, \infty)\}.$$
As was well-known in theory of $C_0$ analytic semigroup (cf. \cite{Yosida}), we have
$$T(t)(\Pi_0, \bU_0) = \frac{1}{2\pi i}\int_{\Gamma+\omega}
(\CS(\lambda), \CR(\lambda))(\Pi_0, \bU_0)\,d\lambda \quad\text{for $t > 0$}.
$$
To show the $L_1$ integrability of $T(t)$, we use Theorem \ref{thm:7.1}. According to 
the formulas in \eqref{solop.7.1}, we divide $T(t)$ into the following three parts:
\begin{align}
T_1(t)\bU_0 &= \frac{1}{2\pi i}\int_{\Gamma+\omega} \CS^1(\lambda)\bU_0\,d\lambda, 
\label{t.3.6.2}\\
T_2(t)(\Pi_0, \bU_0) & =
\frac{1}{2\pi i}\int_{\Gamma+\omega} \CS^2(\lambda)(\Pi_0, \bU_0)\,d\lambda, 
\label{t.3.6.3}\\
T_3(t)(\Pi_0, \bU_0) & = 
\frac{1}{2\pi i}\int_{\Gamma+\omega} \CR(\lambda)(\Pi_0, \bU_0)\,d\lambda.
\label{t.3.6.4}
\end{align}
We have $T(t)(\Pi_0, \bU_0) = (T_3(t)(\Pi_0, \bU_0), T_1(t)\bU_0 + T_2(t)(\Pi_0, \bU_0)).$
\par

We first show that 
\begin{equation}\label{L1T_1}
\int^\infty_0 e^{-\omega t}\|T_1(t)\bU_0\|_{B^{s+2}_{q,1}(\HS)}\,d t \leq C\|\bU_0\|_{B^s_{q,1}(\HS)}.
\end{equation}
To this end, 
in view of Proposition \ref{prop-real-interpolation}, we first prove that  
for every $t > 0$ and $\bU_0 \in C^\infty_0(\HS)^N$ there hold
\begin{align}
\label{L1est.0-1}
\|T_1(t)\bU_0\|_{B^{s+2}_{q,1}(\HS)}
\le Ce^{\omega t}t^{-1+\frac{\sigma}{2}}\|\bU_0\|_{B^{s+\sigma}_{q,1}(\HS)}, \\
\label{L1est.0-2}
\|T_1(t)\bU_0\|_{B^{s+2}_{q,1}(\HS)}
\le Ce^{\omega t}t^{-1-\frac{\sigma}{2}}\|\bU_0\|_{B^{s-\sigma}_{q,1}(\HS)}.
\end{align}
Notice that $\lambda = \omega + re^{\pm i(\pi-\epsilon)}$ for $\lambda
\in \Gamma_\pm + \omega$, and thus $|e^{\lambda t}| = e^{\omega t}
e^{\cos(\pi-\epsilon)rt} =e^{\omega t} e^{-rt\cos\epsilon}$
for $\lambda \in \Gamma_\pm + \omega$. Since 
$\| \CS^1(\lambda)\bU\|_{B^{s+2}_{q,1}(\HS)} \le C|\lambda|^{-\frac{\sigma}{2}}
\|\bU\|_{B^{s+\sigma}_{q,1}(\HS)}$ as follows from  Theorem \ref{thm:7.1}, 
using \eqref{t.3.6.2}, for $t>0$ we have
\begin{align}
\| T_1(t)\bU_0\|_{B^{s+2}_{q,1}(\HS)}
&\le Ce^{\omega t}\int^\infty_0 e^{-rt\cos\epsilon }r^{-\frac{\sigma}{2}}d r 
\|\bU_0\|_{B^{s+\sigma}_{q, 1}(\HS)}\\
&= Ce^{\omega t}t^{-1+\frac\sigma2}\int^\infty_0e^{-\ell \cos\epsilon}\ell^{-\frac{\sigma}{2}}d\ell\,
\|\bU_0\|_{B^{s+\sigma}_{q, 1}(\HS)},
\end{align}
which yields \eqref{L1est.0-1}. To prove \eqref{L1est.0-2}, 
by integration by parts we write 
\begin{equation}
T_1(t)\bU_0 = -\frac{1}{2\pi i t}\int_{\Gamma + \omega} e^{\lambda t}\pd_\lambda  \CS^1(\lambda)
\bU_0 d\lambda.
\end{equation}
Since $\|\pd_\lambda\CS^1(\lambda)\bU_0\|_{B^{s+2}_{q,1}(\HS)} 
\le C|\lambda|^{-(1-\frac{\sigma}{2})}
\|\bU_0\|_{B^{s-\sigma}_{q,1}(\HS)}$ as follows from  Theorem \ref{thm:7.1}, we have
\begin{align}
\|T_1(t)\bU_0\|_{B^{s+2}_{q,1}(\HS)}
&\le Ct^{-1}e^{\omega t}\int^\infty_0 e^{-rt\cos\epsilon }r^{-(1-\frac\sigma2)} d r \;
\|\bU_0\|_{B^{s-\sigma}_{q, 1}(\HS)}\\
&= Ce^{\omega t}t^{-1-\frac\sigma2}\int^\infty_0e^{-\ell \cos\epsilon}  
\ell^{-1+\frac\sigma2}d\ell \;
\|\bU_0\|_{B^{s-\sigma}_{q, 1}(\HS)},
\end{align}
which yields \eqref{L1est.0-2}.  
Choosing $\theta=1/2$ in Proposition \ref{prop-real-interpolation} 
and using the fact that 
$$(B^{s+\sigma}_{q, 1}(\HS), B^{s-\sigma}_{q,1}(\HS))_{1/2,1}
= B^s_{q,1}(\HS),$$ by Proposition \ref{prop-real-interpolation}, we have \eqref{L1T_1}
for $\bU_0 \in C^\infty_0(\HS)^N$.  But, since $C^\infty_0(\HS)^N$ is dense in 
$B^s_{q,r}(\HS)^N$, the estimate \eqref{L1T_1} holds for any $\bU_0 \in 
B^s_{q,1}(\HS)^N$. 
\par

We now show that 
\begin{equation}\label{L1T.23}\begin{aligned}
\int^\infty_0 e^{-\omega t}\|T_2(t)(\Pi_0, \bU_0)\|_{B^{s+2}_{q,1}}\,d t
\leq C\|(\Pi_0, \bU_0)\|_{\CH^s_{q,1}(\HS)}, \\
\int^\infty_0 e^{-\omega t}\|T_3(t)(\Pi_0, \bU_0)\|_{B^{s+1}_{q,1}}\,d t
\leq C\|(\Pi_0, \bU_0)\|_{\CH^s_{q,1}(\HS)}.
\end{aligned}\end{equation}
In fact, using Theorem \ref{thm:7.1} and $|\lambda| \geq \lambda_0$, we have
\begin{align*}
\|(\lambda^{1/2}\nabla, \nabla^2)\CS^2(\lambda)(f, \bg)\|_{B^{s}_{q,r}(\HS)} 
&\leq C|\lambda|^{-1}\|(f, \bg)\|_{\CH^s_{q,r}(\HS)}\\
&\leq C\lambda_0^{-(1-\frac{\sigma}{2})}|\lambda|^{-\frac{\sigma}{2}}
\|(f, \bg)\|_{\CH^s_{q,r}(\HS)}, \\
\|(\lambda^{1/2}\nabla, \nabla^2)\pd_\lambda\CS^2(\lambda)(f, \bg)\|_{B^s_{q,r}(\HS)} 
&\leq C|\lambda|^{-2}\|(f, \bg)\|_{\CH^s_{q,r}(\HS)} \\
&\leq C\lambda_0^{-(1+\frac{\sigma}{2})}|\lambda|^{-(1-\frac{\sigma}{2})}
\|(f, \bg)\|_{\CH^s_{q,r}(\HS)}, \\
\|\CR(\lambda)f\|_{B^{s+1}_{q,r}(\HS)} & \leq C|\lambda|^{-1}\|(f, \bg)
\|_{\CH^s_{q,r}(\HS)} \\
&\leq C\lambda_0^{-(1-\frac{\sigma}{2})}
|\lambda|^{-\frac{\sigma}{2}}\|(f, \bg)\|_{\CH^s_{q,r}(\HS)},\\
\|\pd_\lambda \CR(\lambda)f\|_{B^{s+1}_{q,r}(\HS)} & \leq C|\lambda|^{-2}\|(f, \bg)
\|_{\CH^s_{q,r}(\HS)} \\
& \leq C\lambda_0^{-(1+\frac{\sigma}{2})}|\lambda|^{-(1-\frac{\sigma}{2})}\|(f, \bg)
\|_{\CH^s_{q,r}(\HS)} 
\end{align*}
for any $\lambda \in \Sigma_\epsilon + \omega$ and $(f, \bg) \in \CH^s_{q,1}$.
In view of \eqref{t.3.6.3} and \eqref{t.3.6.4}, employing the same argument as in
the proof of \eqref{L1est.0-1} and \eqref{L1est.0-2}, we have
\begin{align*}
\|T_2(t)(\Pi_0, \bU_0)\|_{B^{s+2}_{q,1}(\HS)}
&\le C\lambda_0^{-\frac{\sigma}{2}}e^{\omega t}t^{-1+\frac{\sigma}{2}}
\|(\Pi_0, \bU_0)\|_{\CH^s_{q,1}}, \\
\|T_2(t)(\Pi_0, \bU_0)\|_{B^{s+2}_{q,1}(\HS)}
&\le C\lambda_0^{-(1+\frac{\sigma}{2})}
e^{\omega t}t^{-1-\frac{\sigma}{2}}\|(\Pi_0, \bU_0)\|_{\CH^s_{q,1}}, \\
\|T_3(t)(\Pi_0, \bU_0)\|_{B^{s+1}_{q,1}(\HS)}
&\le C\lambda_0^{-\frac{\sigma}{2}}e^{\omega t}t^{-1+\frac{\sigma}{2}}
\|(\Pi_0, \bU_0)\|_{\CH^s_{q,1}}, \\
\|T_3(t)(\Pi_0, \bU_0)\|_{B^{s+1}_{q,1}(\HS)}
&\le C\lambda_0^{-(1+\frac{\sigma}{2})}
e^{\omega t}t^{-1-\frac{\sigma}{2}}\|(\Pi_0, \bU_0)\|_{\CH^s_{q,1}}.
\end{align*}
Thus, using Proposition \ref{prop-real-interpolation} and noting that
$(\CH^s_{q,1}, \CH^s_{q,1})_{1/2,1} = \CH^s_{q,1} = B^{s+1}_{q,1}(\HS)\times B^s_{q,1}(\HS)^N$, 
we have \eqref{L1T.23}.  This completes the proof of Theorem \ref{thm:7.2}.
\end{proof}
\section{Acknowledgments}
First and foremost, I would like to express my sincere gratitude to my research supervisor, Professor Yoshihiro Shibata. He has given me kindly guidance and valuable advice all along.
I would also like to thank Dr. Keiichi Watanabe and Dr. Kenta Oishi, who guided me with solving difficult estimates and writing this paper. 

\end{document}